\RequirePackage{fix-cm}
\documentclass[smallcondensed]{svjour3}
\smartqed

\usepackage{booktabs}
\usepackage{amssymb}
\usepackage{float}
\usepackage{xcolor}
\usepackage[most]{tcolorbox}
\usepackage{cite}
\usepackage[title,titletoc]{appendix}
\usepackage[ruled,vlined]{algorithm2e}
\UseRawInputEncoding
\usepackage{graphics,graphicx,epstopdf}
\usepackage{physics}
\usepackage{hyperref}
\usepackage{amsmath}
\usepackage{bm}
\usepackage{bbm}
\usepackage{subcaption}
\usepackage[export]{adjustbox}
\usepackage{array}
\usepackage{verbatim}
\allowdisplaybreaks
 \newcommand{\BR}{\mathsf{BR}}
\definecolor{blue}{rgb}{0,0,0.9}
\definecolor{red}{rgb}{0.9,0,0}
\definecolor{green}{rgb}{0,0.6,0}

\newcommand{\blue}[1]{\textcolor{blue}{#1}}

\SetKwInput{KwInput}{Input}
\SetKwInput{KwConstants}{Constants}
\SetKwInput{KwOutput}{Output}
\SetKwFor{AgentFor}{each agent}{in parallel:}{}
\SetKwFor{SnodeFor}{each supernode}{in parallel:}{}
\SetKwFor{SuperEdgeFor}{each superedge}{in parallel:}{}
\SetKwFor{EdgeFor}{each edge}{in parallel:}{}
\SetKwFor{SubgraphFor}{each subgraph}{in parallel:}{}
\SetInd{0.25em}{0.45em}
\newcommand{\ben}{\begin{enumerate}}
\newcommand{\een}{\end{enumerate}}
\newcommand{\bfr}{\begin{frame}}
\newcommand{\efr}{\end{frame}}
\newcommand{\bit}{\begin{itemize}}
\newcommand{\eit}{\end{itemize}}
\newcommand{\argmin}{\mathop{\rm argmin}}

\newcommand{\R}{\mathbb R}

\newcommand{\cA}{\mathcal A}
\newcommand{\cB}{\mathcal B}
\newcommand{\cC}{\mathcal C}
\newcommand{\cD}{\mathcal D}
\newcommand{\cE}{\mathcal E}
\newcommand{\cF}{\mathcal F}
\newcommand{\cG}{\mathcal G}
\newcommand{\cH}{\mathcal H}
\newcommand{\cI}{\mathcal I}
\newcommand{\cJ}{\mathcal J}
\newcommand{\cL}{\mathcal L}
\newcommand{\cM}{\mathcal M}

\newcommand{\cP}{\mathcal P}

\newcommand{\cR}{\mathcal R}
\newcommand{\cS}{\mathcal S}
\newcommand{\cT}{\mathcal T}
\newcommand{\cU}{\mathcal U}
\newcommand{\cV}{\mathcal V}

\def\inner#1{\left\langle#1\right\rangle}

\makeatletter
\@ifundefined{definition}{\newtheorem{definition}{Definition}[section]}{}
\@ifundefined{theorem}{\newtheorem{theorem}{Theorem}[section]}{}
\@ifundefined{corollary}{\newtheorem{corollary}{Corollary}[section]}{}
\@ifundefined{lemma}{\newtheorem{lemma}{Lemma}[section]}{}
\@ifundefined{proposition}{\newtheorem{proposition}{Proposition}[section]}{}
\@ifundefined{remark}{\newtheorem{remark}{Remark}[section]}{}
\@ifundefined{assumption}{\newtheorem{assumption}{Assumption}}{}
\@ifundefined{thm}{}{}
\@ifundefined{example}{\newtheorem{example}{Example}[section]}{}
\makeatother

\def\dist{{\rm dist}}

\def\Range{{\rm Range}}

\journalname{Journal Name}

\begin{document}

\title{ From Mixing to Tearing: Graph Decomposition in Decentralized Optimization via Message Passing}
\titlerunning{Decentralized Optimization via Message Passing}
\author{Kuangyu Ding \and Gesualdo Scutari \vspace{-3cm}}
\institute{K. Ding \at Edwardson School of Industrial Engineering, Purdue University, West Lafayette, IN. 
              \email{ding433@purdue.edu}. \and G. Scutari \at Edwardson School of Industrial Engineering \& Elmore Family School of Electrical and Computer Engineering, Purdue University, West Lafayette, IN. \email{gscutari@purdue.edu}. \smallskip \newline The work of Ding and Scutari has been supported by the ONR Grant \#  N00014-24-1-2751.       }
\date{}
\maketitle
 
 \begin{abstract}\vspace{-1.4cm }We study   decentralized minimization of sums of smooth strongly convex functions, one per  agent, with  communication restricted to   neighbors in a graph. Existing decentralized algorithms typically emulate a {\it centralized} method; {\it the network serves only to exchange information}, from which agents approximate   the global quantities that the method   needs. This paper breaks with this paradigm. We introduce a {\it graph decomposition framework} that uses
  the network to design the method, not merely execute it.
  It jointly determines how agreement is imposed, which variables
  are optimized jointly (blocks), and which connected subgraphs host
  the agents solving each block. The convergence analysis
  quantifies the trade-offs among these choices. 
  GATE (Graph-Tearing message passing) instantiates the framework  by assigning one dual variable per edge and partitioning the
  edges into trees. Agents interleave edge updates with local
  message passing recursions that summarize optimization within each
  tree (block). 
  Its surrogate variant, GATE-S, uses tractable local models and
  compact messages to reduce computation and communication.
  We establish linear convergence and show that   the topology-aware design acts as a \emph{preconditioner}:      the rate depends {\it jointly} on
  topology and function geometry through the coupling between
  subproblems defined on the selected subgraphs,    whereas consensus-based optimization methods pay the worst case of each separately, the centralized condition number times the inverse spectral gap. 
   On tree-dominant graphs, convergence is provably faster: the rate scales with the tree {\it diameter}, linearly in the number of agents on rings and barbells against quadratically for non-accelerated consensus-optimization methods. Experiments support the predicted scaling and savings in iterations and communication. 
   \vspace{-0.5cm}
\end{abstract}
\section{Introduction}
\label{sec:introduction} \vspace{-0.2cm}
The paper studies
  decentralized optimization over an undirected, connected, and simple   graph \(\cG=(\cV,\cE)\), where   $\cV$ is the set of nodes (hereafter, agents), with \(|\cV|=m\), and $\cE$ is the set of edges; neighboring agents can exchange information over single-hop links.    Each agent \(i\in\cV\) knows only its own  cost function \(f_i:\R^d\to\R\),  assumed strongly convex and smooth on  \(\R^d\).  The agents seek to   solve  cooperatively  the following optimization problem
\begin{equation}
    \min_{x\in\R^d} F(x):=\sum_{i\in\cV} f_i(x).
    \label{eq:clean_prob}\tag{P}\vspace{-0.1cm}
\end{equation}
This class of problems and setting   arise in  several applications of interest, including signal processing, machine learning, multi-agent systems, and communications. 

\textbf{(i) Inexact methods. }
 Early methods solved \eqref{eq:clean_prob} by decentralizing the gradient algorithm:   
 the Distributed Gradient Descent (DGD) was introduced in~\cite{Lopes:ja} for least squares and later generalized to nonquadratic functions  in~\cite{Nedic2009}, with a refined rate analysis in~\cite{Yuan:2016dv}.  Each iteration combines a weighted average of the neighbors' estimates of the solution with a local gradient step;  the schemes differ  in how the two operations are interleaved.  The stepsize sets a trade-off between speed and accuracy: a constant stepsize $\alpha$ yields a {\it linear} rate, but only to an $\mathcal{O}(\alpha)$-neighborhood of the solution~\cite{Nedic2009,Yuan:2016dv,sayed2014adaptation},  whereas exact convergence requires a vanishing stepsize and hence a {\it sublinear} rate.

\textbf{(ii) Exact methods. } The speed-accuracy dilemma motivated a second generation of schemes aiming for {\it exact} solutions using a {\it constant} stepsize.      They differ in how the local gradient direction is corrected, along three routes.
  \textbf{(a)} The first acts  in the primal domain, either by tracking the global gradient $\nabla F$  \cite{Xu2015augmented,LorenzoScutari-J'16,qu2016harnessing,scutari2019distributed,sun2022distributed,nedic2017achieving,Pu20}, or through correction terms  based on past iterates and gradients \cite{shi2015extra,Berahas:bs}.  \textbf{(b)}  The second  enforces consensus
  among the
  local copies $x_i$'s through linear constraints, whose  (augmented) Lagrangian   multipliers      adjust the local gradient direction  \cite{shi2014linear,boyd2011distributed,6926737,pmlr-v70-scaman17a,hendrikx2021optimal}.
\textbf{(c)}  The third  turns the design into one of communication: dedicated routing or aggregation protocols replace gossip and  approximate an {\it all-reduce} average of the agents' iterates or gradients over the network, which then feed the centralized update.
Finite-time consensus~\cite{Sundaram08JSAC,Hendrickx14,jiang2025sparse}   factorizes the averaging matrix into finitely many graph-compliant communication steps~\cite{jiang2025sparse,doi:10.1137/24M1661455,lu2021optimal}, reproducing  a centralized gradient step   exactly every finitely many rounds. Routing over a spanning tree achieves the same end: the agents' updates are relayed along the tree with    hop-distance delays~\cite{vogels2021relaysum}, or the gradients are aggregated at a root that runs a centralized  accelerated method~\cite{pmlr-v70-scaman17a}.  

The routes are not independent: several primal exact methods admit a dual interpretation~\cite{Jakovetic:da}, and the third realizes with all-reduce averaging what the first two realize with mixing.
The  common approach is   to rewrite~\eqref{eq:clean_prob} as a {\it centralized} problem with linear consensus constraints, making centralized optimization techniques readily applicable. The  constraints are chosen to support implementation over the network, and   the resulting formulation  determines the dual variables on which the algorithm operates.
Specifically, stack the   copies   \(X=[x_1^\top;\ldots;x_m^\top]\), and call  $\mathsf A\in \mathbb{R}^{m\times q}$  {\it faithful} if    \(\ker(\mathsf A^\top)=\operatorname{span}\{\mathbf 1\}\) (\(\mathbf 1\) is the all-ones vector). Then,   \(\mathsf A^\top X=0\) enforces  exactly consensus.  
The lifted problem of \eqref{eq:clean_prob} and its dual read    \vspace{-0.1cm}
\begin{align}
&\min_{X\in\R^{m\times d}}\ \sum_{i\in\cV}f_i(x_i)
\qquad\text{s.t.}\quad \mathsf A^\top X=0,\nonumber
\\[-1mm]
&\min_{Y\in\R^{q\times d}}\ g_{\mathsf A}(Y):=\sum_{i\in\cV}f_i^*\bigl((\mathsf A Y)_i\bigr).\label{eq:clean_dual}\tag{D}\vspace{-0.2cm}
\end{align}
Here,  \(f_i^*\) is the Fenchel conjugate of $f_i$, and each  column \(j\) of \(\mathsf A\) introduces one   multiplier row \(Y_j\in \mathbb{R}^d\).  A dual solution $Y^\star$ recovers the primal solution through $x_i^\star=\nabla f_i^* ((\mathsf A Y^\star)_i)$, $i\in \mathcal V$.

Faithful choices used in the decentralized methods discussed above   
include   \textbf{(i)}  \(\mathsf A=I-W\), for a symmetric gossip matrix   \(W\)   with \(\ker(I-W)=\operatorname{span}\{\mathbf 1\}\)
\cite{Nedic2009,shi2015extra,Jakovetic:da,Xu2015augmented,LorenzoScutari-J'16,qu2016harnessing,scutari2019distributed,sun2022distributed,nedic2017achieving}; \textbf{(ii)} the graph Laplacian  or operators built from it ~\cite{6926737,pmlr-v70-scaman17a,hendrikx2021optimal};
  \textbf{(iii)} edge-indexed choices--the oriented node-edge incidence matrix \(B\) of \(\cG\), or that of a spanning tree, as in the master-node implementation of~\cite{pmlr-v70-scaman17a}; and \textbf{(iv)} the all-reduce averaging operator \(\mathsf A=I-\tfrac1m\mathbf 1\mathbf 1^\top\), realized through finite-time consensus or tree routing~\cite{lu2021optimal,doi:10.1137/24M1661455,vogels2021relaysum}.
These representations all enforce the {\it same} consensus constraint in \eqref{eq:clean_dual}, but generally induce {\it different} algorithms.
Each column \(j\) of \(\mathsf A\) creates a dual variable \(Y_j\), and its support identifies the conjugates \(f_i^*\)  that depend on it; grouping columns   determines which dual variables are optimized jointly.
  This is a design layer preceding the block partition,  absent from classical (centralized) block-coordinate methods (of Jacobi or Gauss--Seidel type), where the variables are given and only their grouping is a design choice.  Here the representation \(\mathsf A\)  first  determines the variables available for grouping and their dependencies.  

\textbf{The missing design step.} 
The methods above start from a prescribed centralized method:
a first- or a second-order scheme on \eqref{eq:clean_prob}, or a Lagrangian-based method on its consensus reformulation \eqref{eq:clean_dual}. 
A communication protocol is then added to supply  the information   the method
 needs to be implemented  on the network: gossip, spanning-tree relaying,
  finite-time averaging, or exchanges along an edge-supported
  representation \(\mathsf A\). {\it The network serves only as a communication medium} over which the method's variables are exchanged, mixed, or aggregated. Topology guides communication and parameter
  choices, but {\it does not influence the decomposition of the optimization problem}. 
   Long paths, dense regions, and sparse interfaces
  therefore serve communication, rather than define distinct
  optimization subproblems.    Representative convergence rate bounds reflect this separation.
  Function geometry enters through the  condition number of $F$, while
   communication enters through a spectral gap, an aggregation
  time, or a delay-dependent mixing constant.
  The rate combines these {\it separately} assessed quantities,
  possibly after
  acceleration~\cite{sun2022distributed,pmlr-v70-scaman17a,Li2020,vogels2021relaysum}.   Lower bounds, even when matched, establish limits of
  {\it prescribed} computation and communication models, {\it not
  limits of the graph topology itself}~\cite{pmlr-v70-scaman17a}.
  The missing design step is thus not another improvement of
  the optimization and communication factors separately, but
  a joint design in which graph structure determines the
  subproblems   and how agents  solve them.\\

 \indent The ring illustrates the potential of this change.
  On a ring of \(m\) agents, standard nearest-neighbor gossip averaging
  has inverse spectral gap \(\Theta(m^2)\)~\cite{Olshevsky09},
   despite a  diameter of \(\Theta(m)\).
  Non-accelerated consensus baselines inherit this quadratic
  network dependence
  (Sec.~\ref{subsec:clean_consensus_comparison}).
  For RelaySum on a spanning path, the analysis yields  an \(O(m^{3/2})\)   network factor, which becomes \(O(m^{5/2})\) for convex losses~\cite{vogels2021relaysum}.   Yet a path can do more than transport or average iterates or gradients.
  With edge-indexed dual variables, {\it one can  ``open'' the ring into a path}: one long path of edges is optimized jointly, with the remaining edges defining separate, individual subproblems. As shown in  Sec.~\ref{sec:clean_mp_jacobi},  the path subproblem can be  solved {\it exactly} by eliminating one node after another in $O(m)$ steps--structural grounds for $O(m)$ overall complexity, 
  rather than \(O(m^2)\).
The path   determines both what is optimized jointly and
  how the computation is distributed.
  Establishing a general framework for this joint design is
  the goal of this paper.\vspace{-0.5cm}

\subsection{From algorithm execution on a graph to graph decomposition}
 \vspace{-0.2cm}
We rethink   decomposition methods  from the foundations, accounting for the graph as part of the algorithm design.
   The major contributions are summarized next.\smallskip

 \noindent \textbf{(i) Graph decomposition framework.}  Starting from the dual formulation~\eqref{eq:clean_dual}, we introduce a formal paradigm of algorithmic decomposition for~\eqref{eq:clean_prob} over $\mathcal G$.  Two
observations shape the construction.   \begin{itemize}\item[\textbf{(a)}] {\bf The algebraic factor graph:}  The dependency of the conjugate terms \(f_i^*((\mathsf A Y)_i)\) on the dual coordinates \(Y_j\) is itself a graph--the bipartite {\it factor graph} induced by \(\mathsf A\):
  \begin{equation*}
  \cH_{\mathsf A}
  :=
  \bigl(
  \cV\sqcup\{1,\ldots,q\},\,
  \{(i,j):\mathsf A_{ij}\neq0\}
  \bigr),
  \end{equation*}
 with one factor node per conjugate \(f_i^*\), one variable node per dual
coordinate \(Y_j\), and an edge exactly when \(Y_j\) enters \(f_i^*\).   An algebraic decomposition is then obtained by
selecting \(p\) nonempty blocks of dual coordinates. A block
\(\cJ_r\subseteq\{1,\ldots,q\}\) specifies the coordinates optimized jointly,
while its neighborhood in \(\cH_{\mathsf A}\) identifies the conjugates
entering the resulting block subproblem.
\item[\textbf{(b)}] {\bf Blocks on the physical network:}  The factor graph records algebraic dependencies, not communication links: the agents involved in a block need not induce a connected subgraph of \(\cG\). Since \(\cG\) is connected, they can still cooperate through intermediate agents, possibly the whole network, at substantial communication cost. We therefore make the cooperating agents part of the design, assigning each block a \emph{carrier} \(\cC_r\subseteq\cV\): a set of agents containing all those whose conjugates enter the block subproblem and inducing a connected physical subgraph \(\cG[\cC_r]\). Additional carrier agents serve as relays.\end{itemize}

 The joint choice of the three--representation,
blocks, and carriers--is a {\it graph decomposition} of~\eqref{eq:clean_prob}
over \(\cG\): a triple
\begin{equation*}
  \left(
  \underbrace{\mathsf A}_{\text{representation}},\,
  \underbrace{\{\cJ_r\}_{r=1}^p}_{\text{dual-coordinate blocks}},\,
  \underbrace{\{\cC_r\}_{r=1}^p}_{\text{physical carriers}}
  \right).
  \end{equation*}
  Here, \(\mathsf A\) is faithful; the nonempty blocks \(\cJ_r\) cover \(\{1,
  \ldots,q\}\) and need not be disjoint; and each carrier \(\cC_r\subseteq\cV\) contains every agent \(i\)
  such that \(\mathsf A_{i,\cJ_r}\neq0\) and induces a connected subgraph   of \(\cG\). The three layers have distinct roles: the
representation creates the variables and their dependencies; the blocks fix
{\it what} is optimized jointly; the carriers fix {\it where} it is executed.   Sec.~\ref{sec:clean_problem_dual} develops this definition and  gives the triple its algorithmic meaning. 

\smallskip
\noindent\textbf{(ii) GATE as an instance.}
GATE (Graph-Tearing message passing)  instantiates the framework above by tearing the graph along its edges.  Agreement is imposed edge by edge, so that each edge carries one dual variable shared by its two endpoints only  (\(\mathsf A=\mathsf B\), the node-edge incidence matrix).  The edges are   partitioned into trees (\(\cP=\{\cE_r\}_{r=1}^p\), with \(\cJ_r=\cE_r\) and \(\cT_r=(\cV_r,\cE_r)\) a tree). Each tree defines one block optimized cooperatively by the
  agents in \(\cV_r\), with carrier \(\cC_r=\cV_r\).  On a tree, fixing one edge variable splits the block subproblem
  into two independent ones, one on each side of the edge.
  With out-of-block variables held fixed, the optimal value on
  each side, as a function of that edge variable, defines an
  endpoint \emph{cost-to-go message}, computed recursively from
  the leaves toward the edge.  GATE interleaves one step of the message recursion with updates
  of the edge variables. Each edge update is based on minimizing
  the sum of its two endpoint messages.
  The message updates use the latest out-of-block
  variables, accounting for the coupling between trees in a
  Jacobi fashion. The implementation is node-owned and uses single-hop
  exchanges, with no edge- or block-level coordinator  (Secs.~\ref{sec:clean_mp_jacobi}--\ref{subsec:clean_node_owned_dmp}). 

\noindent\textbf{(iii) Delayed block-Jacobi identification
  and convergence analysis.}  We establish linear convergence of GATE through an analysis of independent interest that identifies GATE's single-loop dynamics with an \emph{exact delayed block-Jacobi scheme}. The identification is  nontrivial: edge variables and messages evolve together, and the tree recursion is never completed while keeping the exterior coordinates fixed. Unfolding the messages through earlier rounds shows that, after a finite propagation phase, each edge response is a coordinate of an exact block minimizer. Its exterior coordinates are delayed by the distance from the edge to the boundary agents of the block. Linear convergence then follows by combining a block error
  bound with a perturbation analysis of these minimizers
  that exploits the tree structure of the delays.
  The resulting explicit rate is discussed next.

  \noindent\textbf{(iv) Block preconditioning, scalability,
  and partition design.}
  The   linear rate \(\rho\) reads:
\(\rho=1-\tau/(2\kappa_{\cP}),\)  where \(\tau\) is the
  relaxation stepsize and \(\kappa_{\cP}\) is a novel  block error-bound constant; it  measures the
  coupling remaining between block subproblems after exact
  block minimization; it depends on both the partition and
  the agents' cost functions.
  Three further quantities govern the admissible stepsize $\tau$, hence the final rate:  block overlap, the largest number of
  blocks sharing an agent; propagation delays, determined by within-tree distances from
  boundary agents to updated edges;
  and ``interface sensitivity'', measuring how changes in exterior coordinates affect block
  minimizers.
Function geometry and graph topology therefore enter the
  rate \emph{jointly}, unlike standard consensus bounds built
  from separately assessed conditioning and communication
  factors. The partition \(\cP\) lets us exploit this interaction:
  choosing which variables are minimized jointly changes the
  coupling between block subproblems--a
  \emph{block-preconditioning effect}--whose benefits must be
  balanced against block overlap, propagation delays, and
  interface sensitivity.

These rate dependences provide concrete guidance for partition design.  For instance, under the curvature conditions specified in Sec.~\ref{subsec:clean_rate_examples}, the bound on rings favors one long path block with \(O(1)\) singleton edges over equal-length paths.  The former gives \((1-\rho)^{-1}=O(m/\log m)\) and, under mild bounds on initial transients, an \(\epsilon\)-solution in \(O(m+(m/\log m)\log(1/\epsilon))\) rounds, including initial propagation.  For equal-length paths, the best bound we derive on
     \((1-\rho)^{-1}\) is \(O(m^{3/2})\), compared with
     \(O(m^2)\) for nonaccelerated consensus baselines.  The same principle gives \((1-\rho)^{-1}=O(D/\log D)\) on tree-dominant graphs with dominant-tree diameter \(\Theta(D)\), and   finite termination on trees when all edges form
     one block and \(\tau=1\) (Sec.~\ref{subsec:clean_rate_examples}).

\noindent\textbf{(v) Surrogate variant--GATE-S.}
To reduce GATE's per-iteration computation and communication,
  we develop GATE-S using tractable local surrogates.
  Our examples include first-order quadratic models with fixed
  curvature and second-order models whose curvature is updated
  using local Hessian information at the current iterate.
  We provide implementations requiring only vector-sized
  exchanges, using scalar curvature coefficients or incremental
  updates of locally cached Hessians.
  Under regularity  assumptions, we establish linear convergence with partition-dependent
  rates quantifying the effect of the surrogate approximation
  (Sec.~\ref{sec:clean_surrogate}).

\noindent\textbf{(vi) Numerical experiments.}  We conduct numerical experiments to validate the theoretical results and evaluate the iteration and communication efficiency of GATE and its surrogate variants relative to consensus-based methods. The results corroborate the theoretical scaling and demonstrate the efficiency of our algorithms in both iteration and communication.

\vspace{-0.3cm}

\vspace{-0.2cm}
\subsection{Notation and paper organization}
\label{subsec:intro_notation}\vspace{-0.2cm}

For a positive integer \(q\), write \([q]:=\{1,\ldots,q\}\), and use \(\bigsqcup\) for disjoint union.  For functions \(F(x,y)\) and \(G(x,y)\), write \(F\sim_x G\) when \(F-G\) is independent of \(x\).  When the variable is clear from context, we write \(F\sim G\).

We use row-stacked matrices for \(d\)-vectors.  If the rows of a matrix \(A\) are indexed by a finite set \(\mathcal J\), then \(A_\alpha\) denotes the row indexed by \(\alpha\in\mathcal J\), and \(A_S\) denotes the submatrix of rows indexed by \(S\subseteq\mathcal J\).  For matrices with both row and column indices, \(A_{I,S}\) denotes the submatrix with rows \(I\) and columns \(S\).  Whenever such a row vector appears as an argument of \(f_i^*\), in an inner product, or in a norm, it is identified with the corresponding vector in \(\R^d\).  For node variables \((x_i)_{i\in\cV}\), \(X\) denotes the row-stacked matrix whose \(i\)-th row is \(x_i^\top\).  Inner products between row-stacked matrices use the Frobenius convention.  For scalar functions on row-stacked matrices, \(\nabla\Phi(A)^\top\) denotes the row-stacked gradient.  We write \(\mathbb S^d\) for the space of \(d\times d\) symmetric matrices, and \(\mathbb S_+^d\) and \(\mathbb S_{++}^d\) for its positive semidefinite and positive definite cones.  For a scalar differentiable function \(\phi\), \(\nabla\phi\) is the column gradient and \(\nabla^2\phi\) is the Hessian.  For a vector-valued map \(F=(F_1,\ldots,F_k):\R^d\to\R^k\), we set \(\nabla F(x):=[\nabla F_1(x),\ldots,\nabla F_k(x)]\in\R^{d\times k}\); hence \(\nabla F(x)^\top\) is the Jacobian and \(\nabla F(x)^\top h\) is the directional variation along \(h\).  For a matrix-valued map \(H\), \(\nabla H(x)[h]\) denotes the directional derivative in direction \(h\).  For \(T:\R^d\to\R^d\), write
  \(
      \operatorname{Lip}(T)
      :=
      \sup_{\eta\ne\eta'}
       {\|T(\eta)-T(\eta')\|}/{\|\eta-\eta'\|}
  \)
  for its global Lipschitz modulus. We use \(O(\cdot)\) for asymptotic upper bounds and \(\Theta(\cdot)\) for matching asymptotic upper and lower bounds, up to positive constant factors.

\paragraph{Organization.}
The remainder of the paper is organized as follows. Sec.~\ref{sec:clean_problem_dual} develops the graph decomposition framework. Sec.~\ref{sec:clean_gate} specializes it to the incidence representation and tree blocks to derive GATE. Sec.~\ref{sec:clean_convergence} establishes linear convergence and partition-dependent rate bounds. Sec.~\ref{sec:clean_surrogate} develops the surrogate variants of GATE,   establishes their  convergence, and compare their performance with consensus-based and tree-based state-of-the-art methods. Sec.~\ref{sec:clean_numerics} presents numerical results, and Sec.~\ref{sec:conclusion} concludes.
\vspace{-0.3cm}
 
 \section{Graph Decomposition Framework}
\label{sec:clean_problem_dual}

\subsection{Problem and graph setting}
\label{subsec:clean_problem_setup}
We study \eqref{eq:clean_prob} over the graph $\mathcal G$
under the following standard assumptions.

\begin{assumption}[problem class]
\label{ass:clean_problem_class}
For every $i\in\cV$, $f_i:\R^d\to\R$ is $\mu_{f_i}$-strongly convex and
$L_{f_i}$-smooth for some $\mu_{f_i},L_{f_i}>0$.  Define\vspace{-0.2cm}
\[
    \mu:=\min_{i\in\cV}\mu_{f_i},
    \quad
    L:=\max_{i\in\cV}L_{f_i},
    \quad
    \kappa:=\frac{L}{\mu}.
\]
\end{assumption}
 
 Consequently, each \(f_i^*\) is finite-valued, differentiable, \(1/L_{f_i}\)-strongly convex and \(1/\mu_{f_i}\)-smooth.

\begin{assumption}[Network]
\label{ass:clean_network}
The graph $\cG=(\mathcal V,\mathcal E)$ is undirected, connected, and simple.
\end{assumption}
Fix an arbitrary orientation of $\cE$ and let
$\mathsf B\in\R^{m\times|\cE|}$ be the node--edge incidence matrix, with
\[
  \mathsf  B_{ve}:=
    \begin{cases}
        1, & \text{if \(v\) is the tail of \(e\)},\\
        -1, & \text{if \(v\) is the head of \(e\)},\\
        0, & \text{otherwise}.
    \end{cases}
\]
For $v\in\cV$ and $e\in\cE$, let 
  $$\cA_v:=\{e\in\cE:\ e\text{ is incident to }v\}
\quad\text{and}\quad
 \cI_e:=\{v\in\cV:\ v\text{ is an endpoint of }e\}. $$

The node-edge incidence graph of a graph is the bipartite
graph on its nodes and edges in which a node and an edge are adjacent exactly
when the node is an endpoint of the edge.

For a physical subgraph
$\cG_s\subseteq\cG$, write ${\mathcal V}(\cG_s)$ and ${\mathcal E}(\cG_s)$ for its vertex and edge sets; for
$S\subseteq\cV$, $\cG[S]$ denotes the subgraph of $\cG$ induced by $S$.\vspace{-0.3cm}

\subsection{Graph decomposition}
This section builds the proposed decomposition framework.

\vspace{-0.2cm}
\subsubsection{Representations and their factor graphs}
\label{subsec:factor_graphs}
Recall from Sec.~1 that a matrix \(\mathsf A\in\R^{m\times q}\) is
{\it faithful} if \(\ker(\mathsf A^\top)=\operatorname{span}\{\mathbf 1\}\),
so that \(\mathsf A^\top X=0\) holds exactly for consensual stacked copies
\(X=[x_1^\top;\ldots;x_m^\top]\).  As derived there, the lifted
consensus form of the problem  has the dual formulation~\eqref{eq:clean_dual},
with one multiplier row \(Y_j\in\R^d\) per column \(j\) of \(\mathsf A\).
Every faithful \(\mathsf A\) represents the same consensus set, but the dual
\eqref{eq:clean_dual} depends on \(\mathsf A\) itself: the columns create the
dual coordinates, and the sparsity pattern of \(\mathsf A\) decides which
coordinates enter which conjugate.  This dependency structure is captured by the factor graph.
\begin{definition}[Factor graph]
\label{def:clean_factor_graph}
The {\it factor graph induced by \(\mathsf A\)} is the bipartite graph
\begin{equation*}
  \cH_{\mathsf A}
  :=
  \bigl(
    \cV\sqcup\{1,\ldots,q\},\;
    \{(i,j):\mathsf A_{ij}\neq0\}
  \bigr),
\end{equation*}
whose factor node \(i\) represents the term \(f_i^*((\mathsf A Y)_i)\) and
whose variable node \(j\) represents the dual coordinate \(Y_j\).
\end{definition}

The two sides of \(\cH_{\mathsf A}\) are not symmetric: each factor node
is anchored at a physical agent, whereas a variable node need not have  a prescribed physical location
on \(\cG\).  Consequently, a single variable \(Y_j\)--column
\(\mathsf A_{:,j}\)--may enter the conjugates of nonadjacent agents.
\begin{example}\label{ex:clean_faithful_matrices}
 Given the  path graph  \(1-2-3\), with both edges oriented from left to
right (Figure~\ref{fig:clean_alignment}(a)),  consider the following three faithful
  choices of \(\mathsf A\):\vspace{-0.1cm}
\begin{equation}
    \mathsf A=\mathsf A_{\rm ex}:=
    \begin{bmatrix}
      1&1\\
      0&-2\\
      -1&1
    \end{bmatrix},
    \qquad
    \mathsf A=\mathsf L:=\frac{1}{3}\mathsf B\mathsf B^\top,
    \qquad\text{or}\qquad
    \mathsf A=\mathsf B.
    \label{eq:example_faithful_matrices}\vspace{-0.2cm}
  \end{equation}
  Here, \(\mathsf B\) is the incidence matrix of \(\cG\), and \(\mathsf L\) is
  its scaled Laplacian. The three matrices enforce the same consensus constraint,
yet they induce the distinct factor graphs of
Figure~\ref{fig:clean_alignment}(b)-(d).  \footnote{In panels~(b) and~(c), the subscripts of \(Y\) index the columns of
  \(\mathsf A_{\rm ex}\) and \(\mathsf L\), respectively; in panel~(d),
  \(Y_{12}\) and \(Y_{23}\) are indexed by the physical edges \((1,2)\) and
  \((2,3)\).}   For \(\mathsf A_{\rm ex}\), the
coordinate \(Y_1\) enters the factors of the {\it nonadjacent} agents \(1\)
and \(3\).  For \(\mathsf L\), \(Y_1\) and \(Y_2\) both enter \(f_1^*\) and
\(f_2^*\), so \(\cH_{\mathsf L}\) contains a {\it cycle} although \(\cG\) is
a path.  For \(\mathsf B\), each coordinate corresponds to {\it one physical
edge} and enters exactly its endpoint factors, so \(\cH_{\mathsf B}\) is the
node-edge incidence graph of \(\cG\),   that is, the bipartite
graph on \(\cV\sqcup\cE\) in which \(i\) and \(e\) are adjacent exactly when
\(i\in\cI_e\).
   This distinction matters when the dual coordinates are grouped and optimized
jointly: a group of coordinates of \(\mathsf B\) {\it is} a set of physical
edges, one per coordinate, whereas a group under \(\mathsf L\) or
\(\mathsf A_{\rm ex}\) may 
couple nonadjacent agents.\end{example}
\begin{figure}[t]
\centering
\resizebox{0.82\linewidth}{!}{%
\begin{tikzpicture}[
  every node/.style={font=\small},
  fnode/.style={draw,fill=black!12,rectangle,minimum size=5mm,inner sep=1pt},
  vnode/.style={draw,circle,minimum size=5mm,inner sep=1pt},
  pnode/.style={draw,circle,fill=black!12,minimum size=5mm,inner sep=1pt}
]
\begin{scope}[xshift=0cm,yshift=4cm]
  \node[pnode] (g1) at (0,0) {$1$};
  \node[pnode] (g2) at (1.3,0) {$2$};
  \node[pnode] (g3) at (2.6,0) {$3$};
  \draw (g1)--(g2)--(g3);
  \node[align=center] at (1.3,-1.8) {(a) physical path $\cG$};
\end{scope}
\begin{scope}[xshift=5cm,yshift=4cm]
  \node[fnode] (f1) at (0,.8) {$f_1^*$};
  \node[fnode] (f2) at (1.3,.8) {$f_2^*$};
  \node[fnode] (f3) at (2.6,.8) {$f_3^*$};
  \node[vnode] (v1) at (.65,-.8) {$Y_1$};
  \node[vnode] (v2) at (1.95,-.8) {$Y_2$};
  \draw (f1)--(v1) (f3)--(v1);
  \draw (f1)--(v2) (f2)--(v2) (f3)--(v2);
  \node[align=center] at (1.3,-1.8)
    {(b) $\cH_{\mathsf A_{\rm ex}}$};
\end{scope}
\begin{scope}[xshift=0cm]
  \node[fnode] (m1) at (0,.8) {$f_1^*$};
  \node[fnode] (m2) at (1.3,.8) {$f_2^*$};
  \node[fnode] (m3) at (2.6,.8) {$f_3^*$};
  \node[vnode] (z1) at (0,-.8) {$Y_1$};
  \node[vnode] (z2) at (1.3,-.8) {$Y_2$};
  \node[vnode] (z3) at (2.6,-.8) {$Y_3$};
  \draw (m1)--(z1) (m2)--(z1);
  \draw (m1)--(z2) (m2)--(z2) (m3)--(z2);
  \draw (m2)--(z3) (m3)--(z3);
  \node[align=center] at (1.3,-1.8)
    {(c) $\cH_{\mathsf L}$};
\end{scope}
\begin{scope}[xshift=5cm]
  \node[fnode] (h1) at (0,.8) {$f_1^*$};
  \node[fnode] (h2) at (1.3,.8) {$f_2^*$};
  \node[fnode] (h3) at (2.6,.8) {$f_3^*$};
  \node[vnode] (y12) at (.65,-.8) {$Y_{12}$};
  \node[vnode] (y23) at (1.95,-.8) {$Y_{23}$};
  \draw (h1)--(y12) (h2)--(y12);
  \draw (h2)--(y23) (h3)--(y23);
  \node[align=center] at (1.3,-1.8) {(d) $\cH_B$};
\end{scope}
\end{tikzpicture}
}
\caption{Three faithful representations (see  \eqref{eq:example_faithful_matrices}) of the same path graph \(1-2-3\) induce distinct factor graphs: the generic \(\mathsf A_{\rm ex}\) couples nonadjacent factors, the Laplacian \(\mathsf L\) produces a cyclic factor graph, and the incidence matrix \( \mathsf B\) reproduces the node-edge incidence structure.}
\label{fig:clean_alignment}\vspace{-0.4cm}
\end{figure}
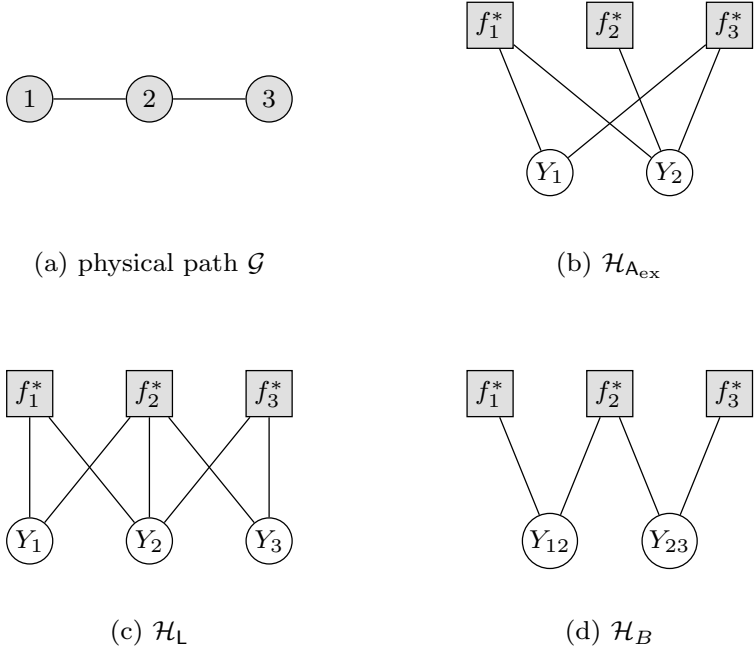
\vspace{-0.3cm}

\subsubsection{Block covers and conditional responses}\vspace{-0.2cm}
To minimize \(g_{\mathsf A}\) {\it blockwise}, we
  group the dual coordinates into (possibly overlapping) blocks. Each
  block-subproblem minimizes \(g_{\mathsf A}\) over its block coordinates, holding   the remaining coordinates fixed. 
\begin{definition}[Block cover and conditional responses]
\label{def:clean_block_decomposition}
For a fixed $\mathsf A$, a \emph{block cover} 
is a family $\{\cJ_r\}_{r=1}^p$ of nonempty subsets  covering  $[q]$ (the index set of the dual coordinates). Let  \(\overline{\cJ}_r:=[q]\setminus\cJ_r\). 
For each block \(\cJ_r\), define its touched factor set    and block factor
graph by\vspace{-0.1cm}
\begin{equation}
  \cV_{\mathsf A,r}
  :=
  \{i\in\cV:\mathsf A_{i,\cJ_r}\neq0\},
  \qquad
  \cF_{\mathsf A,r}
  :=
  \cH_{\mathsf A}[\cV_{\mathsf A,r}\sqcup\cJ_r].
  \label{eq:clean_generic_block_cluster}\vspace{-0.1cm}
\end{equation}
Here, \(\cH_{\mathsf A}[S]\) denotes the subgraph of \(\cH_{\mathsf A}\) induced by the vertex set \(S\).

The   conditional response
  of block \(r\) is\vspace{-0.1cm}
\begin{equation}
  \BR_{\mathsf A,r}(Y_{\overline{\cJ}_r})
  :=
  \argmin_{W\in\R^{|\cJ_r|\times d}}
  \sum_{i\in\cV_{\mathsf A,r}}
  f_i^*\!\left(
    \mathsf A_{i,\cJ_r}W
    +\mathsf A_{i,\overline{\cJ}_r}Y_{\overline{\cJ}_r}
  \right).
  \label{eq:clean_general_block_response}
\end{equation}
Under Assumption~\ref{ass:clean_problem_class}, \(\BR_{\mathsf A,r}(Y_{\overline{\cJ}_r})\) is nonempty for every \(r\in[p]\) and \(Y_{\overline{\cJ}_r}\), but need not be a singleton.
When $\mathsf A$ is fixed, we suppress it and write $\cV_r$, $\cF_r$,
and $\BR_r$.
\end{definition}

Thus \(\cJ_r\) identifies the coordinates optimized jointly,
  \(\cV_{\mathsf A,r}\) the agents whose conjugates enter the block
  subproblem, and \(\cF_{\mathsf A,r}\) records which block coordinates
  enter each conjugate.

\begin{proposition}[Block fixed points]
\label{prop:clean_generic_block_fixed_point} Let
Assumption~\ref{ass:clean_problem_class}
  hold, let \(\mathsf A\) be
faithful, and let \(\{\cJ_r\}_{r=1}^p\) be a block cover.   Then,
  \begin{equation}
    Y^\star\in\argmin_{Y\in\R^{q\times d}} g_{\mathsf A}(Y)
    \quad\Longleftrightarrow\quad
    Y_{\cJ_r}^\star\in
    \BR_r(Y_{\overline{\cJ}_r}^\star)
    \quad \forall r\in[p].
    \label{eq:clean_generic_block_fixed_point}
  \end{equation}
\end{proposition}

 Proposition~\ref{prop:clean_generic_block_fixed_point} holds for
  {\it every} faithful \(\mathsf A\) and {\it every} block cover:
  optimality restricts neither choice.  Different choices produce
different response maps $\BR_{\mathsf A,r}(Y_{\overline{\cJ}_r})$--hence, potentially,  different algorithms--but the algebraic decomposition \((\mathsf A,\{\cJ_r\})\)  does not
  specify where these responses are computed on \(\cG\).

Since \(f_i^*\) is available only at agent \(i\), computing \(\BR_r\)
requires the agents indexed by \(\cV_r\) to exchange information over the
links of \(\cG\).  These exchanges can be confined to the agents of
\(\cV_r\) only if \(\cG[\cV_r]\) is connected, which
Definition~\ref{def:clean_block_decomposition} does not guarantee.  In
Example~\ref{ex:clean_faithful_matrices}, the subproblem of the block
\(\{1\}\) (the variable \(Y_1\)) of \(\mathsf A_{\rm ex}\) contains the
conjugates \(f_1^*\) and \(f_3^*\), stored at the nonadjacent agents \(1\)
and \(3\), which can communicate only through agent \(2\), whose conjugate
does not enter the subproblem.  The computation of a block response thus
involves, besides the agents of \(\cV_r\), those that relay between them;
the set of all agents taking part in it must induce a connected subgraph
containing \(\cV_r\).  This is formalized next.\vspace{-0.3cm}

\subsubsection{Carriers and realizability}\vspace{-0.2cm}
 We associate with each block \(\cJ_r\) a set of agents \(\cC_r\)
that enable the joint computation of  the conditional response \(\BR_r\) of the block on \(\cG\).  This
set, called the \emph{carrier} of \(\cJ_r\), must induce a connected
subgraph of \(\cG\) and contain the agents in \(\cV_r\), whose conjugates
enter the block subproblem \eqref{eq:clean_general_block_response}.  We
formalize the joint choice of representation, block cover, and carriers
as a {\it graph decomposition}.

\begin{definition}[Graph decomposition]
\label{def:clean_graph_decomposition}
A {\it graph decomposition over \(\cG\)} is a triple
\begin{equation}
  \cD:=\bigl(\mathsf A,\ \{\cJ_r\}_{r=1}^p,\ \{\cC_r\}_{r=1}^p\bigr),
  \label{eq:clean_triple}
\end{equation}
where \(\mathsf A\) is faithful,
\(\{\cJ_r\}_{r=1}^p\) is a block cover of \(g_{\mathsf A}\), and the  set \(\emptyset\neq \cC_r\subseteq\cV\) is the {\it carrier} of block \(r\) satisfying
\begin{equation}
       \cV_{r}\subseteq\cC_r
  \qquad\text{and}\qquad
  \cG[\cC_r]\ \text{is connected}.
  \label{eq:clean_carrier_conditions}
\end{equation}
\end{definition}

The three components have distinct roles: \(\mathsf A\) creates the dual
coordinates and their dependencies, \(\cJ_r\) fixes what is optimized
jointly, and \(\cC_r\) specifies the subgraph on which that computation is
performed.  The agents indexed by \(\cC_r\setminus\cV_r\) contribute no
term to the subproblem \eqref{eq:clean_general_block_response}: their only
role is to relay information between agents of \(\cV_r\) that are not
adjacent in \(\cG\) (as agent \(2\) does for the block \(\{1\}\) of
\(\mathsf A_{\rm ex}\) in Example~\ref{ex:clean_faithful_matrices}).  When
\(\cG[\cV_r]\) is connected, the carrier can be taken equal to \(\cV_r\)
and no additional relay agents are required.

For a subgraph \(\cG_s=(\cV_s,\cE_s)\subseteq\cG\), let
  \(\mathsf B_{\cG_s}\in\R^{m\times|\cE_s|}\) be the matrix
  obtained from \(\mathsf B\) by retaining the columns indexed
  by \(\cE_s\). The rows indexed by
  \(\cV\setminus\cV_s\) are zero. 
The carrier conditions~\eqref{eq:clean_carrier_conditions}
admit the following equivalent matrix factorization.
\begin{proposition}[realizability]
\label{prop:clean_realizability}
Fix \(r\in[p]\). Let \(\mathsf A\) be faithful, let \(\emptyset\neq\cJ_r\subseteq[q]\), and
let \(\emptyset\neq\cC_r\subseteq\cV\). Conditions
\eqref{eq:clean_carrier_conditions} hold if and only if there exist a
connected subgraph \(\cG_r\subseteq\cG\), with \(\mathcal V(\cG_r)=\cC_r\), and a
matrix \(M_r\in\R^{|\cJ_r|\times|\mathcal E(\cG_r)|}\) such that\vspace{-0.1cm}
\begin{equation}
  \mathsf A_{:,\cJ_r}=\mathsf B_{\cG_r}M_r^\top.
  \label{eq:clean_realizability}\vspace{-0.1cm}
\end{equation}
\end{proposition}


  Proposition~\ref{prop:clean_realizability} decomposes the linear couplings
  recorded by adjacency in the block factor graph \(\cF_r\)
  (see \eqref{eq:clean_generic_block_cluster})--which agents'
  conjugate terms depend on each dual coordinate \(Y_j\)--into couplings
  between adjacent agents in the connected physical subgraph \(\cG_r\)
  (with vertex set \(\cC_r\)). This opens the way to design methods that
  compute \(\BR_r\) through neighbor-to-neighbor communication within
  the carrier.

  To make the connection explicit, recall from
  \eqref{eq:clean_general_block_response} that each block coordinate
  \(Y_j\) enters the argument of \(f_i^*\) with coefficient
  \(\mathsf A_{ij}\).
  By \eqref{eq:clean_realizability}, \(\mathsf A_{:,j}\) is a linear
  combination of the columns of \(\mathsf B\) indexed by the edges
  of \(\cG_r\), with coefficients given by the corresponding row
  of \(M_r\); its nonzero entries are the agents whose factor nodes are adjacent to the
variable node \(j\) in \(\cF_r\).  If the row of \(M_r\) corresponding to \(Y_j\) has a single nonzero
  entry, \(Y_j\) appears only in the conjugates  \(f_i^*\) of the two endpoint
  agents of one physical edge, which are neighbors in \(\cG\) and
  can communicate directly.
 If that row has several nonzero entries, \(Y_j\) may appear in
  the conjugates  of nonadjacent agents or of more than two agents;
  these agents can exchange information along the edges of
  \(\cG_r\), with intermediate agents relaying it when needed.
At agents in \(\cC_r\setminus\cV_r\), the linear combinations
  in \eqref{eq:clean_realizability} vanish for every block coordinate:
  their conjugate terms are independent of the block variables,
  and these agents can serve as relays.
  Example~\ref{ex:clean_faithful_matrices} revisited and
  Figure~\ref{fig:clean_carriers} illustrate these distinctions.

\begin{figure}[t!]
\centering
\resizebox{0.98\linewidth}{!}{%
\begin{tikzpicture}[
  every node/.style={font=\small},
  fnode/.style={draw,fill=black!12,rectangle,minimum size=5mm,inner sep=1pt},
  vnode/.style={draw,circle,minimum size=5mm,inner sep=1pt},
  pnode/.style={draw,circle,fill=black!12,minimum size=5mm,inner sep=1pt},
  rnode/.style={draw,dashed,circle,minimum size=5mm,inner sep=1pt},
  onode/.style={draw=black!30,circle,minimum size=5mm,inner sep=1pt,text=black!30}
]
\begin{scope}[xshift=0cm,yshift=3.2cm]
  \node[fnode] (f1) at (0,.8) {$f_1^*$};
  \node[fnode] (f3) at (2.6,.8) {$f_3^*$};
  \node[vnode] (v1) at (1.3,-.6) {$Y_1$};
  \draw (f1)--(v1) (f3)--(v1);
  \node[align=center] at (1.3,-1.5) {(a) $\cF_{\mathsf A_{\rm ex},\{1\}}$};
\end{scope}
\begin{scope}[xshift=5cm,yshift=3.2cm]
  \node[fnode] (m1) at (0,.8) {$f_1^*$};
  \node[fnode] (m2) at (1.3,.8) {$f_2^*$};
  \node[fnode] (m3) at (2.6,.8) {$f_3^*$};
  \node[vnode] (z2) at (1.3,-.6) {$Y_2$};
  \draw (m1)--(z2) (m2)--(z2) (m3)--(z2);
  \node[align=center] at (1.3,-1.5) {(b) $\cF_{\mathsf L,\{2\}}$};
\end{scope}
\begin{scope}[xshift=10cm,yshift=3.2cm]
  \node[fnode] (h1) at (0,.8) {$f_1^*$};
  \node[fnode] (h2) at (1.3,.8) {$f_2^*$};
  \node[vnode] (y1) at (.65,-.6) {$Y_{e_1}$};
  \draw (h1)--(y1) (h2)--(y1);
  \node[align=center] at (1.3,-1.5) {(c) $\cF_{\mathsf B,\{e_1\}}$};
\end{scope}
\begin{scope}[xshift=0cm]
  \node[pnode] (a1) at (0,0) {$1$};
  \node[rnode] (a2) at (1.3,0) {$2$};
  \node[pnode] (a3) at (2.6,0) {$3$};
  \draw (a1)--(a2) node[midway,above] {\scriptsize $1$};
  \draw (a2)--(a3) node[midway,above] {\scriptsize $1$};
  \node[align=center] at (1.3,-1.5) {(d) $\cG_r$ for block $\{1\}$ of $\mathsf A_{\rm ex}$\\[1pt] $\cV_r=\{1,3\}$, $\cC_r=\{1,2,3\}$};
\end{scope}
\begin{scope}[xshift=5cm]
  \node[pnode] (b1) at (0,0) {$1$};
  \node[pnode] (b2) at (1.3,0) {$2$};
  \node[pnode] (b3) at (2.6,0) {$3$};
  \draw (b1)--(b2) node[midway,above] {\scriptsize $-\tfrac13$};
  \draw (b2)--(b3) node[midway,above] {\scriptsize $\tfrac13$};
  \node[align=center] at (1.3,-1.5) {(e) $\cG_r$ for block $\{2\}$ of $\mathsf L$\\[1pt] $\cV_r=\cC_r=\{1,2,3\}$};
\end{scope}
\begin{scope}[xshift=10cm]
  \node[pnode] (c1) at (0,0) {$1$};
  \node[pnode] (c2) at (1.3,0) {$2$};
  \node[onode] (c3) at (2.6,0) {$3$};
  \draw (c1)--(c2) node[midway,above] {\scriptsize $1$};
  \draw[dashed,black!30] (c2)--(c3);
  \node[align=center] at (1.3,-1.5) {(f) $\cG_r$ for block $\{e_1\}$ of $\mathsf B$\\[1pt] $\cV_r=\cC_r=\{1,2\}$};
\end{scope}
\begin{scope}[xshift=14.2cm,yshift=1.6cm]
  \node[fnode] at (0,1.4) {};  \node[anchor=west] at (.35,1.4) {conjugate of an agent};
  \node[vnode] at (0,.7) {};   \node[anchor=west] at (.35,.7) {dual coordinate};
  \node[pnode] at (0,0) {};    \node[anchor=west] at (.35,0) {touched agent};
  \node[rnode] at (0,-.7) {};  \node[anchor=west] at (.35,-.7) {relay agent};
  \node[onode] at (0,-1.4) {}; \node[anchor=west] at (.35,-1.4) {outside the carrier};
\end{scope}
\end{tikzpicture}
}
\caption{Example~\ref{ex:clean_faithful_matrices} revisited.  Top: block factor graphs.  Bottom: carrier subgraphs \(\cG_r\), with the entries of \(M_r\) on the edges.
}
\label{fig:clean_carriers}\vspace{-0.3cm}
\end{figure}
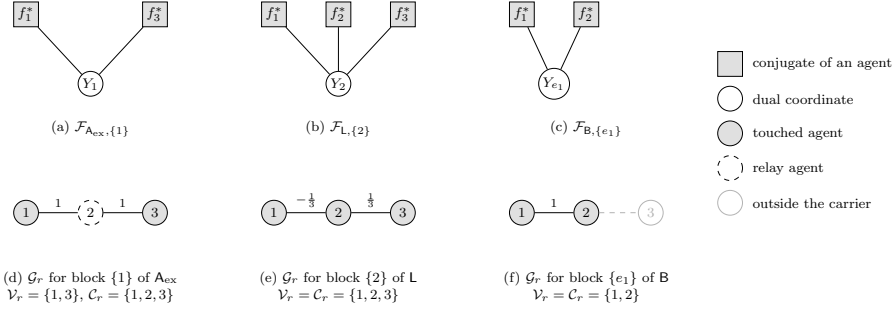

\paragraph{Example~\ref{ex:clean_faithful_matrices} revisited.}
 Let \(e_1=\{1,2\}\) and  \(e_2=\{2,3\}\), with both edges oriented
  from left to right.
  For each representation in Example~\ref{ex:clean_faithful_matrices},
  choose the singleton block cover, with one block per dual coordinate.
  We examine the blocks \(\{1\}\) and \(\{2\}\) of
  \(\mathsf A_{\rm ex}\), the block \(\{2\}\) of \(\mathsf L\),
  and the block \(\{e_1\}\) of \(\mathsf B\).  Their conditional responses, with terms in \(g_{\mathsf A}\) independent of \(W\) omitted, are
\begin{equation*}
\begin{aligned}
  \BR_{\mathsf A_{\rm ex},\{1\}}(Y_2)
  &=\argmin_W\bigl\{f_1^*(W+Y_2)+f_3^*(-W+Y_2)\bigr\},\\
  \BR_{\mathsf A_{\rm ex},\{2\}}(Y_1)
  &=\argmin_W\bigl\{f_1^*(Y_1+W)+f_2^*(-2W)+f_3^*(-Y_1+W)\bigr\},\\
  \BR_{\mathsf L,\{2\}}(Y_1,Y_3)
  &=\argmin_W\bigl\{f_1^*\bigl(\tfrac13(Y_1-W)\bigr)
  +f_2^*\bigl(\tfrac13(2W-Y_1-Y_3)\bigr)
  +f_3^*\bigl(\tfrac13(Y_3-W)\bigr)\bigr\},\\
  \BR_{\mathsf B,\{e_1\}}(Y_{e_2})
  &=\argmin_W\bigl\{f_1^*(W)+f_2^*(Y_{e_2}-W)\bigr\}.
\end{aligned}
\end{equation*}
The factor sets touched by the four blocks, read off from the responses,
are \(\cV_r=\{1,3\}\) for the first, as in Figure~\ref{fig:clean_carriers}(a);
\(\{1,2,3\}\) for the second and third, with the Laplacian case shown in
panel~(b); and \(\{1,2\}\) for the last, as in panel~(c). Only the first
fails to induce a connected subgraph of \(\cG\).  Its carrier must
therefore contain agent~\(2\), giving \(\cC_r=\{1,2,3\}\) (panel~(d)): this agent relays between agents~\(1\) and~\(3\) without
contributing a conjugate to the subproblem.  For the other three blocks,
choose \(\cC_r=\cV_r\); panels~(e),(f) illustrate the Laplacian and
incidence cases.  Taking \(\cG_r=\cG[\cC_r]\) gives the whole path for
the first three blocks and the single edge \(e_1\) for the last, so that
\(\mathsf B_{\cG_r}=\mathsf B_{:,e_1}\) there.  The corresponding
factorizations from Proposition~\ref{prop:clean_realizability} are\vspace{-0.2cm}
\begin{equation}
\begin{aligned}
  (\mathsf A_{\rm ex})_{:,1}
  &=\mathsf B_{:,e_1}+\mathsf B_{:,e_2},
  &\qquad
  (\mathsf A_{\rm ex})_{:,2}
  &=\mathsf B_{:,e_1}-\mathsf B_{:,e_2},\\
  \mathsf L_{:,2}
  &=\tfrac13\bigl(-\mathsf B_{:,e_1}+\mathsf B_{:,e_2}\bigr),
  &\qquad
  \mathsf B_{:,e_1}
  &=\mathsf B_{\cG_r}.
\end{aligned}
\label{eq:clean_path_cases}\vspace{-0.1cm}
\end{equation}
The matrices \(M_r\) are respectively \((1,1)\), \((1,-1)\),
\(\tfrac13(-1,1)\), and \((1)\); their entries label the physical edges in
Figure~\ref{fig:clean_carriers}(d)--(f) for the illustrated blocks.  In
the first identity, the two columns of \(\mathsf B\) have opposite entries
at agent~\(2\), which cancel.  Thus \(f_2^*(-2Y_2)\) is independent of
\(W\): agent~\(2\) is absent from the block factor graph but present in
the physical subgraph as a relay.  The second and third combinations have
nonzero entries at all three agents, so every agent contributes a
conjugate to the response.  For the last block, \(M_r=(1)\) uses only
\(e_1\), whose endpoint agents are neighbors and communicate directly.
Here \(f_3^*\) is absent because \(e_1\) is not incident to agent~\(3\),
rather than through cancellation.\vspace{-0.3cm}

\section{ Graph Tearing Along Edges: Message Passing on Tree Blocks}\vspace{-0.2cm}
\label{sec:clean_gate}

Proposition~\ref{prop:clean_generic_block_fixed_point}
    holds for every faithful \(\mathsf A\) and every block cover, while Proposition~\ref{prop:clean_realizability} characterizes the admissible carriers (i.e., satisfying \eqref{eq:clean_carrier_conditions}).
  The framework leaves these choices and the methods for computing the block responses open, with   different communication and computation costs. Choices   range from 
  Laplacian-type representations
  in which a coordinate couples an agent with all its neighbors, to
  the incidence  representation, in which a coordinate couples one pair of neighbors. %

 We now use this freedom to design a decentralized algorithm, termed GATE
(Graph-Tearing message passing). 
We choose a graph decomposition in which
{\it selecting dual coordinates amounts to selecting edges of \(\cG\)}, so that
each coordinate appears in the conjugates of two neighboring agents, and
a block response can be computed by exchanges between the neighboring agents  whose conjugates enter it, without relay agents. More formally, we have the following.\vspace{-0.2cm}
 \begin{itemize}\item[\textbf{(i)}] \textbf{Representation:} We choose \(\mathsf A\) so
that, for each block, the factorization~\eqref{eq:clean_realizability}
admits a matrix \(M_r\) with {\it one nonzero entry per row} and with
{\it different} rows selecting {\it distinct} edges. Each coordinate then
enters exactly the conjugates at the two endpoints of its selected edge,
avoiding redundant coordinates on the same edge. Let \(\cE_r\) denote the
set of these edges. Indexing the coordinates by \(\cE_r\) and absorbing
each nonzero coefficient into its coordinate
 (a change of variables in the dual~\eqref{eq:clean_dual}) gives
  \(\mathsf A_{:,\cJ_r}=\mathsf B_{:,\cE_r}\).  
   \item[\textbf{(ii)}] \textbf{Partition: } We retain one coordinate for every physical edge, leaving the whole network available for block
selection. We partition the  edge-set $\mathcal E$  into  nonempty sets \(\cP:=\{\cE_r\}_{r=1}^p\), so that each variable receives one block
update. Write
  \(\cT_r:=(\cV_r,\cE_r)\) for the subgraph selected by block \(r\); by Definition 2, $\cV_r$ is the set of endpoints of $\cE_r$.
  We take each \(\cT_r\) to be connected, without loss of generality (the response \eqref{eq:clean_general_block_response} of a
disconnected \(\cT_r\) separates into independent responses on its connected
components, which can be taken as separate blocks).
  \item[\textbf{(iii)}] \textbf{Carriers:} We use only the agents whose conjugates enter the response:  take \(\cC_r:=\cV_r\). Connectivity of \(\cT_r\) makes \(\cG[\cV_r]\) connected,
so the carrier conditions~\eqref{eq:clean_carrier_conditions} hold, and no
additional relay agents are needed.\vspace{-0.2cm}
 \end{itemize}

The three design choices are summarized by\vspace{-0.2cm}
\begin{equation}
  \mathsf A_{:,\cJ_r}=\mathsf B_{:,\cE_r},
  \qquad
  \bigsqcup_{r=1}^p\cE_r=\cE,
  \qquad
  \cC_r=\cV_r,\quad r\in[p].
  \label{eq:clean_unit_alignment}\vspace{-0.2cm}
\end{equation}

 After ordering the columns, \(\mathsf A=\mathsf B\), with \(\ker(\mathsf B^\top)=\operatorname{span}\{\mathbf 1\}\) (under Assumption~\ref{ass:clean_network}); hence   $\mathsf A$ is faithful.    Since \(\{\cE_r\}_{r=1}^p\) is a block cover
and \eqref{eq:clean_carrier_conditions} holds with \(\cC_r=\cV_r\),
\eqref{eq:clean_unit_alignment} defines a graph decomposition over \(\cG\)
(Definition~\ref{def:clean_graph_decomposition}): the block factor graph \(\cF_r\) is the node-edge incidence graph of
\(\cT_r\), that is, the bipartite graph on \(\cV_r\sqcup\cE_r\) in which \(i\)
and \(e\) are adjacent exactly when \(i\in\cI_e\); and
\eqref{eq:clean_realizability} holds with \(\cG_r=\cT_r\) and \(M_r=I\). 
The partition remains a design choice: finer partitions
give smaller block subproblems but may need more rounds to resolve the
coupling between blocks. The convergence analysis in Sec.~\ref{sec:clean_convergence}
  will quantify the effect of the partition on the rate and guide
  this computation-communication tradeoff.

The remainder of this section derives the algorithm GATE:
 the structure of the blocks is specified, and the block
fixed-point conditions of Proposition~\ref{prop:clean_generic_block_fixed_point}
are rewritten as relations for individual edge variables and local messages.
Their relaxed parallel iteration yields GATE.\vspace{-0.3cm}

\subsection{Tree-block responses and message equations}\vspace{-0.2cm}
\label{sec:clean_mp_jacobi}
\label{subsec:clean_logical_dmp}

Fix a block $r\in [p]$ and the exterior coordinates \(Y_{\overline{\cE}_r}\), where \(\overline{\cE}_r:=\cE\setminus\cE_r\); under
  \eqref{eq:clean_unit_alignment},
  the conditional response becomes\vspace{-0.2cm}
  \begin{equation}
   \BR_r(Y_{\overline{\cE}_r})
   =\argmin_{W\in\R^{|\cE_r|\times d}}
   \sum_{i\in\cV_r}
   f_i^*\bigl(\mathsf B_{i,\cE_r}W+
         \mathsf B_{i,\overline{\cE}_r}Y_{\overline{\cE}_r}\bigr).
   \label{eq:gate_block_response}\vspace{-0.2cm}
  \end{equation}
The conjugates in~\eqref{eq:gate_block_response} are held by different
agents.  We seek subgraphs  \(\cT_r\) that enable the computation of  \(\BR_r\) through exchanges between neighboring agents only.   Fixing the coordinate \(Y_e\) separates the remaining minimization in~\eqref{eq:gate_block_response}
into two independent subproblems {\it whenever removing \(e\) disconnects
\(\cT_r\)}. {\bf Trees} permit this separation at {\it every} edge. 
 They also eliminate coordinate redundancy: minimizing over the
  \(|\cV_r|-1\) edge variables of a spanning tree, while fixing those
  on the other edges of \(\cT_r\), gives the same optimal value as
  minimizing over the entire block.  This is established in the following proposition.
 \begin{proposition}
  \label{prop:connected}  Let \(\cT_r\) be connected and let \(\cE'_r\) be the edge
  set of any spanning tree of \(\cT_r\). Then
  $ \Range(\mathsf B_{:,\cE'_r})
   =\Range(\mathsf B_{:,\cE_r})$ and  $
   \ker(\mathsf B_{:,\cE'_r})=\{0\}.$ Furthermore, under Assumption~\ref{ass:clean_problem_class}, the response
  \(\BR_r(Y_{\overline{\cE}_r})\) is a singleton for every fixed
  \(Y_{\overline{\cE}_r}\) if and only if \(\cT_r\) is a tree.
  \end{proposition}
   We therefore choose each \(\cT_r\) to be a tree.
 \begin{assumption}
  \label{ass:clean_tree_delay}For every block $r\in [p]$, the associated subgraph \(\cT_r\) 
   is a tree.
  \end{assumption}
We  now   express the block-optimality conditions
  \eqref{eq:clean_generic_block_fixed_point} through
  quantities
  that the two endpoints of each edge can form and
  exchange. To distinguish edges inside and outside the block,
  define, for \(i\in\cV_r\),  \vspace{-0.1cm}
  \[\begin{gathered}
   \cA_i^\circ:=\cA_i\cap\cE_r,\qquad
   \cA_i^\partial:=\cA_i\cap\overline{\cE}_r,\\
   \cV_r^\circ:=\{i\in\cV_r:\cA_i^\partial=\emptyset\},
   \qquad
   \cV_r^\partial:=\cV_r\setminus\cV_r^\circ.
  \end{gathered}
  \] Relative to \(\cT_r\), agents in \(\cV_r^\circ\) are
  \emph{interior}: all their incident edges belong to
  \(\cE_r\).
  Agents in \(\cV_r^\partial\) are \emph{boundary}: they
  also
  have incident edges outside \(\cE_r\).
  In~\eqref{eq:gate_block_response}, only the conjugate
  terms
  of boundary agents depend on the exterior coordinates;
  these terms may also depend on several block variables.

\vspace{-0.3cm}

 \paragraph{\bf Splitting at the edge.}
For each block \(r\), uniqueness of the conditional response
(Proposition~\ref{prop:connected}) gives the following equivalent
conditions:
\vspace{-0.1cm}
\begin{equation}
\begin{aligned}
 &Y_{\cE_r}^\star\in\BR_r(Y_{\overline{\cE}_r}^\star)\\
 &\quad\Longleftrightarrow\quad
 \ Y_e^\star\in\argmin_{y_e\in\R^{1\times d}}\Bigg\{
 \min_{\substack{Y_{\cE_r}\\Y_e=y_e}}
 \sum_{i\in\cV_r}
 f_i^*\bigl(\mathsf B_{i,\cE_r}Y_{\cE_r}+
       \mathsf B_{i,\overline{\cE}_r}Y_{\overline{\cE}_r}^\star\bigr)
 \Bigg\} \quad\forall e\in\cE_r.
\end{aligned}
 \label{eq:clean_conditional_value}
\end{equation}
By Proposition~\ref{prop:clean_generic_block_fixed_point},
\(Y^\star\) solves~\eqref{eq:clean_dual} if and only if these
conditions hold for every block.
 To decompose the inner minimization
  in~\eqref{eq:clean_conditional_value}, fix \(Y_e=y_e\).
  Removing \(e\) from \(\cT_r\) separates this
  minimization
  into two decoupled subproblems, one on each resulting
  connected subgraph.
  We solve them recursively from the leaves
  toward \(e\).

 For each endpoint \(i\in\cI_e\) of \(e\), let
\(\cV_e(i)\) and \(\cE_e(i)\) denote the vertex and edge sets
of the resulting connected subgraph containing \(i\).
 Figure~\ref{fig:messages_tree_block}(a) illustrates the two
resulting subgraphs.
Every block edge other than \(e\)
lies in one of these two subgraphs.
Thus, with \(Y_e=y_e\), no remaining block
variable enters conjugates from both subgraphs, and
\begin{equation}
\begin{aligned}
 &\min_{\substack{Y_{\cE_r}\in\R^{|\cE_r|\times d}\\Y_e=y_e}}
 \sum_{k\in\cV_r}f_k^*\bigl(\mathsf B_{k,\cE_r}Y_{\cE_r}
       +\mathsf B_{k,\overline{\cE}_r}Y^\star_{\overline{\cE}_r}\bigr)\\
 &\quad=\sum_{i\in\cI_e}\underbrace{
 \min_{Y_{\cE_e(i)}\in\R^{|\cE_e(i)|\times d}}\sum_{k\in\cV_e(i)}
 f_k^*\bigl(\mathsf B_{k,\cE_e(i)}Y_{\cE_e(i)}
       +\mathsf B_{k,e}y_e
       +\mathsf B_{k,\overline{\cE}_r}Y^\star_{\overline{\cE}_r}\bigr)
 }_{\displaystyle =:\,\mu_{i\to e}(y_e)}.
\end{aligned}
\label{eq:clean_conditional_split}
\end{equation}
 The sum over \(\cI_e\)
in~\eqref{eq:clean_conditional_split} reads
\(\mu_{i\to e}(y_e)+\mu_{i'\to e}(y_e)\), where $i$ and $i'$ are the two endpoints of $e$; \(\mu_{i\to e}(y_e)\) [resp. \(\mu_{i'\to e}(y_e)\)] represents the   cost-to-go for the subgraph $(\cV_e(i),\cE_e(i))$ [resp. $(\cV_e(i'),\cE_e(i'))$];    together, they determine \(Y_e^\star\)
  through~\eqref{eq:clean_conditional_value}.
We call them \emph{messages}, as this information is exchanged
between the endpoints of \(e\); the arrows in
Figure~\ref{fig:messages_tree_block}(a) indicate the two messages.
Their dependence on \(Y^\star_{\overline{\cE}_r}\) is suppressed
for notational convenience. 
  \vspace{-0.2cm}
\begin{figure}[t!]
\centering
\includegraphics[width=0.98\linewidth]{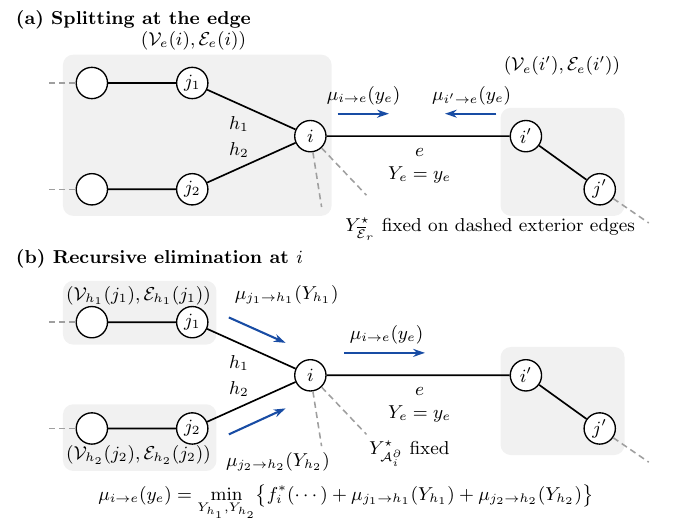}
\caption{Messages on the same tree block: (a) fixing
\(Y_e=y_e\) separates the minimization into two subgraphs;
(b) eliminating \(Y_{h_1},Y_{h_2}\) combines the two incoming
messages with \(f_i^*\), while the dashed exterior-edge coordinates
remain fixed at \(Y^\star_{\overline{\cE}_r}\).}
\label{fig:messages_tree_block}
\end{figure}

\paragraph{\bf Computing \(\mu_{i\to e}\) via recursive elimination along the subtrees.}
  We evaluate the minimization defining \(\mu_{i\to e}
  (y_e)\)
  by successively eliminating the edge variables \(Y_{\cE_e(i)}\) of the subgraph
  \((\cV_e(i),\cE_e(i))\), proceeding from its leaves
  toward \(e\) (through the endpoint
  $i$), while  \(Y_e=y_e\) and \(Y^\star_{\overline{\cE}_r}\)
  remain fixed.

\textit{(i) Leaves.} If \(\cV_e(i)=\{i\}\), then
\(\cE_e(i)=\emptyset\), and
\begin{equation}
 \mu_{i\to e}(y_e)
 =f_i^*\bigl(\mathsf B_{i,e}y_e+
       \mathsf B_{i,\overline{\cE}_r}Y^\star_{\overline{\cE}_r}\bigr).
 \label{eq:gate_leaf_message}
\end{equation}
The message is finite-valued and strongly convex by
Assumption~\ref{ass:clean_problem_class}, since
\(\mathsf B_{i,e}=\pm1\).

\textit{(ii) Remaining agents.}
 Otherwise, \(|\cV_e(i)|>1\). Removing \(i\) from
  \((\cV_e(i),\cE_e(i))\) leaves
  \(|\cA_i^\circ|-1\) connected subgraphs, one for each
  edge
  \(h\in\cA_i^\circ\setminus\{e\}\).
  For each such \(h\), let \(j\) be its endpoint other
  than
  \(i\), so that \(\{j\}=\cI_h\setminus\{i\}\).
  The resulting subgraph containing \(j\) is
  \((\cV_h(j),\cE_h(j))\): its agents \(k\) are exactly
  those
  whose unique path from \(i\) to \(k\) in \(\cT_r\)
  begins
  with \(h\). As \(h\) varies over \(\cA_i^\circ\setminus\{e\}\),
   the sets
  \(\cV_h(j)\)
  are pairwise disjoint and, together
  with \(\{i\}\), form a partition of \(\cV_e(i)\).
  The corresponding edge sets \(\{h\}\cup\cE_h(j)\)
  partition \(\cE_e(i)\): any additional edge would join
  two distinct subgraphs \((\cV_h(j),\cE_h(j))\) and
  create a cycle in \(\cT_r\). Hence,
\begin{equation}
\begin{aligned}
 \cV_e(i)&=\{i\}\sqcup
  \bigsqcup_{\substack{h\in\cA_i^\circ\setminus\{e\}\\j\in\cI_h\setminus\{i\}}}
             \cV_h(j),\\
 \cE_e(i)&=
  \bigsqcup_{\substack{h\in\cA_i^\circ\setminus\{e\}\\j\in\cI_h\setminus\{i\}}}
             \bigl(\{h\}\sqcup\cE_h(j)\bigr).
\end{aligned}
\label{eq:gate_component_partition}
\end{equation}

 We proceed by induction on \(|\cV_e(i)|\), starting
  from
  the leaf case~\eqref{eq:gate_leaf_message}.
  For each \(h\in\cA_i^\circ\setminus\{e\}\), let
  \(\{j\}=\cI_h\setminus\{i\}\).
  Since \(|\cV_h(j)|<|\cV_e(i)|\), the induction
  hypothesis
  ensures that \(\mu_{j\to h}\) is finite-valued and
  strongly
  convex, and its defining minimum is uniquely attained
  for every fixed \(Y_h\).

   In the minimization defining \(\mu_{i\to e}(y_e)\)
  in~\eqref{eq:clean_conditional_split}, first hold
  \(Y_{\cA_i^\circ\setminus\{e\}}\) fixed.
  By~\eqref{eq:gate_component_partition}, the sum of
  conjugates
  on each \(\cV_h(j)\) can be minimized independently
  over
  \(Y_{\cE_h(j)}\), giving \(\mu_{j\to h}(Y_h)\), since every block edge of an agent of \(\cV_h(j)\) lies in \(\{h\}\cup\cE_h(j)\).
  Minimizing the sum of these messages and agent \(i\)'s
  conjugate over \(Y_{\cA_i^\circ\setminus\{e\}}\), while
  \(Y_e=y_e\) and \(Y^\star_{\overline{\cE}_r}\) remain
  fixed,
  gives\vspace{-0.2cm}
\begin{equation}
\begin{aligned}
 \mu_{i\to e}(y_e)=
 \min_{Y_{\cA_i^\circ\setminus\{e\}}}\Bigg\{
 &f_i^*\bigl(\mathsf B_{i,e}y_e+
      \mathsf B_{i,\cA_i^\circ\setminus\{e\}}Y_{\cA_i^\circ\setminus\{e\}}+
      \mathsf B_{i,\overline{\cE}_r}Y^\star_{\overline{\cE}_r}\bigr)\\
 &+\sum_{h\in\cA_i^\circ\setminus\{e\}}\ \sum_{j\in\cI_h\setminus\{i\}}
       \mu_{j\to h}(Y_h)
 \Bigg\},\quad i\in\cV_r,\ e\in\cA_i^\circ.
\end{aligned}
\label{eq:clean_fixed_factor_to_var}
\end{equation}
 Figure~\ref{fig:messages_tree_block}(b) illustrates this
recursion for two edges incident to \(i\) other than \(e\).

  The objective in this recursion is finite-valued and
  jointly
  strongly convex in
  \((y_e,Y_{\cA_i^\circ\setminus\{e\}})\).
  For each \(y_e\), its minimum is therefore uniquely
  attained,
  and partial minimization yields a finite-valued,
  strongly
  convex \(\mu_{i\to e}\).
  Once the minimizing values \(Y_h\),
  \(h\in\cA_i^\circ\setminus\{e\}\), are fixed, the
  induction
  hypothesis uniquely determines the minimizing
  \(Y_{\cE_h(j)}\), where \(\{j\}=\cI_h\setminus\{i\}\).
  By~\eqref{eq:gate_component_partition}, these variables
  constitute the unique minimizer \(Y_{\cE_e(i)}\) in
  \eqref{eq:clean_conditional_split}, completing the
  induction.
  The same induction shows that any family satisfying
  \eqref{eq:clean_fixed_factor_to_var} coincides with the
  messages defined in~\eqref{eq:clean_conditional_split}.

  \begin{proposition}[Message fixed points]
  \label{prop:gate_message_fixed_point}
  Under Assumptions~\ref{ass:clean_problem_class},
  \ref{ass:clean_network},
  and~\ref{ass:clean_tree_delay},
  consider the
  decomposition~\eqref{eq:clean_unit_alignment}.
  For \(Y^\star\in\R^{|\cE|\times d}\), let
  \(\mu_{i\to e}\) be the messages defined in
  \eqref{eq:clean_conditional_split}, with
  \(Y^\star_{\overline{\cE}_r}\) fixed in each block.
  Then \(Y^\star\) solves~\eqref{eq:clean_dual} if and
  only if
  \begin{equation}
   \{Y_e^\star\}
   =
   \argmin_{y_e\in\R^{1\times d}}
   \sum_{i\in\cI_e}\mu_{i\to e}(y_e),
   \qquad e\in\cE_r,\quad r\in[p].
   \label{eq:clean_edge_fixed_point_decomposition}
  \end{equation}
  \end{proposition}

  The message
  recursion~\eqref{eq:clean_fixed_factor_to_var}
  and the edge conditions
  \eqref{eq:clean_edge_fixed_point_decomposition}
  form the fixed-point system whose relaxed parallel
  iteration defines GATE. This is formalized next.\vspace{-0.4cm}

\subsection{The node-owned GATE algorithm}
\label{subsec:clean_node_owned_dmp}\vspace{-0.2cm}
GATE iterates the edge conditions~\eqref{eq:clean_edge_fixed_point_decomposition} and the recursion~\eqref{eq:clean_fixed_factor_to_var}, with relaxed edge updates, as summarized in
  Algorithm~\ref{alg:gate_uniform_proposed}.
  To avoid an edge-level coordinator, an edge-node assignment is introduced,
  \(\mathfrak a:\cE\to\cV\) with \(\mathfrak
  a(e)\in\cI_e\).   Define
  \(\cA_i^{\mathfrak a}:=\{e\in\cA_i:\mathfrak
  a(e)=i\}\).
  These sets are pairwise disjoint and cover \(\cE\).
  Node \(i\) updates \(Y_e\) for \(e\in\cA_i^{\mathfrak
  a}\),
  while both endpoint-nodes retain a copy.
  For each edge \(e\), \(r\) denotes its block index, and
  \(\cA_i^\circ\), \(\cA_i^\partial\) are understood
  relative to that block.

GATE is fully decentralized: each agent
  computes locally
  and communicates only with its neighbors, without a
  central coordinator. This is discussed next.

 \textbf{Initialization.}
  For each edge \(e\), its owner \(\mathfrak a(e)\)
  chooses \(Y_e^0\)
  arbitrarily and sends it to the other endpoint.
  Each node \(i\) then selects a finite-valued strongly
  convex message
  \(\mu_{i\to e}^0\) for every \(e\in\cA_i\) and sends it
  to the
  other endpoint of \(e\). For instance, one local choice is
  \(
  \mu_{i\to e}^0(y_e)
  =
  f_i^*\!\big(
  \mathsf B_{i,e}y_e+
  \mathsf B_{i,\cA_i\setminus\{e\}}
  Y_{\cA_i\setminus\{e\}}^0
  \big).
  \)

\begin{algorithm}[t!]
{\fontsize{6.2}{7.0}\selectfont
\DontPrintSemicolon
\SetKwFor{ForEach}{foreach}{}{end foreach}
\SetKwFor{NodeActions}{Each node}{executes the following steps, in parallel with the other nodes:}{}
\SetAlgoRefName{1}
\caption{GATE: Graph-Tearing message passing}
\label{alg:gate_uniform_proposed}\medskip
\KwIn{Edge assignment \(\mathfrak a\) with ownership sets
\(\{\cA_i^{\mathfrak a}\}_{i\in\cV}\); edge partition
\(\{\cE_r\}_{r=1}^p\); initial edge variables \(Y^0\);
finite-valued strongly convex initial messages
\(\mu_{i\to e}^0\), \(e\in\cE\), \(i\in\cI_e\);
relaxation parameters \(\tau_r^\nu\in[0,1]\); \(N\geq1\).}

\For{\(\nu=0,1,\ldots,N-1\)}{
\NodeActions{\(i\in\cV\)}{
\BlankLine
\texttt{(1) Assigned edge-variable update.}\;
\ForEach{\(r\in[p]\) and \(e\in\cA_i^{\mathfrak a}\cap\cE_r\), compute}{
\begin{subequations}\label{eq:clean_node_owned_updates}
\begin{align}\vspace{-0.4cm}
\hat Y_e^{\nu+1}
&\in\argmin_{y_e\in\R^{1\times d}}
\left\{\sum_{j\in\cI_e}\mu_{j\to e}^{\nu}(y_e)\right\},
\label{eq:clean_dmp_updates_a}\\
Y_e^{\nu+1}
&=Y_e^\nu+\tau_r^\nu\bigl(\hat Y_e^{\nu+1}-Y_e^\nu\bigr).
\label{eq:clean_dmp_updates_b}
\end{align}
\end{subequations}
Send \(Y_e^{\nu+1}\) to the other endpoint of \(e\).\;
}
\BlankLine
\texttt{(2) Node-to-edge message update:}
\ForEach{\(r\in[p]\) and \(e\in\cA_i\cap\cE_r\), compute}{
\begin{equation}
\begin{aligned}
\mu_{i\to e}^{\nu+1}(y_e)
&=\min_{Y_{\cA_i^\circ\setminus\{e\}}}
\Bigg\{
f_i^*\!\left(\mathsf B_{i,e}y_e
+\mathsf B_{i,\cA_i^\circ\setminus\{e\}}
Y_{\cA_i^\circ\setminus\{e\}}
+\mathsf B_{i,\cA_i^\partial}
Y_{\cA_i^\partial}^{\nu+1}\right)\\
&\qquad+\sum_{h\in\cA_i^\circ\setminus\{e\}}
\ \sum_{j\in\cI_h\setminus\{i\}}
\mu_{j\to h}^{\nu}(Y_h)
\Bigg\}.
\end{aligned}
\label{eq:clean_factor_to_var}
\end{equation}
Send \(\mu_{i\to e}^{\nu+1}\) to the other endpoint of \(e\).\;
}
}
}
\BlankLine
\textbf{Primal recovery.}
Each node \(i\in\cV\) computes
\(x_i^N:=\nabla f_i^*((\mathsf BY^N)_i)\).
}
\end{algorithm}

\textbf{Local computations and communications.}
 \texttt{Step (1):}
Each node updates only its assigned edges, using the two endpoint
messages at round \(\nu\). Relaxation damps the edge update, which is
then shared with the other endpoint.
 \texttt{Step (2):} Each node updates a message on every incident edge, whether or not it
owns that edge.  The local minimization uses its conjugate and the
incoming messages \(\mu^\nu\), with incident exterior coordinates
fixed at the newly available \(Y^{\nu+1}\). Its minimizing variables
are local to this computation and do not overwrite the stored edge
iterates. The outgoing messages are sent to the other endpoints for
round \(\nu+1\).

Nodes work in parallel within each step, once the required data have
arrived. Each round advances every message once and sends one \(d\)-vector and two
functions per edge. Primal recovery uses only the incident edge coordinates. 
 
 \begin{remark}[Relation to \cite{ding2025hypergraphmp}] 
 MP-Jacobi~\cite{ding2025hypergraphmp} develops message-passing
  algorithms for nonlinear programs with a prescribed interaction
  (hyper)graph. Although its hypergraph model can represent the
  consensus dual for a fixed \(\mathsf A\), MP-Jacobi cannot be
  applied directly to obtain GATE or establish its convergence.
  GATE's boundary-message updates jointly minimize over other
  in-block variables that MP-Jacobi holds fixed in its
  inter-cluster factors. Moreover, the    convergence
  theory in \cite{ding2025hypergraphmp} requires strong convexity in the optimization variables,
  which the incidence dual lacks on cyclic graphs. 
  Beyond these algorithmic and analytical differences, a major departure from \cite{ding2025hypergraphmp} is the development here of   a general
  graph-decomposition framework that jointly designs the
  consensus representation, dual-coordinate blocks, and
  connected sets of agents hosting their computation;
  GATE is one instance. 
\vspace{-0.5cm} 
\end{remark}
\section{Convergence Analysis}
\label{sec:clean_convergence}\vspace{-0.2cm}

In this section, we establish  convergence of Algorithm~\ref{alg:gate_uniform_proposed}.  We begin with the quantities the analysis is built on.\vspace{-0.3cm}
\subsection{Image variables and block error bound}\vspace{-0.1cm}
For the incidence representation \(\mathsf A=\mathsf B\), write \(g:=g_{\mathsf B}\) and define
\[
    G(U):=\sum_{i\in\cV}f_i^*(U_i),
    \qquad U\in\R^{m\times d}.
\]
 Under Assumption~\ref{ass:clean_problem_class}, \(G\) is strongly convex. On loopy graphs, however, the incidence map is not injective and  \(g(Y)=G(\mathsf BY)\) is constant along \(\ker\mathsf B\): it is not strongly convex in the edge variables $Y$. The analysis therefore runs on the image variable \(Z:=\mathsf BY\), which ranges over the subspace
\[
    \cS:=\{\mathsf BY:\ Y\in\R^{|\cE|\times d}\}\subseteq\R^{|\cV|\times d}.
\] Changing \(Y_{\cE_r}\) alone moves \(Z\) along the affine subspace \(Z+\cU_r\), where
\[
    \cU_r:=\{\mathsf B_{\cV,\cE_r}H:\ H\in\R^{|\cE_r|\times d}\}\subseteq \cS,\qquad r\in[p].
\]
For \(Z=\mathsf BY\in\cS\), the value of the block minimization~\eqref{eq:gate_block_response} equals the minimum of \(G\) over the slice \(Z+\cU_r\):
\begin{equation}
    \min_{W_r\in\R^{|\cE_r|\times d}} g(W_r,Y_{\overline{\cE}_r})
    \;=\;
    \min_{U\in Z+\cU_r}G(U)
    \;=:\;G(O_r(Z)),
\label{eq:clean_block_image_minimizer}
\end{equation}
where \(O_r(Z)\) is the unique minimizer, well defined because \(G\) is strongly convex and \(Z+\cU_r\) is closed. Let
\begin{equation}\label{eq:def_Delta_r}
    \Delta_r(Z):=G(Z)-G(O_r(Z))\ge 0
\end{equation}be the progress of one exact minimization of block \(r\) from \(Z\). 
 Let \(G^\star:=\min_{U\in\cS}G(U)\), attained because \(G\) is strongly convex and \(\cS\) is closed. The analysis rests on the following condition, which bounds the optimality gap \(G(Z)-G^\star\) by the progress of all blocks.


\begin{assumption}[block error bound]
\label{ass:clean_beb}
There exists a partition-dependent constant \(\kappa_{\cP}<\infty\) such that \vspace{-0.2cm}
\begin{equation}
    G(Z)-G^\star\le \kappa_{\cP}\sum_{r=1}^p\Delta_r(Z),
    \quad Z\in\cS.
    \label{eq:clean_beb_image}
\end{equation}
\end{assumption}

\begin{remark}
\label{rem:clean_beb_motivation}
The sum \(\sum_{r=1}^p\Delta_r(Z)\) collects the decreases of \(G\) obtained by minimizing each block separately from the same \(Z\), with its exterior coordinates fixed. These separate minimizations need not add up to the global gap, because the block subproblems interact through the conjugate terms of their shared agents. Condition~\eqref{eq:clean_beb_image} controls how much suboptimality can remain relative to the progress they offer, and its smallest constant is the largest value of this ratio over \(\cS\). Thus \(\kappa_{\cP}\) quantifies the coupling among the blocks, accounting jointly for the partition and the functions. With the single-block partition \(\cP=\{\cE\}\) on a tree, \(\cU_1=\cS\), so \(\Delta_1(Z)=G(Z)-G^\star\) and \(\kappa_{\cP}=1\), independently of the condition number of $F$  and graph spectrum.

The coupling measured by \(\kappa_{\cP}\) admits a \emph{block-preconditioning} interpretation, which becomes explicit for quadratic local functions. With the edge vectors stacked, the dual Hessian is
\[
H=[H_{rs}]_{r,s\in[p]}
=(\mathsf B^\top\otimes I_d)
\operatorname{blkdiag}\!\left((\nabla^2 f_1)^{-1},\ldots,(\nabla^2 f_m)^{-1}\right)
(\mathsf B\otimes I_d)\succeq0.
\]
Set \(M_{\cP}:=\operatorname{blkdiag}(H_{11},\ldots,H_{pp})\). Each tree incidence block has full column rank, so \(H_{rr}\succ0\). Writing \(h\) for the vectorized difference between \(Y\) and a dual minimizer, exact block minimization gives \(\sum_{r=1}^p\Delta_r(Z)=\frac12h^\top HM_{\cP}^{-1}Hh\), with \(Z=\mathsf BY\). Thus the smallest admissible block error-bound constant is
\[
\kappa_{\cP}
=\sup_{h^\top Hh>0}\frac{h^\top Hh}{h^\top HM_{\cP}^{-1}Hh}
=\frac{1}{\lambda_{\min}^{+}\!\left(M_{\cP}^{-1/2}HM_{\cP}^{-1/2}\right)},
\]
where \(\lambda_{\min}^{+}\) denotes the smallest positive eigenvalue. The preconditioned Hessian has identity diagonal blocks; its off-diagonal blocks \(H_{rr}^{-1/2}H_{rs}H_{ss}^{-1/2}\) capture coupling normalized by the block curvatures.

Generic bounds based only on global curvature and spectral estimates need not capture this partition dependence; exploiting it yields sharper bounds in the regimes studied in Sec.~\ref{subsec:clean_rate_examples}.
\end{remark}

We prove convergence of Algorithm~\ref{alg:gate_uniform_proposed}
  in two steps. {\bf Step~1 }establishes an exact delayed block-Jacobi
  representation of its edge updates, which is a technical contribution of independent interest; {\bf Step~2} controls the delays
  to prove linear convergence.\vspace{-0.2cm}

\subsection{Step 1: GATE as a delayed block-Jacobi scheme}
\label{subsec:clean_delayed_form}\vspace{-0.2cm}

GATE interleaves message and edge updates without waiting for the message recursion to run through a whole tree (no block subproblem is ever solved to completion). Nevertheless, after a finite transient, each proposed edge update coincides with the corresponding coordinate of an exact block minimizer, evaluated at past values of the exterior coordinates, with delays that depend on the updated edge.

To establish it, we expand the messages used in \texttt{step~1} of Algorithm~\ref{alg:gate_uniform_proposed} through the recursion in \texttt{step~2}.  Each substitution goes back one round. If the expansion reaches a leaf of \(\cT_r\) at distance \(k\) from the nearer endpoint of \(e\), its message has index \(\nu-k\). For \(\nu-k\ge1\), substituting the recursion replaces this message by the leaf's conjugate term; for \(\nu-k=0\), the expansion stops at the arbitrary initial function \(\mu^0\). Since \(k\le\operatorname{diam}(\cT_r)-1\), \(\nu\ge\operatorname{diam}(\cT_r)\) ensures that no initial message remains. The sum \(\sum_{j\in\cI_e}\mu^\nu_{j\to e}(y_e)\) minimized in \texttt{step~1} then equals the minimum of the block objective over \(Y_{\cE_r\setminus\{e\}}\), with \(Y_e=y_e\). In each boundary agent's conjugate term, the exterior coordinates retain the values used when that agent formed its message. The delay associated with a boundary agent \(i\) is its distance from the nearer endpoint of the updated edge \(e\):
\begin{equation}
    \delta_{e,i}
    :=\min_{j\in\cI_e}\dist_{\cT_r}(i,j)
    =\frac{\dist_{\cF_r}(e,i)-1}{2},
    \qquad e\in\cE_r,\quad i\in\cV_r^\partial .
    \label{eq:clean_dual_delay}
\end{equation}
In particular, \(0\le\delta_{e,i}\le\operatorname{diam}(\cT_r)-1\).

The following proposition makes this representation precise.

\begin{proposition}
\label{prop:clean_delayed_bilevel}
Suppose Assumptions~\ref{ass:clean_problem_class} and~\ref{ass:clean_tree_delay} hold.  Then, for every \(r\in[p]\), \(e\in\cE_r\) and \(\nu\ge\operatorname{diam}(\cT_r)\), the updates of Algorithm~\ref{alg:gate_uniform_proposed} read
\begin{subequations}
\label{eq:clean_delayed_block_update}
\begin{align}
Y_e^{\nu+1}
&=Y_e^\nu+\tau_r^\nu(\hat Y_e^{\nu+1}-Y_e^\nu),
\label{eq:clean_delayed_block_update_a}
\\
\hat Y_e^{\nu+1}
&=
\argmin_{Y_e}\ \min_{Y_{\cE_r\setminus\{e\}}}
\left\{
\sum_{i\in\cV_r}
f_i^*\!\left(
\mathsf B_{i,\cA_i^\circ}Y_{\cA_i^\circ}
	+
	\sum_{h\in\cA_i^\partial}\mathsf B_{i,h}Y_h^{\nu-\delta_{e,i}}
	\right)
	\right\},
\label{eq:clean_expanded_delayed_oracle}
\end{align}
\end{subequations}
where the minimizer in~\eqref{eq:clean_expanded_delayed_oracle} is unique.
\end{proposition}
\begin{proof}
See Appendix~\ref{app:clean_delayed_bilevel_proof}. \qed
\end{proof}\vspace{-0.5cm}

\subsection{Step 2: Convergence of the delayed block-Jacobi scheme}
\label{subsec:clean_delayed_block_jacobi_convergence}

We now establish convergence for the delayed block-Jacobi scheme in~\eqref{eq:clean_delayed_block_update}. The analysis runs in the image variable \(Z=\mathsf BY\), where \(G\) is strongly convex, and compares, for each block,   the delayed edge responses that the algorithm implements with the exact block response at the current exterior coordinates. Three ingredients give the rate: \textbf{(i)} the descent of \(G\) along the iterates (Lemma~\ref{lem:clean_one_step}), \textbf{(ii)} the block error bound (Lemma~\ref{lem:clean_main_recursion}), and \textbf{(iii)}  a perturbation analysis of the block minimizers, bounding the mismatch caused by the delays  (Lemma~\ref{lem:clean_tree_mismatch}). 

 Write \(Z^\nu:=\mathsf BY^\nu\), with \(Z_i^\nu\) its \(i\)-th row, and   define, for \(r\in[p]\),\vspace{-0.1cm}\begin{equation}
\begin{aligned}
    Z_r^\nu
    &:=
    \mathsf B_{\cV,\cE_r}Y_{\cE_r}^\nu,\\
    \hat Z_r^{\nu+1}
    &:=
    \mathsf B_{\cV,\cE_r}\hat Y_{\cE_r}^{\nu+1},\\
    \bar Z_r^{\nu+1}
    &:=
    O_r(Z^\nu)-Z^\nu+Z_r^\nu.
\end{aligned}
\label{eq:clean_block_images}
\end{equation}
These are the contributions of block \(r\) to the image of the current iterate, of the delayed edgewise minimizers~\eqref{eq:clean_expanded_delayed_oracle}, and of the exact minimization at the current exterior coordinates.
Applying \(\mathsf B\) to the relaxed edge updates \eqref{eq:clean_delayed_block_update_a} gives\begin{equation}
    Z^{\nu+1}
    =
    Z^\nu+\sum_{r=1}^p\tau_r^\nu
    \bigl(\hat Z_r^{\nu+1}-Z_r^\nu\bigr).
    \label{eq:Z_update}
\end{equation}

Blocks sharing an agent \(i\) all move the row \(Z_i\), each by its own step \(\tau_r^\nu\). The admissible range of \(\tau_r^\nu\) is therefore limited by the number of overlapping blocks at each agent, defined as\vspace{-0.2cm}
\begin{equation}
    \omega_{\cP}
    :=
    \max_{i\in\cV}
    \bigl|\{r\in[p]:i\in\cV_r\}\bigr|,
    \label{eq:clean_block_overlap}
\end{equation}
where \(1\le\omega_{\cP}\le\min\{p,\max_{i\in\cV}|\cA_i|\}\). This is formalized next.\smallskip

\smallskip
\noindent\textbf{(i) Descent of \(G\) along the iterates.} Each round decreases \(G\) by the damped progress of the blocks, up to the error made by using delayed instead of current exterior coordinates. Along \(\cU_r\) only the conjugates of the agents in \(\cV_r\) vary, so \(G\) is \((1/L_r)\)-strongly convex and \((1/\mu_r)\)-smooth there, with
\[
    \mu_r:=\min_{i\in\cV_r}\mu_{f_i},
    \qquad
    L_r:=\max_{i\in\cV_r}L_{f_i}.
\]

\begin{lemma}
\label{lem:clean_one_step}
 Under Assumption~\ref{ass:clean_problem_class}, let  \(0\le\tau_r^\nu\le1/\omega_{\cP}\) for every \(r\in[p]\). Then, 
\begin{equation}
G(Z^{\nu+1})
\le
G(Z^\nu)
-\sum_{r=1}^p\tau_r^\nu\Delta_r(Z^\nu)
+\sum_{r=1}^p\frac{\tau_r^\nu}{2\mu_r}
\|\hat Z_r^{\nu+1}-\bar Z_r^{\nu+1}\|^2,
\label{eq:clean_one_step}\vspace{0.2cm}
\end{equation}
with \(\bar Z_r^{\nu+1}\) as in~\eqref{eq:clean_block_images}.
\end{lemma}
\begin{proof}
Fix \(r\in[p]\);   by~\eqref{eq:clean_block_images}, \(Z^\nu-Z_r^\nu+\bar Z_r^{\nu+1}=O_r(Z^\nu)\) minimizes \(G\) over \(Z^\nu+\cU_r\), so
\[
\inner{
\nabla G(Z^\nu-Z_r^\nu+\bar Z_r^{\nu+1})^\top,H
}=0,
\quad H\in\cU_r .
\]
Since \(\hat Z_r^{\nu+1}-\bar Z_r^{\nu+1}\in\cU_r\), the \(1/\mu_r\)-smoothness of \(G\) along \(\cU_r\) and  \eqref{eq:def_Delta_r} give
\[
G(Z^\nu-Z_r^\nu+\hat Z_r^{\nu+1})
\le
G(Z^\nu-Z_r^\nu+\bar Z_r^{\nu+1})
+\frac{1}{2\mu_r}\|\hat Z_r^{\nu+1}-\bar Z_r^{\nu+1}\|^2 .
\]
Next, \((\hat Z_r^{\nu+1}-Z_r^\nu)_i=0\) for every \(i\notin\cV_r\); so (\ref{eq:Z_update}) reads
\begin{equation}
Z_i^{\nu+1}
=
\left(1-\sum_{r\in[p]:\,i\in\cV_r}\tau_r^\nu\right)Z_i^\nu
+
\sum_{r\in[p]:\,i\in\cV_r}\tau_r^\nu
\left(Z^\nu-Z_r^\nu+\hat Z_r^{\nu+1}\right)_i.
\label{eq:clean_z_relaxed}
\end{equation}
Since \(\tau_r^\nu\le1/\omega_{\cP}\), applying Jensen's inequality to~\eqref{eq:clean_z_relaxed} yields
\[
f_i^*(Z_i^{\nu+1})
\le
\left(1-\sum_{r\in[p]:\,i\in\cV_r}\tau_r^\nu\right)f_i^*(Z_i^\nu)
+
\sum_{r\in[p]:\,i\in\cV_r}\tau_r^\nu
f_i^*\!\left((Z^\nu-Z_r^\nu+\hat Z_r^{\nu+1})_i\right).
\]
For \(i\notin\cV_r\), \((Z^\nu-Z_r^\nu+\hat Z_r^{\nu+1})_i=Z_i^\nu\). Hence,
\[
\begin{aligned}
&G(Z^{\nu+1})-G(Z^\nu)=\sum_{i\in\cV}\left[f_i^*(Z_i^{\nu+1})-f_i^*(Z_i^\nu)\right]\\
\le&\sum_{i\in\cV}\sum_{r\in[p]:\,i\in\cV_r}\tau_r^\nu
\left[f_i^*\!\left((Z^\nu-Z_r^\nu+\hat Z_r^{\nu+1})_i\right)-f_i^*(Z_i^\nu)\right]\\
=&\sum_{r=1}^p\tau_r^\nu
\left[G(Z^\nu-Z_r^\nu+\hat Z_r^{\nu+1})-G(Z^\nu)\right].
\end{aligned}
\]
Combining this inequality with the definition of \(\Delta_r(Z^\nu)\) (cf.  \eqref{eq:def_Delta_r}) yields \eqref{eq:clean_one_step}. \qed
\end{proof}

  \noindent\textbf{(ii) Block error bound.} Assumption~\ref{ass:clean_beb} converts the block progress in~\eqref{eq:clean_one_step} into a contraction of the global gap. Only half of the progress is spent on it; the other half is kept as a negative term in the block increments \(Z_r^{\nu+1}-Z_r^\nu\), instrumental in controlling the adversarial positive term in the mismatch \(\hat Z_r^{\nu+1}-\bar Z_r^{\nu+1}\).

\begin{lemma}
\label{lem:clean_main_recursion}
Under Assumptions~\ref{ass:clean_problem_class} and~\ref{ass:clean_beb}, let \(0<\tau_r^\nu\le1/\omega_{\cP}\) for every \(r\in[p]\), and \(\underline\tau^\nu:=\min_{r\in[p]}\tau_r^\nu\). Then,
\begin{equation}
\begin{aligned}
G(Z^{\nu+1})-G^\star
\le\ &
\left(1-\frac{\underline\tau^\nu}{2\kappa_{\cP}}\right)
\bigl(G(Z^\nu)-G^\star\bigr)
-\sum_{r=1}^p\frac{1}{8L_r\tau_r^\nu}
\|Z_r^{\nu+1}-Z_r^\nu\|^2\\
&\quad
+\sum_{r=1}^p\tau_r^\nu\frac{2/\mu_r+1/L_r}{4}
\|\hat Z_r^{\nu+1}-\bar Z_r^{\nu+1}\|^2.
\end{aligned}
\label{eq:clean_main_recursion}
\end{equation}
\end{lemma}
\begin{proof}
Split the block progress   \(\sum_{r=1}^p\tau_r^\nu\Delta_r(Z^\nu)\) in~\eqref{eq:clean_one_step} into two halves.  For the first, Assumption~\ref{ass:clean_beb} and \(\tau_r^\nu\ge\underline\tau^\nu\) give \vspace{-0.2cm}
\[
\frac12\sum_{r=1}^p\tau_r^\nu\Delta_r(Z^\nu)
\ge
\frac{\underline\tau^\nu}{2\kappa_{\cP}}\bigl(G(Z^\nu)-G^\star\bigr).
\]
For the second, optimality of \(O_r(Z^\nu)\) over \(Z^\nu+\cU_r\) and \(1/L_r\)-strong convexity along \(\cU_r\) give
\[
\Delta_r(Z^\nu)
\ge
\frac{1}{2L_r}\|\bar Z_r^{\nu+1}-Z_r^\nu\|^2
\ge
\frac{1}{4L_r}\|\hat Z_r^{\nu+1}-Z_r^\nu\|^2
-\frac{1}{2L_r}\|\hat Z_r^{\nu+1}-\bar Z_r^{\nu+1}\|^2.
\]   Multiply the preceding lower bound by \(\tau_r^\nu/2\) and use
\(
    Z_r^{\nu+1}-Z_r^\nu
    =
    \tau_r^\nu(\hat Z_r^{\nu+1}-Z_r^\nu),
\)
due to~\eqref{eq:clean_delayed_block_update_a}. Substituting both estimates into Lemma~\ref{lem:clean_one_step} gives \eqref{eq:clean_main_recursion}. \qed
\end{proof}

\smallskip
\noindent\textbf{(iii) Perturbation analysis of the block minimizers.} It remains to bound
  \(\|\hat Z_r^{\nu+1}-\bar Z_r^{\nu+1}\|^2\)
   in Lemma~\ref{lem:clean_main_recursion}.
  The   contribution \(\bar Z_r^{\nu+1}\) uses the current values  \(Y_h^\nu\) of the
  exterior coordinates. By
  Proposition~\ref{prop:clean_delayed_bilevel}, each coordinate
  used to form \(\hat Z_r^{\nu+1}\) instead comes from a block
  minimization with edge-dependent delayed exterior coordinates  \(Y_h^{\nu-\delta_{e,i}}\).
  We first quantify the sensitivity of the block minimizer
  to the exterior parameters,  then bound the mismatch in
  terms of changes in the block contributions to \(Z\)
  over previous iterations.

\smallskip
\noindent\emph{Exterior parameters and delays.}
For block \(r\), let \(\partial_r\) collect the exterior
  edges incident to its agents, and let \(\cM_r\) collect
  the blocks containing these edges:
\[
    \partial_r:=\bigcup_{i\in\cV_r}\cA_i^\partial,
    \qquad
    \cM_r:=
    \{s\in[p]\setminus\{r\}:
    \cE_s\cap\partial_r\ne\emptyset\}. 
\]
The largest delay affecting the block is
\[
    D_r:=
    \begin{cases}
    \displaystyle
    \max_{e\in\cE_r,\,i\in\cV_r^\partial}
    \delta_{e,i},
    &\partial_r\ne\emptyset,\\[1mm]
    0,
    &\partial_r=\emptyset.
    \end{cases}
\]
When \(\partial_r\ne\emptyset\),
\[
    \tfrac12\operatorname{diam}(\cT_r)-1
    \le D_r
    \le\operatorname{diam}(\cT_r)-1.
\]
An exterior edge \(h\) appears in the block subproblem inside the conjugate \(f_i^*\) of each of its endpoints \(i\in\cV_r\), through the term \(\mathsf B_{i,h}Y_h\). Index these appearances by
\[
    \widehat\partial_r
    :=
    \{(i,h):i\in\cV_r^\partial,\ h\in\cA_i^\partial\}.
\]
An exterior edge incident to two agents of \(\cV_r\) therefore has two entries in \(\widehat\partial_r\), since its value can appear at different delays in their conjugate terms.

\smallskip
\noindent\emph{Sensitivity to the exterior parameters.} For each \(E\in\R^{|\widehat\partial_r|\times d}\),
  with rows \(E_{i,h}\) indexed by the boundary incidences
  \((i,h)\in\widehat\partial_r\), define the block solution
\begin{equation}
Y_r(E):=
\argmin_{Y_{\cE_r}\in\R^{|\cE_r|\times d}}
\sum_{i\in\cV_r}
f_i^*\!\left(
\mathsf B_{i,\cA_i^\circ}Y_{\cA_i^\circ}
+
\sum_{h\in\cA_i^\partial}\mathsf B_{i,h}E_{i,h}
\right).
\label{eq:clean_boundary_solution_mapping}
\end{equation}
 The minimizer exists and is unique.

 We next quantify how \(Y_r(E)\) and its contribution
  \(\mathsf B_{\cV_r,\cE_r}Y_r(E)\) to \(Z\) vary with
  the exterior parameters \(E\).
  We first express the latter in terms of \(E\).

Let
\[
  \R^d\ni\eta\mapsto  \chi_r(\eta)
    :=
    \left(\sum_{i\in\cV_r}\nabla f_i\right)^{-1}(\eta),
\]
which   is well defined on \(\R^d\) because \(\sum_{i\in\cV_r}\nabla f_i\) is strongly monotone and continuous. Write \(\eta(E):=\sum_{(j,h)\in\widehat\partial_r}\mathsf B_{j,h}E_{j,h}\) for the sum of the exterior contributions.
At the minimizer~\eqref{eq:clean_boundary_solution_mapping} all agents of the block recover the same primal point.  Lemma~\ref{lem:app_tree_oracle_jacobians} in Appendix~\ref{app:clean_tree_oracle} identifies it as \(\chi_r(\eta(E))\):
\[
    \nabla f_i^*\Bigl(\mathsf B_{i,\cE_r}Y_r(E)+\sum_{h\in\cA_i^\partial}\mathsf B_{i,h}E_{i,h}\Bigr)
    =\chi_r(\eta(E)),
    \qquad i\in\cV_r,
\]
which, using \(\nabla f_i^*=(\nabla f_i)^{-1}\), is equivalent to
\begin{equation}\label{eq:image_formula}
    \mathsf B_{i,\cE_r}Y_r(E)
    =
    \nabla f_i\bigl(\chi_r(\eta(E))\bigr)
    -\sum_{h\in\cA_i^\partial}\mathsf B_{i,h}E_{i,h},
    \qquad i\in\cV_r.
\end{equation}  For \(v\ne i\), a perturbation of \(E_{i,h}\) changes
  \(\mathsf B_{v,\cE_r}Y_r(E)\) only through
  \(\nabla f_v\circ\chi_r\).

We now use~\eqref{eq:image_formula} to describe the
  dependence of the rows \((Y_r(E))_e\), \(e\in\cE_r\),
  on the exterior parameters.  Removing \(e\in\cE_r\) from \(\cT_r\) gives two connected
  subgraphs. For \(i\in\cV_r\), let \(\cV_e(i)\) denote
  the vertex set of the subgraph containing \(i\) and set
  \[
      \overline{\cV}_e(i):=\cV_r\setminus\cV_e(i).
  \]  This  extends the notation of~\eqref{eq:clean_conditional_split} to agents that need not be endpoints of \(e\).
  Summing~\eqref{eq:image_formula} over
  \(v\in\overline{\cV}_e(i)\) gives the following identity,
  since every column of
  \(\mathsf B_{\overline{\cV}_e(i),\cE_r}\) sums to zero
  except that of \(e\), whose sum is \(\pm1\):
  \begin{equation} \label{eq:clean_edge_parameter_formula}
      (Y_r(E))_e
      =
      \pm\sum_{v\in\overline{\cV}_e(i)}
      \left[
          \nabla f_v\bigl(\chi_r(\eta(E))\bigr)
          -
          \sum_{h'\in\cA_v^\partial}
          \mathsf B_{v,h'}E_{v,h'}
      \right],
\end{equation}
  with the sign determined by the orientation of \(e\).
  A perturbation of \(E_{i,h}\) leaves
  \(\sum_{h'\in\cA_v^\partial}\mathsf B_{v,h'}E_{v,h'}\)
  unchanged for every \(v\in\overline{\cV}_e(i)\), since
  \(v\ne i\). It therefore affects \((Y_r(E))_e\) only
  through
  \(\sum_{v\in\overline{\cV}_e(i)}
  \nabla f_v(\chi_r(\eta(E)))\).

  Equations~\eqref{eq:image_formula}
  and~\eqref{eq:clean_edge_parameter_formula} involve
  the maps \(\nabla f_v\circ\chi_r\) and their sums. To treat them in a unified fashion, define, for nonempty \(S\subseteq\cV_r\),
  \[
      \cL_{r,S}(\eta)
      :=
      \sum_{v\in S}\nabla f_v\bigl(\chi_r(\eta)\bigr),
      \qquad \emptyset\ne S\subseteq\cV_r.
  \]
   The choice \(S=\{v\}\) recovers the individual map
  \(\nabla f_v\circ\chi_r\), whereas
  \(S=\overline{\cV}_e(i)\) gives the sum associated
  with \((Y_r(E))_e\).
 Given these two relevant  subsets, we   define
  \[
      \cR_r
      :=
      \bigl\{\{v\}:
          (i,h)\in\widehat\partial_r,\
          v\in\cV_r\setminus\{i\}\bigr\}
      \cup
      \bigl\{\overline{\cV}_e(i):
          e\in\cE_r,\
          (i,h)\in\widehat\partial_r\bigr\}.
  \]

   To express the desired  sensitivity bounds using a single constant,
  we combine the Lipschitz moduli of
  \(\cL_{r,S}\), \(S\in\cR_r\), as follows:
   we seek a common  $\kappa_{\partial,r}\geq 0$  satisfying
  \[
      \operatorname{Lip}(\cL_{r,S})
      \le
      \kappa_{\partial,r}\frac{|S|}{|\cV_r|},
      \qquad \forall S\in\cR_r.
  \] The normalization keeps the subset size explicit: singletons \(S=\{v\}\) carry
  the factor \(1/|\cV_r|\), whereas the sets
  \(S=\overline{\cV}_e(i)\) carry
  \(|\overline{\cV}_e(i)|/|\cV_r|\).
  The smallest  $\kappa_{\partial,r}\geq 0$ satisfying the above inequality   is
  \begin{equation}
      \kappa_{\partial,r}
      :=
      \begin{cases}
          0,
          & \partial_r=\emptyset,\\[1mm]
          \displaystyle
          \max_{S\in\cR_r}
          \frac{|\cV_r|}{|S|}
          \operatorname{Lip}(\cL_{r,S}),
          & \partial_r\ne\emptyset.
      \end{cases}
      \label{eq:clean_kappa_partial}
  \end{equation}

Combining the representations~\eqref{eq:image_formula}
  and~\eqref{eq:clean_edge_parameter_formula} with
  the definition~\eqref{eq:clean_kappa_partial} gives
  the following sensitivity bounds for the rows of
  \(Y_r(E)\) and the rows indexed by \(v\ne i\) of
  \(\mathsf B_{\cV_r,\cE_r}Y_r(E)\), with respect to
  perturbations of a single row \(E_{i,h}\).
 \begin{lemma}
  \label{lem:clean_boundary_sensitivity}
  Suppose Assumptions~\ref{ass:clean_problem_class}
  and~\ref{ass:clean_tree_delay} hold.
  If \(E,E'\in\R^{|\widehat\partial_r|\times d}\)
  differ only in row \((i,h)\in\widehat\partial_r\), then
  \begin{subequations}
  \begin{align}
  \bigl\|(Y_r(E')-Y_r(E))_e\bigr\|
  &\le
  \kappa_{\partial,r}
  \frac{|\overline{\cV}_e(i)|}{|\cV_r|}
  \|E'_{i,h}-E_{i,h}\|,
  &&e\in\cE_r,
  \label{eq:clean_partial_edge_sens}
  \\
  \bigl\|\mathsf B_{v,\cE_r}
         (Y_r(E')-Y_r(E))\bigr\|
  &\le
  \frac{\kappa_{\partial,r}}{|\cV_r|}
  \|E'_{i,h}-E_{i,h}\|,
  &&v\in\cV_r\setminus\{i\}.
  \label{eq:clean_partial_node_sens}
  \end{align}
  \end{subequations}
  \end{lemma}
\begin{proof}
See Appendix~\ref{app:clean_tree_oracle}. \qed
\end{proof}

 \begin{remark}[On \(\kappa_{\partial,r}\)]
    \label{rem:clean_kappa_partial}
     Suppose \(\partial_r\ne\emptyset\).
  By Lemma~\ref{lem:clean_boundary_sensitivity},
  \(\kappa_{\partial,r}\) controls how perturbing \(E_{i,h}\)
  affects \((Y_r(E))_e\) and
  \(\mathsf B_{v,\cE_r}Y_r(E)\) for \(v\ne i\).
  Assumption~\ref{ass:clean_problem_class} gives the crude bound
  \(\kappa_{\partial,r}\le L_r/\mu_r\),  which uses only
  the block's extreme curvature bounds.
 When \(f_i\in C^2\) for \(i\in\cV_r\), write \(H_i(\eta):=\nabla^2 f_i(\chi_r(\eta))\). By the inverse function theorem, \(\chi_r\) is \(C^1\) with Jacobian \(\left(\sum_{j\in\cV_r}H_j(\eta)\right)^{-1}\), so
 \[
\nabla\cL_{r,S}(\eta)^\top=\bigl(\sum_{i\in S}H_i(\eta)\bigr)\left(\sum_{j\in\cV_r}H_j(\eta)\right)^{-1},
\qquad \eta\in\R^d,
\]
 and the mean value inequality gives
\begin{equation}
\label{eq:clean_kappa_partial_smooth}
    \kappa_{\partial,r}
    \le
    \sup_{\eta\in\R^d}\max_{S\in\cR_r}
    \frac{|\cV_r|}{|S|}
    \left\|
    \left(\sum_{i\in S}H_i(\eta)\right)\left(\sum_{j\in\cV_r}H_j(\eta)\right)^{-1}
    \right\|.
\end{equation}
  Eq.~\eqref{eq:clean_kappa_partial_smooth} reveals an implicit preconditioning effect of exact block minimization: the Hessian sum over \(S\) is measured against the aggregate curvature \(\sum_{j\in\cV_r}H_j(\eta)\) of the block; curvature shared by all \(f_i\) cancels in this quotient, which is why \(\kappa_{\partial,r}\) can be much smaller than \(L_r/\mu_r\), as the quadratic case below shows.

  Suppose \(f_i(x)=\frac12x^\top Q_ix-c_i^\top x\), \(Q_i\succ0\), \(i\in\cV_r\). Then  \eqref{eq:clean_kappa_partial_smooth} holds with equality:
\[
\kappa_{\partial,r}
=
\max\left\{
\begin{aligned}
&\max_{\substack{e\in\cE_r\\(i,h)\in\widehat\partial_r}}
\frac{|\cV_r|}{|\overline{\cV}_e(i)|}
\left\|
\left(\sum_{j\in\overline{\cV}_e(i)}Q_j\right)
\left(\sum_{j\in\cV_r}Q_j\right)^{-1}
\right\|,\\
&\max_{\substack{(i,h)\in\widehat\partial_r\\v\in\cV_r\setminus\{i\}}}
|\cV_r|
\left\|
Q_v
\left(\sum_{j\in\cV_r}Q_j\right)^{-1}
\right\|
\end{aligned}
\right\}.
\]
In particular, if \(Q_i=Q\) for all \(i\in\cV_r\),
  then \(\kappa_{\partial,r}=1\).
  Thus \(\kappa_{\partial,r}\) can equal one while \(L_r/\mu_r=\lambda_{\max}(Q)/\lambda_{\min}(Q)\) is arbitrarily large.
\end{remark}

\smallskip
  \noindent\emph{Bounding the mismatch \(\|\hat Z_r^{\nu+1}-\bar Z_r^{\nu+1}\|^2\) caused by delays.}
  We combine Lemma~\ref{lem:clean_boundary_sensitivity}
  with the tree structure of the delays to bound this mismatch.
Fix \(r\in[p]\) and
  \(\nu\ge\max_{s\in[p]}\operatorname{diam}(\cT_s)\). The exact block response defining \(\bar Z_r^{\nu+1}\)
  uses the current exterior coordinates \(Y_h^\nu\). The edge coordinates \(\hat Y_e^{\nu+1}\) forming
  \(\hat Z_r^{\nu+1}\) instead come from block minimizations
  with delayed exterior coordinates \(Y_h^{\nu-\delta_{e,i}}\) (Proposition~\ref{prop:clean_delayed_bilevel}).
  We collect these in the matrices
  \(E^\nu,\widehat E_e^\nu\in
  \R^{|\widehat\partial_r|\times d}\), \(e\in\cE_r\), with rows
  \[
      E^\nu_{i,h}:=Y_h^\nu,
      \qquad
      (\widehat E_e^\nu)_{i,h}
      :=Y_h^{\nu-\delta_{e,i}},
      \qquad (i,h)\in\widehat\partial_r.
  \]
  By~\eqref{eq:clean_block_image_minimizer}, \eqref{eq:clean_block_images} and Proposition~\ref{prop:clean_delayed_bilevel},
  \[
      \bar Z_r^{\nu+1}
      =\mathsf B_{\cV,\cE_r}Y_r(E^\nu),
      \qquad
      \hat Y_e^{\nu+1}
      =(Y_r(\widehat E_e^\nu))_e.
  \]
  The two parameter matrices $E^\nu$ and $\widehat E_e^\nu$ differ row by row as follows:
  \begin{equation}\label{eq:clean_parameter_perturbation}
      E^\nu_{i,h}-(\widehat E_e^\nu)_{i,h}
      =
      \sum_{\ell=\nu-\delta_{e,i}}^{\nu-1}
      (Y_h^{\ell+1}-Y_h^\ell),
  \end{equation}
  at most \(D_r\) increments per row.

  Appendix~\ref{app:clean_tree_mismatch_proof} proves the following bound on \(\|\hat Z_r^{\nu+1}-\bar Z_r^{\nu+1}\|^2\) by combining~\eqref{eq:clean_partial_edge_sens} and~\eqref{eq:clean_partial_node_sens} with the tree structure of the delays \(\delta_{e,i}\). We  remark  that the  exploitation  of the delay structure is essential: a direct Cauchy--Schwarz bound on the squared norm of~\eqref{eq:clean_parameter_perturbation}, a sum of up to \(D_r\) increments, would deteriorate the bound by a factor \(D_r\).

\begin{lemma}
\label{lem:clean_tree_mismatch}
Suppose Assumptions~\ref{ass:clean_problem_class} and~\ref{ass:clean_tree_delay} hold. For every \(r\in[p]\) and \(\nu\ge\max_{s\in[p]}\operatorname{diam}(\cT_s)\),
\begin{equation}
    \|\hat Z_r^{\nu+1}-\bar Z_r^{\nu+1}\|^2
    \le
    4\kappa_{\partial,r}^2|\widehat\partial_r|
    \sum_{\ell=\nu-D_r}^{\nu-1}
    \sum_{(i,h)\in\widehat\partial_r}
    \|Y_h^{\ell+1}-Y_h^\ell\|^2,
    \label{eq:clean_raw_mismatch}
\end{equation}
the sum over \(\ell\) being zero when \(D_r=0\).
\end{lemma}
\begin{proof}
See Appendix~\ref{app:clean_tree_mismatch_proof}. \qed
\end{proof}

Lemma~\ref{lem:clean_tree_mismatch} bounds the positive
  mismatch term in~\eqref{eq:clean_main_recursion}
  through squared changes in the dual variables
  \(Y_h\), \(h\in\partial_r\).
  The negative term in~\eqref{eq:clean_main_recursion},
  however, involves squared block increments
  \(\|Z_r^{\nu+1}-Z_r^\nu\|^2\).   We therefore bound
  \(\|Y_h^{\ell+1}-Y_h^\ell\|^2\)
  in terms of
  \(\|Z_s^{\ell+1}-Z_s^\ell\|^2\),
  where \(s\in\cM_r\) is the block containing
  the exterior edge \(h\).

  The difference \(Y_h^{\ell+1}-Y_h^\ell\) is the row
  indexed by \(h\) of
  \(Y_{\cE_s}^{\ell+1}-Y_{\cE_s}^{\ell}\).   The relation
  \[
      Z_s^{\ell+1}-Z_s^\ell
      =
      \mathsf B_{\cV,\cE_s}
      (Y_{\cE_s}^{\ell+1}-Y_{\cE_s}^{\ell})
  \]
  allows us to recover these edge increments from
  \(Z_s^{\ell+1}-Z_s^\ell\).
  Since \(\cT_s\) is a tree,
  \(\mathsf B_{\cV,\cE_s}\) has full column rank, giving
\begin{equation}
    Y_{\cE_s}^{\ell+1}-Y_{\cE_s}^{\ell}
    =(\mathsf B_{\cV,\cE_s}^{\top}\mathsf B_{\cV,\cE_s})^{-1}\mathsf B_{\cV,\cE_s}^{\top}(Z_s^{\ell+1}-Z_s^\ell).
    \label{eq:clean_edge_from_block}
\end{equation}

Let \(\cB_{r,s}:=\cE_s\cap\partial_r\) collect the edges of block \(s\) touching block \(r\); each \(h\in\cB_{r,s}\) occurs in \(|\cI_h\cap\cV_r|\) rows of \(E^\nu\). With \[\Lambda_{r,s}:=\operatorname{Diag}(|\cI_h\cap\cV_r|:h\in\cB_{r,s})\] and\[
    \Gamma_r
    :=
    \begin{cases}
    \displaystyle
    \max_{s\in\cM_r}\lambda_{\max}\!\left(
    \Lambda_{r,s}^{1/2}
    \bigl[(\mathsf B_{\cV,\cE_s}^{\top}\mathsf B_{\cV,\cE_s})^{-1}\bigr]_{\cB_{r,s},\cB_{r,s}}
    \Lambda_{r,s}^{1/2}\right),
    &\cM_r\ne\emptyset,\\[1mm]
    0,
    &\cM_r=\emptyset,
    \end{cases}
\]
 \eqref{eq:clean_edge_from_block}  gives
\begin{equation}
    \sum_{\substack{(i,h)\in\widehat\partial_r\\ h\in\cE_s}}\|Y_h^{\ell+1}-Y_h^\ell\|^2
    =\sum_{h\in\cB_{r,s}}|\cI_h\cap\cV_r|\,\|Y_h^{\ell+1}-Y_h^\ell\|^2
    \le\Gamma_r\|Z_s^{\ell+1}-Z_s^\ell\|^2 .
    \label{eq:clean_increment_conversion}
\end{equation}
If every \(h\in\cB_{r,s}\) has exactly one endpoint in \(\cV_r\), then \(\Lambda_{r,s}=I\).

\begin{lemma}
\label{lem:clean_tree_mismatch_z}
Suppose Assumptions~\ref{ass:clean_problem_class} and~\ref{ass:clean_tree_delay} hold. For every \(r\in[p]\) and \(\nu\ge\max_{s\in[p]}\operatorname{diam}(\cT_s)\),
\begin{equation}
    \|\hat Z_r^{\nu+1}-\bar Z_r^{\nu+1}\|^2
    \le
    4\kappa_{\partial,r}^2|\widehat\partial_r|\,\Gamma_r
    \sum_{\ell=\nu-D_r}^{\nu-1}
    \sum_{s\in\cM_r}
    \|Z_s^{\ell+1}-Z_s^\ell\|^2,
    \label{eq:clean_tree_mismatch_z}
\end{equation}
the sum over \(\ell\) being zero when \(D_r=0\).
\end{lemma}
\begin{proof}
Since the sets \(\cB_{r,s}\), \(s\in\cM_r\), partition \(\partial_r\), the inner sum in~\eqref{eq:clean_raw_mismatch} splits over \(s\in\cM_r\); bounding each part by~\eqref{eq:clean_increment_conversion} gives the claim. \qed
\end{proof}

The mismatch is thus bounded by past increments of the neighboring blocks, in the form required by~\eqref{eq:clean_main_recursion}.

\smallskip\noindent\textbf{(iv) Linear convergence.} Combining Lemma~\ref{lem:clean_tree_mismatch_z} with Lemma~\ref{lem:clean_main_recursion}  yields the main convergence result.

\begin{theorem}
\label{thm:clean_linear_tree}
Suppose Assumptions~\ref{ass:clean_problem_class}, \ref{ass:clean_tree_delay}, and~\ref{ass:clean_beb} hold, with \(p>1\).  Define
\begin{equation}
A
:=
\max_{s\in[p]}
\sum_{r\in[p]:\,D_r>0,\,s\in\cM_r}
\left[
\left(\frac{2}{\mu_r}+\frac{1}{L_r}\right)
\kappa_{\partial,r}^2|\widehat\partial_r|\Gamma_r
\right],
\quad
D:=\max_{r\in[p]}D_r.
\label{eq:clean_tree_A}
\end{equation}
Let \(X^\star=\mathbf 1(x^\star)^\top\), with \(x^\star\) being the minimizer of $F$ in~\eqref{eq:clean_prob}.   If \(\tau_r^\nu\equiv\tau\) and
\begin{equation}
0<\tau\le\min\left\{\frac{1}{\omega_{\cP}},\,2\kappa_{\cP}\right\},\qquad
16LA\kappa_{\cP}\tau
\left[\left(1-\frac{\tau}{2\kappa_{\cP}}\right)^{-D}-1\right]\le1,
\label{eq:clean_tree_stepsize}
\end{equation}
then, for some constant \(C\ge1\) and all \(\nu\ge2(D+1)\),
\begin{equation}
    \|X^\nu-X^\star\|^2
    \le
    C\rho^\nu,
    \quad
    \rho:=1-\frac{\tau}{2\kappa_{\cP}}.
\label{eq:clean_tree_primal_rate}
\end{equation}
\end{theorem}

\begin{proof}
For each \(r\), choose \(i\in\cV_r^\partial\). By the definition of \(D_r\), \(\dist_{\cT_r}(i,v)\le D_r+1\) for all \(v\in\cV_r\), so \(\operatorname{diam}(\cT_r)\le2(D_r+1)\le2(D+1)\).
 Lemma~\ref{lem:clean_main_recursion} and~\eqref{eq:clean_tree_mismatch_z}, after regrouping delayed increments by block index, give
\[
\begin{aligned}
G(Z^{\nu+1})-G^\star
&\le
\left(1-\frac{\tau}{2\kappa_{\cP}}\right)
\bigl(G(Z^\nu)-G^\star\bigr)
-\frac{1}{8L\tau}
\sum_{r=1}^p\|Z_r^{\nu+1}-Z_r^\nu\|^2\\
&\quad
{}+\tau A
\sum_{\ell=\nu-D}^{\nu-1}
\sum_{r=1}^p\|Z_r^{\ell+1}-Z_r^\ell\|^2.
\end{aligned}
\]
Set \(\nu_0:=2(D+1)\) and define
\begin{equation}
\begin{aligned}
\cL^\nu&:=G(Z^\nu)-G^\star+\tau A\sum_{j=1}^{D}\left(\sum_{k=1}^{D-j+1}\rho^{-k}\right)
\sum_{r=1}^p\|Z_r^{\nu-j+1}-Z_r^{\nu-j}\|^2.
\end{aligned}
\label{eq:clean_tree_lyapunov}
\end{equation}
Conditions in~\eqref{eq:clean_tree_stepsize} give \(8LA\tau^2\sum_{j=1}^{D}\rho^{-j}\le1.\) Then for \(\nu\ge\nu_0\),
\[
\begin{aligned}
\cL^{\nu+1}
&\le\rho\cL^\nu
-\left(\frac{1}{8L\tau}-\tau A\sum_{j=1}^{D}\rho^{-j}\right)
\sum_{r=1}^p\|Z_r^{\nu+1}-Z_r^\nu\|^2\le\rho\cL^\nu.
\end{aligned}
\]
Thus \(G(Z^\nu)-G^\star\le\cL^{\nu_0}\rho^{\nu-\nu_0}\), which yields, with \(C_0:=\cL^{\nu_0}\rho^{-\nu_0}\) and \(g^\star:=\min_Y g(Y)=G^\star\),
\begin{equation}
    g(Y^\nu)-g^\star
    =
    G(Z^\nu)-G^\star
    \le
    C_0\rho^\nu .
\label{eq:clean_tree_linear_rate}
\end{equation}
Let \(Z^\star\) minimize \(G\) over \(\cS\), and let \(X^\star\) be given by \(x_i^\star=\nabla f_i^*(Z_i^\star)\), \(i\in\cV\).  Since \(G\) is \(1/L\)-strongly convex on \(\cS\), and each \(\nabla f_i^*\) is \(1/\mu\)-Lipschitz,
\begin{equation}
    \|X^\nu-X^\star\|^2
    \le
    \frac{1}{\mu^2}\|Z^\nu-Z^\star\|^2
    \le
    \frac{2LC_0}{\mu^2}\rho^\nu .
\label{eq:clean_tree_primal_prefactor}
\end{equation}
Absorbing constants yields~\eqref{eq:clean_tree_primal_rate}. \qed
\end{proof}

The second inequality in~\eqref{eq:clean_tree_stepsize} is implicit in \(\tau\). The following corollary gives an explicit sufficient condition.

\begin{corollary}
\label{cor:clean_tree_explicit_stepsize}
Suppose Assumptions~\ref{ass:clean_problem_class}, \ref{ass:clean_tree_delay}, and~\ref{ass:clean_beb} hold, with \(p>1\). If \(\tau_r^\nu\equiv\tau\) and
\begin{equation}
0<\tau\le
\min\left\{
\frac{1}{\omega_{\cP}},\
\frac{2\kappa_{\cP}}{2D+1},\
\frac{1}{\sqrt{8L(2D+1)A}}
\right\},
\label{eq:clean_tree_sufficient_stepsize}
\end{equation}
then, for some constant \(C\in(0,\infty)\) and all \(\nu\ge2(D+1)\),
\[
    \|X^\nu-X^\star\|^2
    \le C\rho^\nu,
    \quad
    \rho:=1-\frac{\tau}{2\kappa_{\cP}}.
\]
\end{corollary}


\begin{proof}
The first inequality in~\eqref{eq:clean_tree_stepsize} follows directly from~\eqref{eq:clean_tree_sufficient_stepsize}. For the second, when \(D\ge1\), the bound \(1-\rho\le1/(2D+1)\) yields
\[
\sum_{j=1}^{D}\rho^{-j}
\le\frac{D}{\rho^D}
\le\frac{D}{1-D(1-\rho)}
\le2D+1.
\]
The third bound therefore gives \(8LA\tau^2\sum_{j=1}^{D}\rho^{-j}\le1\), which is equivalent to the second inequality in~\eqref{eq:clean_tree_stepsize}. When \(D=0\), we have \(A=0\), so the second condition
  in~\eqref{eq:clean_tree_stepsize} imposes no additional
  restriction on \(\tau\). Theorem~\ref{thm:clean_linear_tree} applies. \qed
\end{proof}

\begin{remark}[On the rate expression]
Choosing \(\tau\) at the upper bound in
\eqref{eq:clean_tree_sufficient_stepsize} gives
\begin{equation}
\rho
=
1-\frac{1}{2\kappa_{\cP}}
\min\left\{
\frac{1}{\omega_{\cP}},\
\frac{2\kappa_{\cP}}{2D+1},\
\frac{1}{\sqrt{8L(2D+1)A}}
\right\}.
\label{eq:clean_rate_factor}
\end{equation}  For this stepsize and \(p>1\), including the initial
  \(2(D+1)\) rounds used in
  Theorem~\ref{thm:clean_linear_tree}, the number of rounds
  to reach \(\|X^\nu-X^\star\|\le\epsilon\) is
\[
O\!\left(
D+
\max\left\{
\kappa_{\cP}\omega_{\cP},\,
D,\,
\kappa_{\cP}\sqrt{LDA}
\right\}
\log\frac{C}{\epsilon}
\right).
\]
Here \(C\) is the constant in Theorem~\ref{thm:clean_linear_tree}.
   The bound quantifies the interplay between function curvature
  and network topology through four partition/graph-dependent quantities.   Two of them, \(\kappa_{\cP}\) and \(A\),
  also depend on function geometry.
\begin{itemize}
 \item[(a)] {\bf Block preconditioning, \(\kappa_{\cP}\).} The parameter \(\kappa_{\cP}\) measures the coupling remaining
  between the block subproblems after exact block minimization  over the variables
  grouped by \(\cP\). The partition determines which interactions are treated
  jointly; the agents' cost functions determine the progress these
  minimizations achieve.
  For fixed admissible \(\tau\), a smaller \(\kappa_{\cP}\)
  reduces \(\rho\), improving the guaranteed convergence rate.
  \item[(b)] {\bf Block overlap, \(\omega_{\cP}\).} This  is the largest number of
  blocks sharing an agent. It accounts for the overlap
  among simultaneous block updates.
  Keeping  each agent's incident edges in fewer blocks
  limits this overlap and relaxes the sufficient restriction
  \(\tau\le1/\omega_{\cP}\).
  The restriction depends on {\it local overlap, not on the
  total number \(p\) of blocks}.

   \item[(c)] {\bf Message propagation, \(D\).} The largest delay \(D=\max_r D_r\) is determined by
  distances within each tree from boundary agents to the
  nearest endpoints of the updated edges.  Limiting block diameters bounds these delays.
  A smaller \(D\) reduces the bound on the initial
  propagation time and relaxes the delay-dependent
  stepsize restrictions.

 \item[(d)] {\bf Sensitivity to outdated information, \(A\).} This parameter measures how sensitive block updates are
  to outdated information from neighboring blocks.
  A smaller \(A\) reduces the performance penalty of
  communication delays: parallel block updates can use
  less damping, yielding a faster convergence rate.   Definition~\eqref{eq:clean_tree_A} combines block curvature,
  squared response sensitivities \(\kappa_{\partial,r}^2\),
  boundary-incidence counts \(|\widehat\partial_r|\), and
  neighboring-block geometry through \(\Gamma_r\).
  The maximum over \(s\) accounts for the largest accumulated
  weighted influence of one block on its neighbors'
  updates; the sum includes only neighbors \(r\) with \(D_r>0\).
  Partition design changes these factors through the
  location of interfaces and the conjugate functions grouped in
  each block.
  In the common-Hessian quadratic example of
  Remark~\ref{rem:clean_kappa_partial},
  \(\kappa_{\partial,r}=1\), irrespective of Hessian
  conditioning.
  \end{itemize}

The rate thus guides partition selection for the given
  functions and network.
  Unlike tuning gossip weights or communication rounds
  for prescribed optimization updates, choosing \(\cP\) provides additional
  design freedom to adapt the optimization itself to the topology.
  It changes which variables are optimized jointly and the
  coupling between the resulting subproblems.
  The convergence guarantee captures function geometry and
  topology {\it jointly}, rather than only through the separately
  assessed conditioning and mixing parameters of standard
  gossip bounds.
\end{remark}
\begin{remark}[Spanning-tree specialization]
  \label{rem:clean_spanning_tree}
  The framework includes a single-block construction on a
  spanning tree of any connected graph \(\cG\).
  For a spanning tree \(\cT\), choose
  \[
      \mathsf A=\mathsf B_{\cV,\cE(\cT)},\qquad
      \cJ_1=\cE(\cT),\qquad
      \cC_1=\cV.
  \]
  This representation retains only the tree-edge variables,
  all in one block. There are therefore no exterior coordinates or inter-block
  coupling, so \(A=0\) and we may take \(\kappa_{\cP}=1\).
  Under Assumption~\ref{ass:clean_problem_class},
  Algorithm~\ref{alg:gate_uniform_proposed}, run on the edges
  of \(\cT\) with exact message updates and
  \(\tau_1^\nu\equiv1\), returns the exact primal solution
  after at most \(\operatorname{diam}(\cT)+1\) rounds, by
  Proposition~\ref{prop:clean_delayed_bilevel}.
  This provides an optimization analogue of finite-time
  consensus. Despite this finite-round guarantee, a single spanning-tree
  block {\it need not be preferable to a multi-block decomposition}.
  The computational and communication considerations governing
  this choice are discussed in
Sec.~\ref{subsec:clean_tree_comparison}. \vspace{-0.5cm}
  \end{remark}
\subsection{Rate scaling and partition design}
  \label{subsec:clean_rate_examples}\vspace{-0,2cm}
 We examine how partition selection controls network-size   scaling on tree-dominant graphs and rings.
  Throughout this subsection, we consider connected graph   sequences \(\cG_m=(\cV_m,\cE_m)\), with \(|\cV_m|=m\).   Problem and partition quantities may depend on \(m\);  we suppress this dependence in the notation.
  All asymptotic bounds are with respect to \(m\to\infty\),
  with implied constants independent of \(m\).

\smallskip
\noindent\textbf{1) Tree-dominant graphs.}
  For \(\kappa_{\cP}=O(1)\),
  \eqref{eq:clean_rate_factor} reads
  \((1-\rho)^{-1}=O(D)\) whenever
  \(\omega_{\cP}=O(D)\) and \(LA=O(D)\).
  If instead \(\omega_{\cP}=O(1)\) and \(LA=O(1)\),
  conditions in ~\eqref{eq:clean_tree_stepsize}
  admits
  \(\tau=\Theta(\kappa_{\cP}\log D/D)\)
  for sufficiently large \(D\), yielding
  \[(1-\rho)^{-1}=O(D/\log D).\]We next identify a graph class and partition that attain
  the latter regime.
A graph family is \emph{tree-dominant} if each graph
  contains a tree subgraph with a number of remaining
  edges bounded independently of network size.
  For each graph, we choose the edge partition consisting
  of the tree edges as one block and each remaining edge
  as a singleton block.
  Under the curvature assumptions below, this partition
  yields the required parameter bounds.

\begin{proposition}
\label{prop:clean_large_tree_scaling} For each connected graph \(\cG=(\cV,\cE)\), choose
  a tree subgraph \(\cT_1=(\cV_1,\cE_1)\) satisfying
  \[
      |\cE\setminus\cE_1|=O(1),
  \]
  and the edge partition
  \[
      \cP
      :=\{\cE_1\}
      \cup\bigl\{\{h\}:h\in\cE\setminus\cE_1\bigr\}.
  \]
Suppose Assumptions~\ref{ass:clean_problem_class}
  and~\ref{ass:clean_tree_delay} hold, with
  \(L=O(1)\) and \(\mu^{-1}=O(1)\).   Assume there exist \(R\succ0\) and
  \(0<\underline\alpha\le\overline\alpha\), with
  \(\overline\alpha/\underline\alpha=O(1)\), such that
  every agent \(i\) incident to an edge in
  \(\cE\setminus\cE_1\) satisfies \(f_i\in C^2\) and
  \[
      \underline\alpha R
      \preceq\nabla^2 f_i(x)
      \preceq\overline\alpha R,
      \qquad x\in\R^d.
  \]
For \(p>1\) and sufficiently large \(D\),
  \eqref{eq:clean_tree_stepsize} admits a stepsize
  fixed across iterations,
  \(\tau=\Theta(\log D/D)\), yielding
  \[
      (1-\rho)^{-1}=O(D/\log D).
  \]
\end{proposition}\begin{proof}
  See Appendix~\ref{app:clean_large_tree_scaling}.\end{proof}
The selected tree has diameter at most \(m-1\),
  and possibly much smaller.
  The following examples have
  \(\operatorname{diam}(\cT_1)=\Theta(m)\)
  and \(D=\Theta(m)\).
\begin{itemize}
\item[(a)] \emph{\bf Ring.}  Choose \(m-q\) consecutive edges as one path block,
  where \(1\le q=O(1)\), and put each remaining edge
  in a singleton block.
  Then \(\operatorname{diam}(\cT_1)=m-q\).
  The common-matrix Hessian bounds in
  Proposition~\ref{prop:clean_large_tree_scaling}
  are required only at the \(O(q)\) agents incident
  to the singleton edges. 
\item[(b)] \emph{\bf Barbell.} Consider two cliques of bounded size, each containing
  at least two agents, joined by a path of
  \(\Theta(m)\) edges.
  Choose this path as \(\cT_1\) and put each clique
  edge in a singleton block.
  The number of edges outside \(\cT_1\), and hence
  the number of agents subject to the common-matrix
  Hessian bounds, is \(O(1)\). 
\end{itemize}
Under the curvature assumptions of
  Proposition~\ref{prop:clean_large_tree_scaling},
  both partitions therefore give
  \[
      (1-\rho)^{-1}=O(m/\log m).
  \]
\smallskip
\noindent\emph{Iteration complexity.}
For the ring and barbell partitions above, including
  the initial propagation rounds,
  Theorem~\ref{thm:clean_linear_tree} gives
  \[
      O\!\left(m+\frac{m}{\log m}\log\frac{C}{\epsilon}\right)
  \]
  rounds to reach \(\|X^\nu-X^\star\|\le\epsilon\), where
  \(0<\epsilon<1\) and \(C\ge1\) is an admissible prefactor
  in~\eqref{eq:clean_tree_primal_rate}.

To estimate \(C\), take \(\tau=\kappa_{\cP}\log D/D\) as in the proof of
  Proposition~\ref{prop:clean_large_tree_scaling}.
  With \(\nu_0:=2(D+1)\),
  \[
      \rho^{-\nu_0}
      =\left(1-\frac{\log D}{2D}\right)^{-2(D+1)}
      =D(1+o(1)).
  \]
  Thus~\eqref{eq:clean_tree_primal_prefactor} gives \(C=O(1+D\cL^{\nu_0})\), where \(\cL^\nu\) is defined in~\eqref{eq:clean_tree_lyapunov}. For sequences of problem instances and initializations satisfying the following finite-propagation scaling conditions,
  \begin{equation}
      \frac{G(Z^{\nu_0})-G^\star}{m}=O(1),
      \qquad
      \frac1m\max_{\nu_0-D\le\ell<\nu_0}
      \sum_{r=1}^p\|\hat Z_r^{\ell+1}-Z_r^\ell\|^2=O(1),
  \label{eq:clean_initial_phase_scaling}
  \end{equation}
  with constants independent of \(m\), we have
  \(\cL^{\nu_0}=O(m)\). Indeed,
  \[
      \tau^3 A\sum_{j=1}^D\sum_{k=1}^{D-j+1}\rho^{-k}
      =4A\kappa_{\cP}^2\tau
      \bigl[\rho^{-D}-1-D(1-\rho)\bigr]
      =O\!\left(\frac{\log D}{\sqrt D}\right).
  \]
  Thus \(C=\max\{1,2L\cL^{\nu_0}\rho^{-\nu_0}/\mu^2\}=O(m^2)\). Since \(\log C=O(\log m)\), its contribution is absorbed by the initial \(O(m)\) term, giving
  \[
      O\!\left(m+\frac{m}{\log m}\log\frac1\epsilon\right)
  \]
  rounds. More generally, the same bound holds whenever an admissible \(C\) grows at most polynomially with \(m\).

\smallskip
\noindent\textbf{2) Rings with equal-length partitions.} For \(\omega_{\cP}=O(1)\) and \(L\,A=O(1)\),
  \eqref{eq:clean_rate_factor} gives, when \(D\ge1\),
  \[
      (1-\rho)^{-1}
      =O\!\left(
          \max\{D,\kappa_{\cP}\sqrt D\}
      \right).
  \]
  Shortening the blocks therefore need not improve the
  rate bound if the resulting increase in
  \(\kappa_{\cP}\) offsets the smaller delays.
  We examine this tradeoff by partitioning a ring into
  equal-length paths.

 \begin{proposition}
  \label{prop:clean_ring_partition_scaling}
  Let \(\cP\) partition the edges of a ring with \(m\)
  agents into \(p\ge2\) consecutive path blocks,
  each containing \(D+1\) edges, where
  \(D\ge1\) and \(m=p(D+1)\).
  Assume \(f_i\in C^2\) for every \(i\in\cV\).
  Suppose there exist \(R\succ0\) and
  \(0<\underline\alpha\le\overline\alpha\), with
  \(\overline\alpha/\underline\alpha=O(1)\), such that
  \[
      \underline\alpha R
      \preceq\nabla^2 f_i(x)
      \preceq\overline\alpha R,
      \qquad i\in\cV,\quad x\in\R^d.
  \]
  Then the block error bound~\eqref{eq:clean_beb_image}
  holds with
  \[
      \kappa_{\cP}=O(m^2/D).
  \]
  \end{proposition}
  \begin{proof}
  See Appendix~\ref{app:clean_ring_partition_scaling}.
  \end{proof}

 For this partition,
  \[
      \omega_{\cP}=2,\qquad
      D_r=D,\qquad
      |\widehat\partial_r|=2,\qquad
      \Gamma_r\le1,
      \qquad r\in[p].
  \]
  After the change of variables \(u=R^{1/2}x\), the normalized Hessian bounds give \(LA=O(1)\). The same contraction factor applies in the original coordinates (Appendix~\ref{app:clean_ring_partition_scaling}).
  Consequently,
  \[
      (1-\rho)^{-1}
      =O\!\left(
          \max\left\{\frac{m^2}{D},D,\frac{m^2}{\sqrt D}\right\}
      \right)
      =O(m^2/\sqrt D).
  \]
  Within the equal-length family, this upper bound
  is minimized in order by taking \(D=\Theta(m)\),
  equivalently a bounded number of blocks,
  giving \(O(m^{3/2})\).

  Under the curvature assumptions of
  Proposition~\ref{prop:clean_large_tree_scaling},
  one long path with a bounded singleton remainder
  instead gives \(O(m/\log m)\).
  Both choices have \(D=\Theta(m)\), but their
  inter-block coupling differs.
  Thus partition design must account for the coupling
  left between blocks, rather than balance their
  lengths alone.
\vspace{-0.2cm}

\section{Surrogate enhancement}
\label{sec:clean_surrogate}

In this section, we introduce a surrogate variant of GATE to reduce per-iteration computational and communication costs.  The updates in~\eqref{eq:clean_dmp_updates_a} and~\eqref{eq:clean_factor_to_var} of Algorithm~\ref{alg:gate_uniform_proposed} require solving local minimization problems exactly, and each intra-cluster message \(\mu_{i\to e}^\nu(\cdot)\) is a function of the receiving edge variable \(Y_e\), which can be expensive to compute and transmit.  This is particularly acute in the nonquadratic case, where messages need not admit compact finite-dimensional parametrizations; even in the quadratic case, exact messages may carry dense matrix parameters.  Our approach is to replace each \(g_i(Y_{\cA_i})=f_i^*((\mathsf BY)_i)\) by a tractable surrogate: this makes the block subproblems cheaper to solve and induces messages with lightweight parametrizations.

More formally, for each node \(i\in\cV\), let \[ \R^{|\cA_i|\times d}\times\R^{|\cA_i|\times d}\ni(Y_{\cA_i},\bar Y_{\cA_i})\mapsto \widetilde g_i(Y_{\cA_i};\bar Y_{\cA_i})\] be a surrogate of \(g_i(Y_{\cA_i})=f_i^*(\mathsf B_{i,\cA_i}Y_{\cA_i})\).  While we can directly work on the general \(\widetilde g_i\), we mainly focus on a structured choice obtained from a surrogate of the local conjugate \(f_i^*\): choose \[ \R^d\times\R^d\ni(U_i,Z_i)\mapsto \widetilde f_i^*(U_i;Z_i),\] and set
\begin{equation}
\widetilde g_i(Y_{\cA_i};\bar Y_{\cA_i})
:=
\widetilde f_i^*(\mathsf B_{i,\cA_i}Y_{\cA_i};\mathsf B_{i,\cA_i}\bar Y_{\cA_i}),
\label{eq:clean_sur_factor_structured}
\end{equation}
where \(\mathsf B_{i,\cA_i}\) only gives the composition with the edge variables $Y_{\cA_i}$ and their center $\bar Y_{\cA_i}$; \(\widetilde f_i^*\) is chosen from the local conjugate \(f_i^*\) alone and is graph-independent. At iteration \(\nu\), write \(Z^\nu:=\mathsf BY^\nu\) and \(Z_i^\nu:=(Z^\nu)_i\).  Replacing the exact local terms in the edge-variable update~\eqref{eq:clean_dmp_updates_a} and the message update~\eqref{eq:clean_factor_to_var} by these surrogates yields Algorithm~\ref{alg:clean_sur_dmp}. When \(\widetilde f_i^*(U;Z)\equiv f_i^*(U)\), Algorithm~\ref{alg:gate_uniform_proposed} is recovered. 

\begin{algorithm}[t!]
{\fontsize{6.2}{7.0}\selectfont
\DontPrintSemicolon
\SetKwFor{ForEach}{foreach}{}{end foreach}
\SetKwFor{NodeActions}{Each node}{executes the following steps, in parallel with the other nodes:}{}
\SetAlgoRefName{2}
\caption{GATE-S: Graph-Tearing message passing with surrogates}
\label{alg:clean_sur_dmp}\medskip
\KwIn{Edge assignment \(\mathfrak a\) with ownership sets
\(\{\cA_i^{\mathfrak a}\}_{i\in\cV}\); edge partition
\(\{\cE_r\}_{r=1}^p\); initial edge variables \(Y^0\);
finite-valued strongly convex initial messages
\(\widetilde\mu_{i\to e}^0\), \(e\in\cE\), \(i\in\cI_e\);
local surrogates \(\widetilde f_i^*(\cdot;\cdot)\);
relaxation parameters \(\tau_r^\nu\in[0,1]\); \(N\geq1\).}

\For{\(\nu=0,1,\ldots,N-1\)}{
\NodeActions{\(i\in\cV\)}{
\BlankLine
\texttt{(1) Assigned edge-variable update.}\;
\ForEach{\(r\in[p]\) and \(e\in\cA_i^{\mathfrak a}\cap\cE_r\), compute}{
\begin{subequations}\label{eq:clean_sur_node_owned_updates}
\begin{align}
\hat Y_e^{\nu+1}
&\in\argmin_{y_e\in\R^{1\times d}}
\left\{\sum_{j\in\cI_e}\widetilde\mu_{j\to e}^{\nu}(y_e)\right\},
\label{eq:clean_sur_dmp_updates_a}\\
Y_e^{\nu+1}
&=Y_e^\nu+\tau_r^\nu\bigl(\hat Y_e^{\nu+1}-Y_e^\nu\bigr).
\label{eq:clean_sur_dmp_updates_b}
\end{align}
\end{subequations}
Send \(Y_e^{\nu+1}\) to the other endpoint of \(e\).\;
}
\BlankLine
\texttt{(2) Node-to-edge message update.}\;
\ForEach{\(r\in[p]\) and \(e\in\cA_i\cap\cE_r\), compute}{
\begin{equation}
\begin{aligned}
\widetilde\mu_{i\to e}^{\nu+1}(y_e)
&=\min_{Y_{\cA_i^\circ\setminus\{e\}}}
\Bigg\{
\widetilde f_i^*\!\left(\mathsf B_{i,e}y_e
+\mathsf B_{i,\cA_i^\circ\setminus\{e\}}
Y_{\cA_i^\circ\setminus\{e\}}
+\mathsf B_{i,\cA_i^\partial}
Y_{\cA_i^\partial}^{\nu+1};\mathsf B_{i,\cA_i}Y_{\cA_i}^{\nu+1}\right)\\
&\qquad+\sum_{h\in\cA_i^\circ\setminus\{e\}}
\ \sum_{j\in\cI_h\setminus\{i\}}
\widetilde\mu_{j\to h}^{\nu}(Y_h)
\Bigg\}.
\end{aligned}
\label{eq:clean_sur_factor_to_var}
\end{equation}
Send \(\widetilde\mu_{i\to e}^{\nu+1}\) to the other endpoint of \(e\).\;
}
}
}
\BlankLine
\textbf{Primal recovery.}
Each node \(i\in\cV\) computes
\(x_i^N:=\nabla f_i^*((\mathsf BY^N)_i)\).
}
\end{algorithm}

\subsection{Surrogate regularity conditions}
\label{subsec:clean_surrogate_regularity}

We require  the following mild regularity conditions on the local surrogates to guarantee  convergence of Algorithm~\ref{alg:clean_sur_dmp}. 

\begin{assumption}
\label{ass:clean_surrogate_nodewise}
For each node \(i\in\cV\), the local surrogate \(\widetilde f_i^*\) is differentiable and satisfies the following conditions:
\begin{enumerate}
    \item[(i)] (\texttt{touching and gradient consistency}).  For all \(Z_i\in\R^d\),
    \[
        \widetilde f_i^*(Z_i;Z_i)=f_i^*(Z_i),\quad
        \nabla_U\widetilde f_i^*(Z_i;Z_i)=\nabla f_i^*(Z_i).
    \]
    \item[(ii)] (\texttt{majorization}).  For all \(U_i,Z_i\in\R^d\),
    \[
        f_i^*(U_i)\le \widetilde f_i^*(U_i;Z_i).
    \]
    \item[(iii)] (\texttt{strong convexity and smoothness in $U_i$}).  There exist constants \(\widetilde\mu_i^*,\widetilde L_i^*\in(0,+\infty)\) such that, for every center \(Z_i\), the map \(U_i\mapsto\widetilde f_i^*(U_i;Z_i)\) is \(\widetilde\mu_i^*\)-strongly convex and \(\widetilde L_i^*\)-smooth.
    \item[(iv)] (\texttt{Lipschitz continuity in $Z_i$}). The gradient \(\nabla_U\widetilde f_i^*(U_i;Z_i)\) is globally Lipschitz in \(Z_i\), uniformly in \(U_i\).
\end{enumerate}
\end{assumption}
For a center \(Z\in\cS\), define
\[
    \widetilde G(U;Z):=\sum_{i\in\cV}\widetilde f_i^*(U_i;Z_i),
    \qquad U\in\R^{m\times d}.
\]
For \(Z=\mathsf BY\), the value of the surrogate block minimization equals the minimum of \(\widetilde G(\cdot;Z)\) over the slice \(Z+\cU_r\):
\begin{equation*}
\begin{gathered}
    \min_{W_r\in\R^{|\cE_r|\times d}}
    \sum_{i\in\cV}\widetilde g_i\!\left((W_r,Y_{\overline{\cE}_r})_{\cA_i};Y_{\cA_i}\right)
    =\min_{U\in Z+\cU_r}\widetilde G(U;Z)
    =:\widetilde G(\widetilde O_r(Z);Z),
\end{gathered}
\end{equation*}
where \(\widetilde O_r(Z)\) is the unique minimizer, well defined because \(\widetilde G(\cdot;Z)\) is strongly convex and \(Z+\cU_r\) is closed. Let
\begin{equation}
    \widetilde\Delta_r(Z):=
    G(Z)-\widetilde G(\widetilde O_r(Z);Z)\ge0
\label{eq:clean_sur_progress_def}
\end{equation}
be the progress of one exact minimization of the surrogate block problem from \(Z\); nonnegativity follows from the touching and optimality conditions. 

We postulate the following surrogate block error bound condition, which is the counterpart of Assumption~\ref{ass:clean_beb};   it bounds the optimality gap by the aggregate surrogate block progress.
\begin{assumption}[surrogate block error bound]
\label{ass:clean_sur_beb}
There exists a partition-dependent constant \(\widetilde\kappa_{\cP}<\infty\) such that
\begin{equation}
    G(Z)-G^\star
    \le
    \widetilde\kappa_{\cP}\sum_{r=1}^p\widetilde\Delta_r(Z),
    \quad Z\in\cS.
\label{eq:clean_sur_beb_image}
\end{equation}
\end{assumption}
As in the exact case, the smallest admissible constant \(\widetilde\kappa_{\cP}\) in Assumption~\ref{ass:clean_sur_beb} is the supremum of the ratio of the global gap \(G(Z)-G^\star\) to the aggregate surrogate block progress \(\sum_{r=1}^p\widetilde\Delta_r(Z)\) over \(Z\in\cS\) with \(G(Z)>G^\star\).
This constant enters the rate of Algorithm~\ref{alg:clean_sur_dmp}; its role is analogous to that of \(\kappa_{\cP}\) for Algorithm~\ref{alg:gate_uniform_proposed}.  

By the touching, optimality, and majorization conditions,
\[
    0\le
    \widetilde\Delta_r(Z)
    \le
    G(Z)-G(\widetilde O_r(Z))
    \le
    \Delta_r(Z).
\]
Consequently, the smallest admissible constants in Assumptions~\ref{ass:clean_beb} and~\ref{ass:clean_sur_beb} satisfy \(\kappa_{\cP}\le\widetilde\kappa_{\cP}\). If there exists \(\eta\in(0,1]\) such that
\begin{equation*}
\widetilde\Delta_r(Z)\ge \eta\,\Delta_r(Z),
\quad \forall Z\in\cS,\quad r\in[p],
\end{equation*}
then \(\widetilde\kappa_{\cP}\le\kappa_{\cP}/\eta\).  Thus, the price of using surrogate block updates is captured by this relative progress fraction: if \(\eta\) remains bounded away from zero, the network-size scaling of the exact block error bound is preserved up to a constant factor.

\subsection{Surrogate examples}
\label{subsec:clean_surrogate_examples}

The following are two practical choices of \(\widetilde f_i^*\): a first-order surrogate with fixed metric, and a second-order surrogate that uses second-order information and allows the preconditioner to vary with the center.

\begin{example}[First-order surrogate]
\label{ex:clean_sur_first}
For a fixed \(M_i\in\mathbb S_{++}^d\), set
\begin{equation}
\widetilde f_i^*(U_i;Z_i)
:=
 f_i^*(Z_i)
 +\inner{\nabla f_i^*(Z_i), U_i-Z_i}
 +\frac12\inner{U_i-Z_i, M_i(U_i-Z_i)} .
\label{eq:clean_sur_first_order}
\end{equation}
A special case is the scalar metric: \(M_i=\alpha_i^{-1}I_d\), with \(\alpha_i>0\).
\end{example}

\begin{example}[Second-order surrogate]
\label{ex:clean_sur_second}
Let \(H_i(Z_i)\in\mathbb S_{++}^d\) be chosen from second-order information, for instance \(H_i(Z_i)=\nabla^2 f_i^*(Z_i)+\zeta_i I_d\) with \(\zeta_i\ge0\) when \(f_i^*\in C^2\), and define
\begin{equation*}
\widetilde f_i^*(U_i;Z_i)
:=
 f_i^*(Z_i)
 +\inner{\nabla f_i^*(Z_i), U_i-Z_i}
 +\frac12\inner{U_i-Z_i, H_i(Z_i)(U_i-Z_i)}.
\end{equation*}
\end{example}

\subsubsection{Regularity and block error bound comparison}
\label{subsec:clean_surrogate_beb_regimes}
\noindent\textbf{Surrogate regularity.}\ Let \(H_i(Z_i)\) denote the quadratic metric in the surrogate, with \(H_i(Z_i)\equiv M_i\) in Example~\ref{ex:clean_sur_first} and center-dependent \(H_i\) in Example~\ref{ex:clean_sur_second}.  Touching and gradient consistency are immediate.  If \(f_i^*\in C^2\), majorization holds whenever \(\nabla^2f_i^*(U_i)\preceq H_i(Z_i)\) for all \(U_i,Z_i\in\R^d\). Assumption~\ref{ass:clean_surrogate_nodewise}(iii) holds with \(\widetilde\mu_i^*:=\inf_{Z_i\in\R^d}\lambda_{\min}(H_i(Z_i))\) and \(\widetilde L_i^*:=\sup_{Z_i\in\R^d}\lambda_{\max}(H_i(Z_i))\) whenever \(0<\widetilde\mu_i^*\le\widetilde L_i^*<\infty\). For Example~\ref{ex:clean_sur_first}, Assumption~\ref{ass:clean_surrogate_nodewise}(iv) follows from the global Lipschitz continuity of \(\nabla f_i^*\). For nonconstant \(H_i\), Example~\ref{ex:clean_sur_second} fails this global requirement and is retained as a practical variant. If \(H_i\) is locally Lipschitz and the edge iterates are bounded, the surrogate gradients are uniformly Lipschitz in the center on the relevant bounded sets. The subsequent center-sensitivity analysis can instead use these local Lipschitz bounds; we do not pursue this extension here.

\noindent\textbf{Block error bound comparison.}\ Given a partition \(\cP\), we next compare $\widetilde\kappa_\cP$ in Assumption~\ref{ass:clean_sur_beb} with the exact constant \(\kappa_{\cP}\) for the fixed-metric surrogate in Example~\ref{ex:clean_sur_first}.  Fix \(r\) and write a block variation as \(U=Z+P\) with \(P\in\cU_r\).  By \eqref{eq:clean_sur_progress_def} and the change of variables \(U=Z+P\), \(\widetilde\Delta_r(Z)=G(Z)-\min_{P\in\cU_r}\widetilde G(Z+P;Z)\).  For \(i\notin\cV_r\), \(\mathsf B_{i,\cE_r}=0\), so every \(P\in\cU_r\) satisfies \(P_i=0\).  Hence \eqref{eq:clean_sur_first_order} gives
\[
    \widetilde G(Z+P;Z)
    =
    G(Z)+\inner{(\nabla G(Z)^\top)_{\cV_r},P_{\cV_r}}
    +\frac12\sum_{i\in\cV_r}\inner{P_i,M_iP_i}.
\]
Substituting this identity into the definition gives
\begin{equation}
    \widetilde\Delta_r(Z)
    =-
    \min_{P\in\cU_r}
    \left\{
        \inner{(\nabla G(Z)^\top)_{\cV_r},P_{\cV_r}}
        +\frac12\sum_{i\in\cV_r}\inner{P_i,M_iP_i}
    \right\}.
\label{eq:clean_fixed_sur_progress_explicit}
\end{equation}
If the fixed metrics on \(\cV_r\) satisfy the curvature sandwich
\begin{equation}
    \rho_i^{-1}M_i
    \preceq \nabla^2 f_i^*(z)
    \preceq M_i,
    \quad \forall z\in\R^d,\quad i\in\cV_r,
\label{eq:clean_fixed_majorant_curvature}
\end{equation}
with \(\rho_i\ge1\) and \(\rho_r:=\max_{i\in\cV_r}\rho_i\), then the lower curvature bound in \eqref{eq:clean_fixed_majorant_curvature} implies
\[
\begin{aligned}
    \Delta_r(Z)
    &\le
    -\min_{P\in\cU_r}
    \left\{
        \inner{(\nabla G(Z)^\top)_{\cV_r},P_{\cV_r}}
        +\frac{1}{2\rho_r}\sum_{i\in\cV_r}\inner{P_i,M_iP_i}
    \right\} \\
    &=
    \rho_r
    \left[-\min_{Q\in\cU_r}
    \left\{
        \inner{(\nabla G(Z)^\top)_{\cV_r},Q_{\cV_r}}
        +\frac12\sum_{i\in\cV_r}\inner{Q_i,M_iQ_i}
    \right\}\right] \\
    &=
    \rho_r\widetilde\Delta_r(Z),
\end{aligned}
\]
where the equality uses \(P=\rho_rQ\), the linearity of \(\cU_r\), and \eqref{eq:clean_fixed_sur_progress_explicit}.  Hence
\begin{equation*}
    \kappa_{\cP}
    \le
    \widetilde\kappa_{\cP}
    \le
    \rho_{\max}\kappa_{\cP},
    \quad
    \rho_{\max}:=\max_{r\in[p]}\rho_r.
\end{equation*}
Since \(\rho_{\max}=\max_{i\in\cV}\rho_i\), it is independent of the partition. If \(\rho_{\max}=O(1)\), then \(\widetilde\kappa_{\cP}\) and \(\kappa_{\cP}\) have the same network-size scaling. If the surrogate Hessian \(M_i\) matches \(\nabla^2 f_i^*\) well (i.e., \(\rho_i\mathrel{\approx}1\)), the resulting degradation from \(\kappa_\cP\) to \(\widetilde\kappa_\cP\) is only a small factor. In the scalar case, choosing \(M_i=\mu_{f_i}^{-1}I_d\) ensures~\eqref{eq:clean_fixed_majorant_curvature} with \(\rho_i=L_{f_i}/\mu_{f_i}\), thereby retaining the network scaling of the block error bound up to a local condition number.

\subsubsection{Quadratic message implementation}
\label{subsec:clean_surrogate_message_structure}

The surrogates in Examples~\ref{ex:clean_sur_first} and~\ref{ex:clean_sur_second} are quadratic in \(U_i\). With quadratic initial messages, the updates in Algorithm~\ref{alg:clean_sur_dmp} preserve this form, so each message is represented by a vector and a Hessian matrix. Concretely, each message can be written as
\begin{equation}
    \widetilde\mu_{i\to e}^{\nu}(Y_e)
    \sim
    \frac12\inner{Y_e,H_{i\to e}^{\nu}Y_e}
    +\inner{P_{i\to e}^{\nu},Y_e},
    \quad H_{i\to e}^{\nu}\in\mathbb S_+^d,\quad P_{i\to e}^{\nu}\in\R^d.
\label{eq:clean_sur_message_quadratic}
\end{equation}
For \(h,h'\in\cA_i\), \(\nabla^2_{Y_hY_{h'}}\widetilde g_i=\mathsf B_{i,h}\mathsf B_{i,h'}M_i\) in Example~\ref{ex:clean_sur_first}, and \(\nabla^2_{Y_hY_{h'}}\widetilde g_i=\mathsf B_{i,h}\mathsf B_{i,h'}H_i(Z_i^\nu)\) in Example~\ref{ex:clean_sur_second}.  These Hessians can be dense, and so can the message Hessians \(H_{i\to e}^{\nu}\).

\smallskip
\noindent\textbf{Stabilization of message Hessians.}\ In Example~\ref{ex:clean_sur_first}, the Hessian recursion does not depend on the surrogate centers.  On the tree cluster, it is time-independent and stabilizes after finitely many sweeps.  After this stabilization phase, the Hessian matrices no longer need to be transmitted; subsequent iterations update and send only the vectors \(P_{i\to e}^{\nu}\).

\smallskip
\noindent\textbf{Scalar-metric messages.}\ The scalar metric even avoids matrix-valued Hessian communication entirely.  Take \(M_i=\alpha_i^{-1}I_d\) with \(\alpha_i>0\), and initialize the message Hessians as scalar multiples of \(I_d\). Then \(\nabla^2_{Y_hY_{h'}}\widetilde g_i =\alpha_i^{-1}\mathsf B_{i,h}\mathsf B_{i,h'}I_d\), and Schur complements preserve this isotropic form.  Therefore every message has the scalar-quadratic form
\[
    \widetilde\mu_{i\to e}^{\nu}(Y_e)
    \sim
    \frac{h_{i\to e}^{\nu}}{2}\|Y_e\|^2+
      \inner{P_{i\to e}^{\nu}, Y_e},
    \quad
    h_{i\to e}^{\nu}\in\R_+,\quad P_{i\to e}^{\nu}\in\R^d .
\]
Only one scalar Hessian coefficient and one \(d\)-vector need to be transmitted. 

\smallskip
\noindent\textbf{Slice-wise Hessian communication.}\ The communication bottleneck is dense Hessian transmission.  In both Examples~\ref{ex:clean_sur_first} and~\ref{ex:clean_sur_second}, \(H_{i\to e}^{\nu}\) in \eqref{eq:clean_sur_message_quadratic} can be a full \(d\times d\) matrix.  In Example~\ref{ex:clean_sur_first}, full-matrix transmission is confined to a finite Hessian-stabilization phase.  In contrast, Example~\ref{ex:clean_sur_second} has no such finite stabilization guarantee: \(H_i(Z_i^\nu)\) varies with the center, so the message Hessians can remain dense and time-varying.
To avoid this repeated full-matrix transmission in practice, we propose the following slice-wise Hessian communication strategy, in which only a slice of the Hessian is sent at each round.
For center-dependent metrics, this is a heuristic communication
strategy whose convergence is not analyzed in this paper.

Under the quadratic message representation~\eqref{eq:clean_sur_message_quadratic}, the vector \(P_{i\to e}^{\nu}\) is sent every round, while the receiving node maintains a Hessian cache \(\mathsf H_{i\to e}^{\nu}\).  Decompose \(\mathbb S^d\) as the direct sum \(\mathbb S^d=\bigoplus_{k=1}^K\mathbb S_k\), where each \(\mathbb S_k\) is chosen to admit a low-cost parametrization for communication and \(\Pi_k\) denotes the projection onto \(\mathbb S_k\).  Examples include diagonal bands, matrix rows or columns, and low-rank subspaces. When node \(i\) computes \(\widetilde\mu_{i\to e}^{\nu+1}\) using~\eqref{eq:clean_sur_factor_to_var}, each incoming message \(\widetilde\mu_{j\to h}^{\nu}\) in the local minimization is evaluated with the current vector \(P_{j\to h}^{\nu}\) and the cached Hessian \(\mathsf H_{j\to h}^{\nu}\).  After this minimization, node \(i\) has computed the full Hessian \(H_{i\to e}^{\nu+1}\) of \(\widetilde\mu_{i\to e}^{\nu+1}\), but sends only the slice indexed by \(s(\nu)\in[K]\).  The receiving node then updates only that slice of its cache:
\[
\Pi_{s(\nu)}\mathsf H_{i\to e}^{\nu+1}
=\Pi_{s(\nu)}H_{i\to e}^{\nu+1},
\quad
\Pi_k\mathsf H_{i\to e}^{\nu+1}
=\Pi_k\mathsf H_{i\to e}^{\nu}\quad(k\ne s(\nu)).
\]
Thus each round sends one slice of each dense message Hessian, while the receiving node keeps the remaining slices from its Hessian cache.  For the fixed-metric surrogate~\eqref{eq:clean_sur_first_order}, if the updates are well defined and each slice is used at least once in every \(S\) consecutive rounds, all Hessian caches in cluster \(r\) equal the stabilized message Hessians after at most
\[
    S\left(\left\lceil\frac{\operatorname{diam}(\cF_r)}{2}\right\rceil+1\right)\quad \text{rounds.}
\]
  
Note that the cached Hessians may be indefinite.  In practice, one can add \(\delta I_d\) to each cache, with \(\delta>0\) chosen to keep the local subproblem well defined. \vspace{-0.2cm}

\subsection{Convergence analysis}
\label{subsec:clean_surrogate_linear}

We establish linear convergence of Algorithm~\ref{alg:clean_sur_dmp}. Similarly to the    analysis of Algorithm~\ref{alg:gate_uniform_proposed} (cf.   Sec.~\ref{sec:clean_convergence}), we proceed in two steps. \textbf{Step~1} expresses Algorithm~\ref{alg:clean_sur_dmp} as a delayed block-Jacobi scheme with delayed exterior coordinates and surrogate centers; \textbf{Step~2} establishes convergence of this scheme. The main difference from the exact case is that the block minimizers also depend on surrogate centers, requiring an additional perturbation analysis with respect to these centers. 

\subsubsection{Step~1: Equivalence to a delayed surrogate block-Jacobi scheme}
Fix a block \(r\).  For \(e\in\cE_r\), extend the delays in~\eqref{eq:clean_dual_delay} to all \(i\in\cV_r\) by
\begin{equation}
\delta_{e,i}
:=\frac{\dist_{\cF_r}(e,i)-1}{2},
\qquad e\in\cE_r,\quad i\in\cV_r,
\label{eq:clean_sur_deltaei}
\end{equation}
\(\widetilde D_r:=\max_{e\in\cE_r,\,i\in\cV_r}\delta_{e,i}=\operatorname{diam}(\cT_r)-1\), and \(\widetilde D:=\max_{r\in[p]}\widetilde D_r\). Unlike \(D_r\), which considers only boundary agents, \(\widetilde D_r\) includes delays from all agents of the block.

We expand the messages used in \texttt{step~1} of Algorithm~\ref{alg:clean_sur_dmp} through the recursion in \texttt{step~2}, obtaining the following equivalence analogous to Proposition~\ref{prop:clean_delayed_bilevel}. The proof is omitted.
\begin{lemma}
\label{lem:clean_sur_delayed}
Under Assumptions~\ref{ass:clean_tree_delay} and~\ref{ass:clean_surrogate_nodewise}, for every \(r\in[p]\), \(e\in\cE_r\), and \(\nu\ge\max_{s\in[p]}\operatorname{diam}(\cT_s)\), the updates of Algorithm~\ref{alg:clean_sur_dmp} read
\begin{equation}
\begin{aligned}
Y_e^{\nu+1}
&=Y_e^\nu+\tau_r^\nu(\hat Y_e^{\nu+1}-Y_e^\nu),\\
\hat Y_e^{\nu+1}
&=
\argmin_{Y_e}\ \min_{Y_{\cE_r\setminus\{e\}}}
\sum_{i\in\cV_r}
\widetilde f_i^*\!\left(
\mathsf B_{i,\cA_i^\circ}Y_{\cA_i^\circ}
+\sum_{h\in\cA_i^\partial}\mathsf B_{i,h}Y_h^{\nu-\delta_{e,i}};
Z_i^{\nu-\delta_{e,i}}
\right).
\end{aligned}
\label{eq:clean_sur_delayed_block}
\end{equation}
\end{lemma}

\subsubsection{Step~2: Linear convergence of the delayed scheme}
The remaining  convergence proof combines three ingredients: \textbf{(i)} a descent estimate for \(G\) along the iterates, \textbf{(ii)} the surrogate block error bound, and \textbf{(iii)} a perturbation analysis of the block minimizers, bounding the mismatch caused by delays. The arguments for \textbf{(i)} and \textbf{(ii)} parallel the exact case and are given in Appendix~\ref{subsec:app_sur_descent}. We focus here on \textbf{(iii)}: the block minimizers in Lemma~\ref{lem:clean_sur_delayed} also depend on delayed surrogate centers, requiring additional control of center perturbations.

\smallskip\noindent\textbf{Perturbation analysis with respect to exterior parameters and surrogate centers.}
For exterior parameters \(E\in\R^{|\widehat\partial_r|\times d}\), with rows indexed by \((i,h)\in\widehat\partial_r\), and surrogate centers \(\vartheta\in\R^{|\cV_r|\times d}\), define the block solution map
\[
\widetilde Y_r(E,\vartheta)
=
\argmin_{Y_{\cE_r}}
\sum_{i\in\cV_r}
\widetilde f_i^*\!\left(
\mathsf B_{i,\cA_i^\circ}Y_{\cA_i^\circ}
+
\sum_{h\in\cA_i^\partial}\mathsf B_{i,h}E_{i,h};
\vartheta_i
\right).
\]
We now introduce the frozen-center counterparts of \(\kappa_{\partial,r}\). For fixed centers \(\vartheta\), let \(\widetilde\chi_r(\cdot;\vartheta)\) be defined by
\[
\sum_{i\in\cV_r}
\bigl(\nabla_u\widetilde f_i^*(\cdot;\vartheta_i)\bigr)^{-1}
(\widetilde\chi_r(\eta;\vartheta))
=\eta,
\]
where \(\widetilde\chi_r(\cdot;\vartheta)\) is well defined by strong monotonicity.  For nonempty \(S\subseteq\cV_r\), set
\begin{equation}
\widetilde\cL_{r,S}(\eta;\vartheta)
:=
\sum_{i\in S}
\bigl(\nabla_u\widetilde f_i^*(\cdot;\vartheta_i)\bigr)^{-1}
(\widetilde\chi_r(\eta;\vartheta)).
\label{eq:clean_sur_aggregate_map}
\end{equation}
Using the same family of vertex subsets \(\cR_r\) as in~\eqref{eq:clean_kappa_partial}, define
\begin{equation}
\widetilde\kappa_{\partial,r}
:=
\begin{cases}
    0,
    & \partial_r=\emptyset,\\[1mm]
    \displaystyle
    \sup_{\vartheta\in\R^{|\cV_r|\times d}}
    \max_{S\in\cR_r}
    \frac{|\cV_r|}{|S|}
    \operatorname{Lip}\bigl(\widetilde\cL_{r,S}(\cdot;\vartheta)\bigr),
    & \partial_r\ne\emptyset.
\end{cases}
\label{eq:clean_sur_kappa_partial}
\end{equation}
The supremum over \(\vartheta\) makes the bounds for exterior perturbations uniform in the surrogate centers.

For a map \(F(\eta;\vartheta)\), let \(\operatorname{Lip}_{\vartheta_i}F(\eta;\vartheta)\) denote its Lipschitz modulus in the \(i\)-th center coordinate of \(\vartheta\), with \(\eta\) and the other center coordinates fixed.  Define the normalized center-coordinate modulus
\begin{equation}
\begin{aligned}
\widetilde\kappa_{c,r}
:=
\sup_{\substack{\eta\in\R^d,\;\vartheta\in\R^{|\cV_r|\times d}}}
\max_{i\in\cV_r}
\max\bigg\{&
\max_{e\in\cE_r}
\frac{|\cV_r|}{|\overline{\cV}_e(i)|}
\operatorname{Lip}_{\vartheta_i}
\bigl(\widetilde\cL_{r,\overline{\cV}_e(i)}(\eta;\vartheta)\bigr),
\\
&\hspace{-5em}
\operatorname{Lip}_{\vartheta_i}
\bigl(\widetilde\cL_{r,\{i\}}(\eta;\vartheta)\bigr),\;
\max_{v\in\cV_r\setminus\{i\}}
|\cV_r|\operatorname{Lip}_{\vartheta_i}
\bigl(\widetilde\cL_{r,\{v\}}(\eta;\vartheta)\bigr)
\bigg\}.
\end{aligned}
\label{eq:clean_sur_kappa_center}
\end{equation}
Assumption~\ref{ass:clean_surrogate_nodewise}(iii)--(iv) guarantees that these moduli are finite, with \(\widetilde\kappa_{c,r}=0\) for exact surrogates; see Appendix~\ref{subsec:app_sur_tree_center} for their interpretation.
These moduli give the following sensitivity bounds.
\begin{lemma}
\label{lem:clean_surrogate_tree_sensitivity}
Suppose Assumptions~\ref{ass:clean_tree_delay} and \ref{ass:clean_surrogate_nodewise} hold.  For every block \(r\), the surrogate block solution map \(\widetilde Y_r(\cdot,\cdot)\) is single-valued.

\smallskip\noindent\textnormal{\emph{(I) Sensitivity to exterior parameters.}} If \(E,E'\in\R^{|\widehat\partial_r|\times d}\) differ only in row \((i,h)\in\widehat\partial_r\), then, for any center \(\vartheta\) and edge \(e\in\cE_r\),
\begin{subequations}
\begin{align}
\|[\widetilde Y_r(E',\vartheta)-\widetilde Y_r(E,\vartheta)]_e\|
&\le
\widetilde\kappa_{\partial,r}\frac{|\overline{\cV}_e(i)|}{|\cV_r|}\|E'_{i,h}-E_{i,h}\|,
\label{eq:clean_sur_boundary_edge_sens}\\
\|\mathsf B_{v,\cE_r}(\widetilde Y_r(E',\vartheta)-\widetilde Y_r(E,\vartheta))\|
&\le
\widetilde\kappa_{\partial,r}\frac{1}{|\cV_r|}\|E'_{i,h}-E_{i,h}\|,
\quad v\ne i .
\label{eq:clean_sur_boundary_node_sens}
\end{align}
\end{subequations}
\smallskip\noindent\textnormal{\emph{(II) Sensitivity to surrogate centers.}} If two centers \(\vartheta,\widehat\vartheta\) differ only in coordinate \(i\), then, for any exterior parameter \(E\) and edge \(e\in\cE_r\),
\begin{subequations}
\begin{align}
\|[\widetilde Y_r(E,\widehat\vartheta)-\widetilde Y_r(E,\vartheta)]_e\|
&\le
\widetilde\kappa_{c,r}\frac{|\overline{\cV}_e(i)|}{|\cV_r|}
\|\widehat\vartheta_i-\vartheta_i\|,
\label{eq:clean_sur_center_edge_sens}\\
\|\mathsf B_{i,\cE_r}(\widetilde Y_r(E,\widehat\vartheta)-\widetilde Y_r(E,\vartheta))\|
&\le
\widetilde\kappa_{c,r}\|\widehat\vartheta_i-\vartheta_i\|,
\label{eq:clean_sur_center_node_sens_local}\\
\|\mathsf B_{v,\cE_r}(\widetilde Y_r(E,\widehat\vartheta)-\widetilde Y_r(E,\vartheta))\|
&\le
\frac{\widetilde\kappa_{c,r}}{|\cV_r|}\|\widehat\vartheta_i-\vartheta_i\|,
\quad v\ne i .
\label{eq:clean_sur_center_node_sens_nonlocal}
\end{align}
\end{subequations}
\end{lemma}
\begin{proof}
See Appendix~\ref{subsec:app_sur_tree_center}. \qed
\end{proof}

\smallskip\noindent\textbf{Bounding the mismatch caused by delays.}
As in Sec.~\ref{subsec:clean_delayed_block_jacobi_convergence}, we bound the mismatch between the image contribution \(\hat Z_r^{\nu+1}\) of the algorithm's delayed edge responses and the contribution \(\bar Z_r^{\nu+1}\) obtained by surrogate block minimization at the current exterior coordinates and surrogate centers.
For this comparison, write \(Z^\nu:=\mathsf BY^\nu\), with \(Z_i^\nu\) its \(i\)-th row, and define, for \(r\in[p]\),
\begin{equation}
\begin{aligned}
Z_r^\nu&:=\mathsf B_{\cV,\cE_r}Y_{\cE_r}^\nu,\\
\hat Z_r^{\nu+1}&:=\mathsf B_{\cV,\cE_r}\hat Y_{\cE_r}^{\nu+1},\\
\bar Z_r^{\nu+1}&:=\widetilde O_r(Z^\nu)-Z^\nu+Z_r^\nu.
\end{aligned}
\label{eq:clean_sur_block_images}
\end{equation}

Fix \(r\in[p]\) and \(\nu\ge\max_{s\in[p]}\operatorname{diam}(\cT_s)\), and let \(Z_{\cV_r}^\nu\) denote the restriction of \(Z^\nu\) to \(\cV_r\). For \(e\in\cE_r\), collect the current and delayed exterior parameters in \(E^\nu,\widehat E_e^\nu\in\R^{|\widehat\partial_r|\times d}\), and the delayed centers in \(Z_{r,e}^\nu\):
\[
\begin{aligned}
(E^\nu)_{i,h}&:=Y_h^\nu,\qquad
(\widehat E_e^\nu)_{i,h}:=Y_h^{\nu-\delta_{e,i}},\quad (i,h)\in\widehat\partial_r,\\
Z_{r,e}^\nu&:=(Z_i^{\nu-\delta_{e,i}})_{i\in\cV_r}.
\end{aligned}
\]
By~\eqref{eq:clean_sur_block_images}, the definition of \(\widetilde O_r\), and Lemma~\ref{lem:clean_sur_delayed},
\[
\begin{aligned}
\bar Z_r^{\nu+1}&=\mathsf B_{\cV,\cE_r}\widetilde Y_r(E^\nu,Z_{\cV_r}^\nu),\\
\hat Y_e^{\nu+1}&=[\widetilde Y_r(\widehat E_e^\nu,Z_{r,e}^\nu)]_e.
\end{aligned}
\]
These sensitivity bounds yield the following image-space mismatch estimate. Compared with Lemma~\ref{lem:clean_tree_mismatch_z}, the bound includes additional contributions involving \(\widetilde\kappa_{c,r}^2\) arising from delayed surrogate centers.
\begin{lemma}
\label{lem:clean_surrogate_tree_mismatch}
Under Assumptions~\ref{ass:clean_tree_delay} and \ref{ass:clean_surrogate_nodewise}, for each block \(r\), define
\[
\begin{aligned}
\widetilde M_{r,\partial}
&:=
\Gamma_r
\left(
8\widetilde\kappa_{\partial,r}^2|\widehat\partial_r|
+
8|\partial_r|(|\cV_r|+\widetilde D_r)\widetilde\kappa_{c,r}^2
\right),\\
\widetilde M_{r,c}
&:=
8(|\cV_r|+\widetilde D_r)\widetilde\kappa_{c,r}^2.
\end{aligned}
\]
Then, for every iteration \(\nu\ge\max_{r\in[p]}\operatorname{diam}(\cT_r)\),
\begin{equation}
\|\hat Z_r^{\nu+1}-\bar Z_r^{\nu+1}\|^2
\le
\widetilde M_{r,\partial}
\sum_{\ell=\nu-\widetilde D_r}^{\nu-1}
\sum_{s\in\cM_r}
\|Z_s^{\ell+1}-Z_s^\ell\|^2
+
\widetilde M_{r,c}
\sum_{\ell=\nu-\widetilde D_r}^{\nu-1}
\|Z_r^{\ell+1}-Z_r^\ell\|^2.
\label{eq:clean_sur_boundary_center_mismatch}
\end{equation}
The sums over \(\ell\) are understood to be zero when \(\widetilde D_r=0\).
\end{lemma}
\begin{proof}
See Appendix~\ref{subsec:app_sur_boundary_center_mismatch}. \qed
\end{proof}
\smallskip\noindent\textbf{Linear convergence.}
For each block \(r\), define $\widetilde\mu_r^*:=\min_{i\in\cV_r}\widetilde\mu_i^*,\;\widetilde L_r^*:=\max_{i\in\cV_r}\widetilde L_i^*.$
Using this mismatch estimate, we obtain the following convergence theorem.

\begin{theorem}
\label{thm:clean_surrogate_linear}
Suppose Assumptions~\ref{ass:clean_problem_class}, \ref{ass:clean_tree_delay}, \ref{ass:clean_surrogate_nodewise}, and~\ref{ass:clean_sur_beb} hold, with \(p>1\).  Define \(\underline{\widetilde\mu}^*:=\min_{r\in[p]}\widetilde\mu_r^*\) and
\begin{equation}
\widetilde A
:=
\max_{s\in[p]}
\left[
\sum_{r\in[p]:\,\widetilde D_r>0,\,s\in\cM_r}
\frac{2\widetilde L_r^*+\widetilde\mu_r^*}{4}\,\widetilde M_{r,\partial}
+
\mathbf 1_{\{\widetilde D_s>0\}}
\frac{2\widetilde L_s^*+\widetilde\mu_s^*}{4}\,\widetilde M_{s,c}
\right].
\label{eq:clean_sur_A}
\end{equation}
Let \(X^\star=\mathbf 1(x^\star)^\top\) be the consensus primal optimizer. If \(\tau_r^\nu\equiv\tau\) and
\begin{equation}
0<\tau\le\min\left\{\frac{1}{\omega_{\cP}},\,2\widetilde\kappa_{\cP}\right\},\quad
\frac{16\widetilde A\widetilde\kappa_{\cP}\tau}{\underline{\widetilde\mu}^*}
\left[\left(1-\frac{\tau}{2\widetilde\kappa_{\cP}}\right)^{-\widetilde D}-1\right]\le1,
\label{eq:clean_sur_stepsize}
\end{equation}
then, for some constant \(C\in(0,\infty)\) and all \(\nu\ge\widetilde D+1\),
\begin{equation}
    \|X^\nu-X^\star\|^2
    \le
    C\widetilde\rho^\nu,
    \quad
    \widetilde\rho:=1-\frac{\tau}{2\widetilde\kappa_{\cP}}.
\label{eq:clean_sur_primal_rate}
\end{equation}
\end{theorem}

\begin{proof}
See Appendix~\ref{subsec:app_sur_descent}. \qed
\end{proof}

\begin{remark}
Theorem~\ref{thm:clean_surrogate_linear} parallels Theorem~\ref{thm:clean_linear_tree}: Algorithm~\ref{alg:clean_sur_dmp} converges linearly, with rate \(\widetilde\rho\) now given by~\eqref{eq:clean_sur_primal_rate}. The stepsize restriction~\eqref{eq:clean_sur_stepsize} has the same form as~\eqref{eq:clean_tree_stepsize}. The main differences are that the exact quantities \((1/L_r,1/\mu_r,D,A)\) are replaced by \((\widetilde\mu_r^*,\widetilde L_r^*,\widetilde D,\widetilde A)\), and that \(\kappa_{\cP}\) is replaced by \(\widetilde\kappa_{\cP}\), which may be larger when the surrogate captures less exact block progress. The resulting deviation from the exact constants is the rate price for cheaper local computation and communication. For the exact surrogate, Theorem~\ref{thm:clean_surrogate_linear} recovers the rate factor of Theorem~\ref{thm:clean_linear_tree} at the same admissible stepsize.

As in~\eqref{eq:clean_tree_sufficient_stepsize}, an explicit sufficient stepsize is
\[
0<\tau\le\min\left\{
\frac{1}{\omega_{\cP}},\
\frac{2\widetilde\kappa_{\cP}}{2\widetilde D+1},\
\sqrt{\frac{\underline{\widetilde\mu}^*}
{8(2\widetilde D+1)\widetilde A}}
\right\}.
\]
The proof is analogous to that of Corollary~\ref{cor:clean_tree_explicit_stepsize} and is omitted.
\par\vspace{-0.4cm}
\end{remark}

\subsection{Comparison with consensus-based methods}
\label{subsec:clean_consensus_comparison}

We now compare the convergence rate of Algorithm~\ref{alg:clean_sur_dmp} with those of consensus-based methods, such as DGD~\cite{Yuan:2016dv}, EXTRA~\cite{shi2015extra}, and SONATA~\cite{sun2022distributed}, in terms of network-size scaling.  We show that, for a class of tree-dominant graphs, sufficiently accurate surrogates yield an acceleration in network-size scaling relative to these consensus baselines. On general graphs, our scheme achieves a rate no worse than the standard spectral-gap bound for the corresponding dual-gradient method when a suitable proximal-linear surrogate is chosen.

\smallskip
\noindent\textbf{Acceleration in tree-dominant graphs.} Consider a tree-dominant graph, and choose  a partition \(\cP\) that has a dominant tree block and only \(O(1)\) remaining singleton edge blocks. The analysis in Sec.~\ref{subsec:clean_rate_examples} gives \(\kappa_{\cP}=O(1)\), \(A=O(1)\).  Theorem~\ref{thm:clean_linear_tree} then yields \((1-\rho)^{-1}=O(D/\log D)\).  For the ring and barbell examples, the exact inverse rate gap is \(O(m/\log m)\).  In contrast, consensus baselines have \(O(m^2)\) rate scaling on these graphs through the inverse spectral-gap factor \(1/\lambda_{\min}^+(\mathsf B\mathsf B^\top)\).  The exact \(O(m/\log m)\) inverse rate gap therefore gives network-size acceleration over those baselines.

The constant \(\widetilde\kappa_{\cP}\) preserves the network-size scaling of \(\kappa_{\cP}\) whenever the surrogate block progress captures the exact block progress up to a constant fraction.  Indeed, if, for some network-independent \(\eta\in(0,1]\), \(\widetilde\Delta_r(Z)\ge \eta\,\Delta_r(Z)\) holds for all \(Z\in\cS\) and \(r\in[p]\), then \(\kappa_{\cP}\le\widetilde\kappa_{\cP}\le\kappa_{\cP}/\eta\). For example, once a quadratic surrogate satisfies the Hessian sandwich inequality~\eqref{eq:clean_fixed_majorant_curvature}, it gives \(\widetilde\kappa_{\cP}\le\rho_{\max}\kappa_{\cP}\) with network-independent \(\rho_{\max}\).  Consequently, for these fixed-metric quadratic surrogates on the tree-dominant families and partitions above, \(\widetilde\kappa_{\cP}=O(1)\) and Lemma~\ref{lem:clean_surrogate_tree_mismatch} gives \(\widetilde A=O(m)\). Theorem~\ref{thm:clean_surrogate_linear} then gives \((1-\widetilde\rho)^{-1}=O(\sqrt{m\widetilde D})\), for \(\widetilde D\ge1\). For the ring and barbell examples, where \(\widetilde D=\Theta(m)\), this yields an \(O(m)\) inverse rate gap, which matches the exact method's rate scaling up to a logarithmic factor. 

\smallskip
\noindent\textbf{Consensus scaling recovery.} On general graphs, consider a direct proximal-linear surrogate for \(g_i(Y_{\cA_i})=f_i^*(\mathsf B_{i,\cA_i}Y_{\cA_i})\), outside the structured subclass~\eqref{eq:clean_sur_factor_structured}. Specifically, choose \(c_{i,e}>0\) so that the following surrogate majorizes \(g_i\) at every center:
\begin{equation}
\widetilde g_i(Y_{\cA_i};Y_{\cA_i}^\nu)
:=
g_i(Y_{\cA_i}^\nu)
+
\sum_{e\in\cA_i}
\inner{
\nabla_{Y_e}g_i(Y_{\cA_i}^\nu),\,Y_e-Y_e^\nu
}
+
\frac12\sum_{e\in\cA_i}c_{i,e}\|Y_e-Y_e^\nu\|^2 .
\label{eq:clean_edge_separable_first_order_factor}
\end{equation}
Since \eqref{eq:clean_edge_separable_first_order_factor} is separable in \(\{Y_e:e\in\cA_i\}\), the message update satisfies
\begin{equation*}
\widetilde\mu_{i\to e}^{\nu+1}(Y_e)
\sim
\inner{
\mathsf B_{i,e}\nabla f_i^*((\mathsf BY^{\nu+1})_i),\,Y_e-Y_e^{\nu+1}
}
+
\frac{c_{i,e}}{2}\|Y_e-Y_e^{\nu+1}\|^2 .
\end{equation*}
For \(\nu\ge1\) and \(e\in\cE_r\) with \(\cI_e=\{i,j\}\), the edge update is the dual-gradient step

\[
\widehat Y_e^{\nu+1}
=Y_e^\nu-
\frac{
\mathsf B_{i,e}\nabla f_i^*(Z_i^\nu)
+\mathsf B_{j,e}\nabla f_j^*(Z_j^\nu)
}{c_{i,e}+c_{j,e}},
\]
followed by \(Y_e^{\nu+1}=Y_e^\nu+\tau_r^\nu(\widehat Y_e^{\nu+1}-Y_e^\nu)\).
For this direct surrogate, the block progress depends on \(Y\) only through \(Z=\mathsf BY\):
\[
    \widetilde\Delta_r(Z)
    :=
    \sum_{e\in\cE_r}
    \frac{\|\mathsf B_{\cV,e}^{\top}\nabla G(Z)^{\top}\|^2}{2(c_{i,e}+c_{j,e})},
    \qquad Z\in\cS,\quad \cI_e=\{i,j\}.
\]
Consequently, the smallest constant in the corresponding block error bound is
\[
    \widetilde\kappa_{\cP}
    =
    \sup_{Z\in\cS:\,G(Z)>G^\star}
    \frac{G(Z)-G^\star}
    {\frac12\sum_{e\in\cE}
    (c_{i,e}+c_{j,e})^{-1}
    \|\mathsf B_{\cV,e}^{\top}\nabla G(Z)^{\top}\|^2},
\]
and is independent of the partition \(\cP\). Strong convexity of \(G\) on \(\cS\) gives
\[
    G(Z)-G^\star
    \le
    \frac{L}{2\lambda_{\min}^+(\mathsf B\mathsf B^\top)}
    \|\mathsf B^\top\nabla G(Z)^\top\|^2.
\]
If \(L\) and the coefficients \(c_{i,e}+c_{j,e}\) are bounded independently of network size, then
\begin{equation}
    \widetilde\kappa_{\cP}
    =
    O\left(\frac{1}{\lambda_{\min}^+(\mathsf B\mathsf B^\top)}\right),
    \label{eq:clean_direct_sur_consensus_kappa}
\end{equation}
where \(\lambda_{\min}^+\) is the smallest positive eigenvalue.  Thus, this surrogate recovers the usual consensus spectral-gap rate scaling \(O(1/\lambda_{\min}^+(\mathsf B\mathsf B^\top))\), matching classical nonaccelerated consensus-based methods.\vspace{-0.3cm}
  \subsection{Comparison with tree-based aggregation methods}\vspace{-0.2cm}
  \label{subsec:clean_tree_comparison}
   Spanning-tree methods relay iterates~\cite{vogels2021relaysum}
  or aggregate gradients for master-node
  updates~\cite{pmlr-v70-scaman17a}.  
   The tree can be constructed in a decentralized manner,
  by electing a root and selecting paths to
  it~\cite{vogels2021relaysum}.
    However, {\it using a single spanning tree need not be desirable
  in every application}. \textbf{(i)} It provides {\it no route redundancy}:
  a permanent failure of any tree edge or non-leaf agent
  disconnects it, requiring repair to restore communication. \textbf{(ii)}
   Distances along the tree can exceed those in the full graph, lengthening propagation delays.  \textbf{(iii)}
  In master-node schemes, the root is additionally a single
  point of failure for the optimization, and each update
  must wait for all gradient
  contributions~\cite{pmlr-v70-scaman17a}. 

  GATE and GATE-S are significantly more flexible than these schemes. \textbf{(a)}  They permit both a {\it single} spanning-tree block
  and decompositions retaining {\it all} physical edges, without
  a computational master. \textbf{(b)} Their trees define optimization blocks, not merely
  aggregation routes: the decomposition can therefore be
  adapted jointly to the agents' functions and the network,
  exploiting the block-preconditioning effect quantified
  by the analysis. \textbf{(c)} Retaining edges also provides reconfiguration options.
  If the physical graph remains connected after a link
  failure, deleting the failed edge and splitting its block
  into the remaining nonempty tree blocks preserves a valid
  decomposition without changing the other blocks.
  The retained variables can initialize a new run under
  the theorem assumptions for the new partition.

   Even within GATE and GATE-S, {\it a single spanning-tree block need not be preferable}, albeit being an option. For GATE with exact messages, it gives the finite-round guarantee of Remark~\ref{rem:clean_spanning_tree}; partitioning into smaller blocks is still convenient when some agents have high degree, since each message minimization at an agent involves all its other edges in the block. When the messages admit no closed form, GATE-S uses local surrogates to obtain tractable messages; a single block $r$  then uses centers up to  \(\widetilde D_r=\operatorname{diam}(\cT_r)-1\) rounds old, a delay that affects the rate. Smaller-diameter blocks shorten these delays and can permit larger stepsizes and faster guaranteed convergence, provided the added inter-block coupling and overlap remain controlled (Theorem~\ref{thm:clean_surrogate_linear}
).
 
\vspace{-0.3cm}
\section{Numerical experiments}
\label{sec:clean_numerics}
\vspace{-0.2cm}
In this section, we evaluate GATE and its surrogate variants on strongly convex quadratic and regularized logistic regression problems over various graphs, with and without a dominant tree structure. The quadratic instances have the form
\[
    f_i(x)=\frac12x^\top H_i x+q_i^\top x,
    \qquad H_i\succ0.
\]
The non-quadratic instances have the form
\[
    f_i(x)
    =
    \frac1{N_i}\sum_{j=1}^{N_i}
    \log\!\left(1+\exp(-b_{ij}a_{ij}^{\top}x)\right)
    +\frac12x^\top R_i x-h_i^\top x,
    \qquad R_i\succ0.
\]
The consensus baselines use the Metropolis mixing matrix \(W\). Stepsizes for the proposed methods and the consensus baselines are tuned using an ad-hoc grid search on a logarithmic scale, yielding the optimal numerical performance. Performance is measured by the error of the iterates from the solution of the problem, $\left\|X^\nu-\mathbf 1(x^\star)^\top\right\|.$

\subsection{Scaling verification and partition effects}
\label{subsec:clean_numerical_scaling}

In this part, we verify the scaling with network size and dominant-tree diameter and examine the effect of the partition.

\par\addvspace{\bigskipamount}
\noindent\emph{Network-size scaling on rings.}\enspace
We first verify network-size scaling for quadratic problems on ring networks with \(d=20\).  Each \(H_i\) is generated as a random symmetric Gaussian matrix with a positive diagonal correction chosen so that \(\kappa(H_i)=2\). The edge partition consists of one \((m-2)\)-edge path cluster and two singleton-edge clusters. GATE-S1 uses the fixed diagonal metric with entries \((M_i)_{kk}=(H_i^{-1})_{kk}+\sum_{\ell\ne k}|(H_i^{-1})_{k\ell}|\), so \(M_i\succeq H_i^{-1}\). Figure~\ref{fig:clean_ring_qp_scaling}(a) shows that GATE and GATE-S1 are faster  than EXTRA and SONATA, with the gap widening as \(m\) grows.  Figure~\ref{fig:clean_ring_qp_scaling}(b) indicates linear scaling for the GATE variants and quadratic scaling for the consensus baselines, which is consistent with the theoretical comparison in Sec.~\ref{subsec:clean_consensus_comparison}.

\begin{figure}[th]
    \centering
    \setlength{\tabcolsep}{2pt}
    \begin{tabular}{ccc}
        \includegraphics[width=0.31\linewidth]{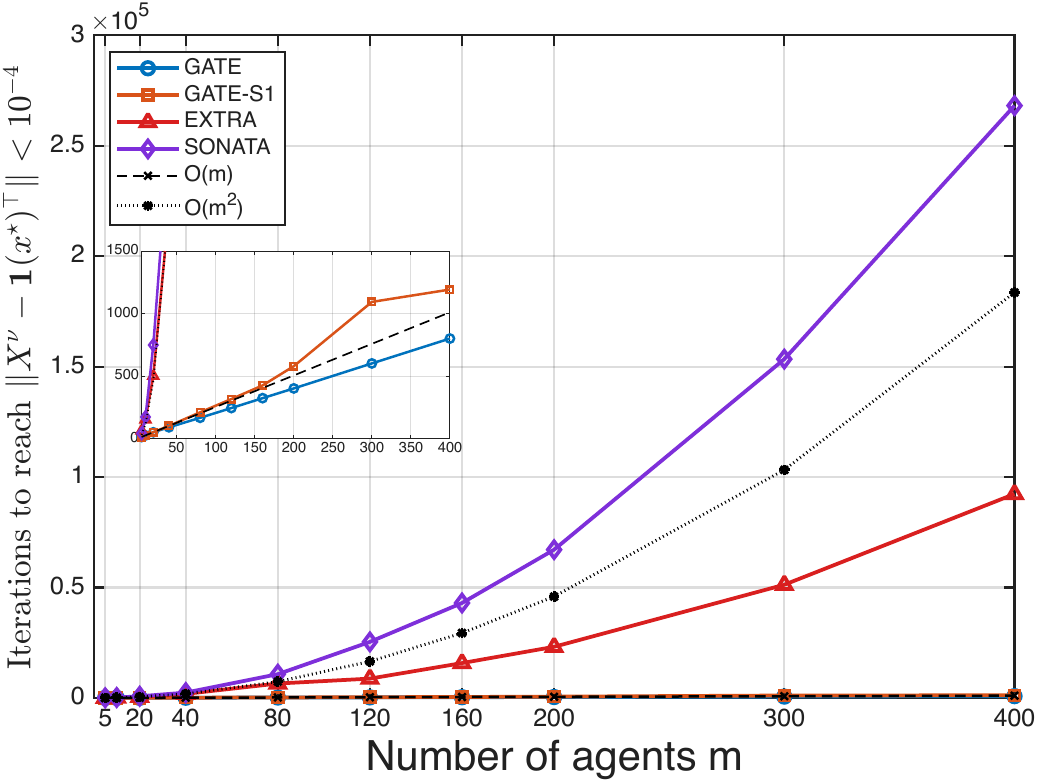}
        &
        \includegraphics[width=0.31\linewidth]{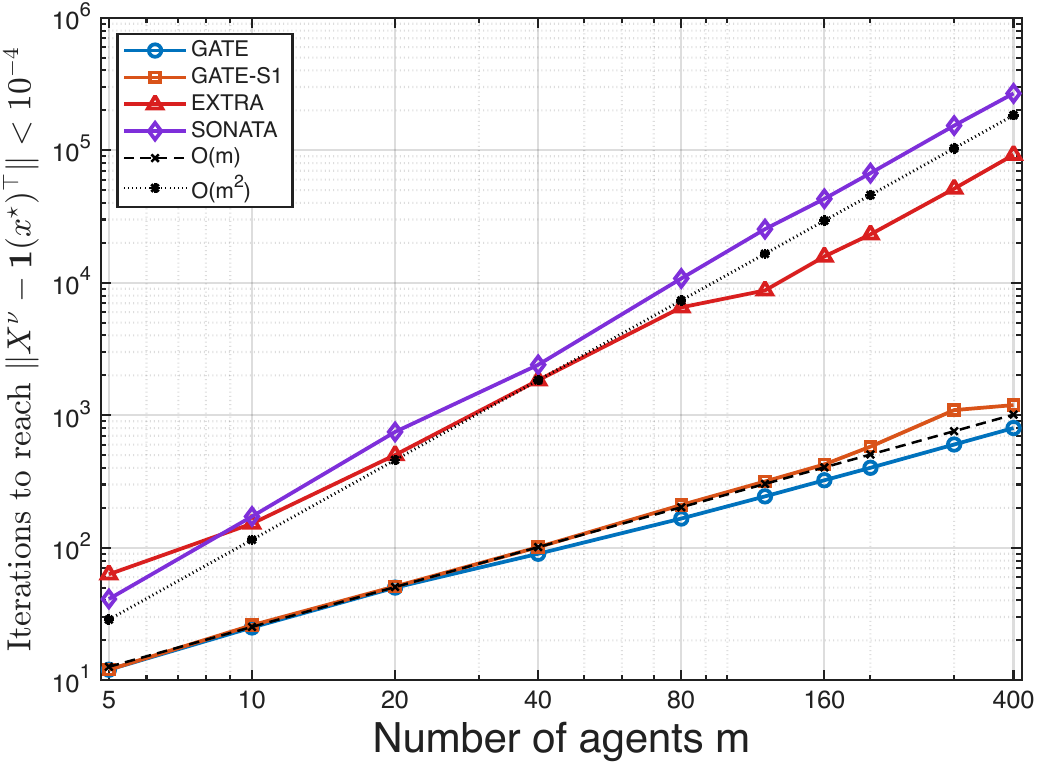}
        &
        \includegraphics[width=0.31\linewidth]{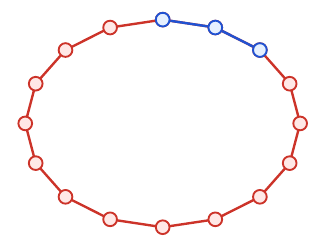}
        \\
        \parbox{0.31\linewidth}{\centering\footnotesize (a) Iteration count.}
        &
        \parbox{0.31\linewidth}{\centering\footnotesize (b) Iteration count (log--log).}
        &
        \parbox{0.31\linewidth}{\centering\footnotesize (c) Graph and partition.}
    \end{tabular}
    \caption{Network-size scaling for quadratic problems on a ring. Panel~(a) reports the iteration counts versus the number of agents \(m\), while panel~(b) presents the same data on log--log axes to make their dependence on \(m\) more apparent.  Panel~(c) shows the graph and edge partition.  The stopping criterion is \(\|X^\nu-\mathbf 1(x^\star)^\top\|<10^{-4}\).}
    \label{fig:clean_ring_qp_scaling}
\end{figure}
To test whether the scaling observed above persists beyond the quadratic setting, we next consider the non-quadratic logistic regression problem, for which neither GATE surrogate is exact.  We use \(d=10\), \(N_i=40\), and the same partition as in the quadratic experiment.  The Hessian condition number is set to five. The quadratic-surrogate metrics are \(\nabla^2 f_i^*(Z_i^\nu)+\delta I_d\) for GATE-S2 and \(\operatorname{Diag}(\nabla^2 f_i^*(Z_i^\nu))+\delta I_d\) for GATE-S1, with \(\delta\) tuned over \(\{0,10^{-6},10^{-4},10^{-2},1\}\) separately for each method. Figure~\ref{fig:clean_ring_logreg_dense_scaling}(a) shows that GATE-S1 and GATE-S2 remain faster than EXTRA and SONATA, with the gap widening as \(m\) grows.  Figure~\ref{fig:clean_ring_logreg_dense_scaling}(b) indicates near-linear scaling for the GATE surrogates, whereas the consensus baselines approach quadratic scaling over the upper range.

\begin{figure}[th]
    \centering
    \setlength{\tabcolsep}{2pt}
    \begin{tabular}{ccc}
        \includegraphics[width=0.31\linewidth]{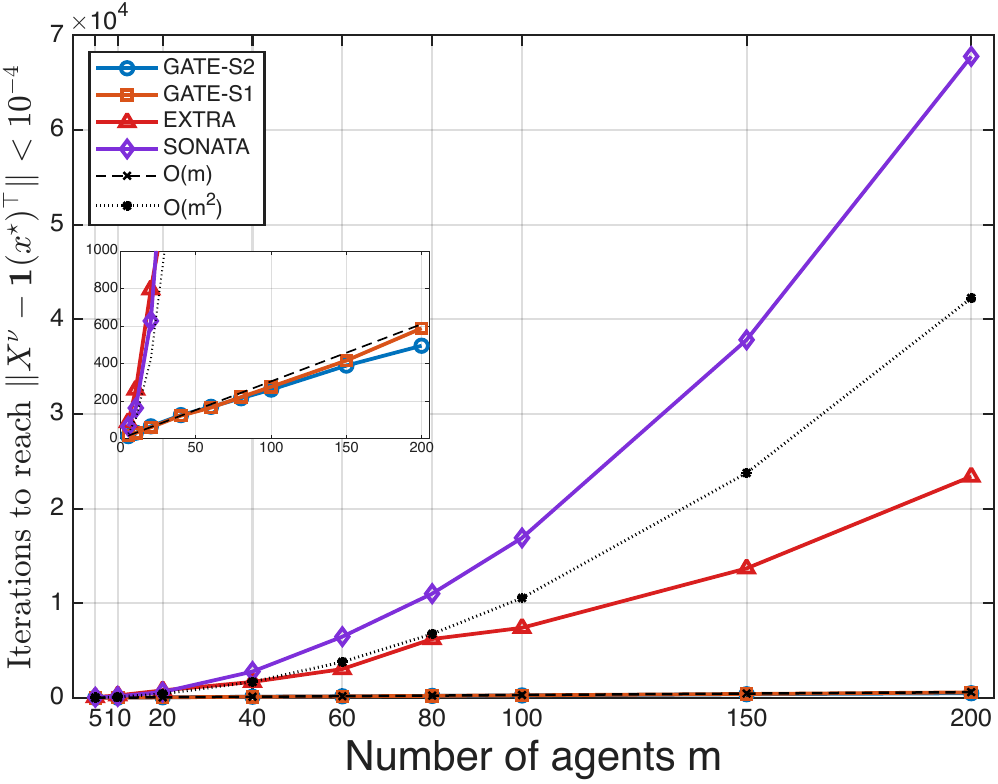}
        &
        \includegraphics[width=0.31\linewidth]{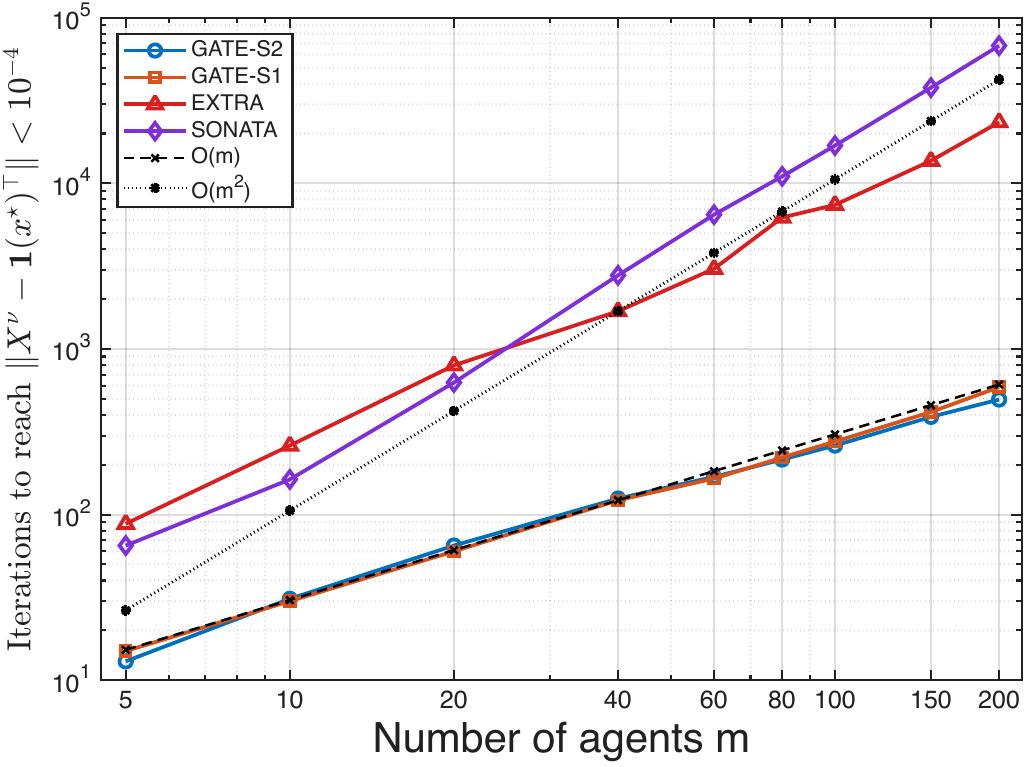}
        &
        \includegraphics[width=0.31\linewidth]{ring_scaling_graph_pattern.pdf}
        \\
        \parbox{0.31\linewidth}{\centering\footnotesize (a) Iteration count.}
        &
        \parbox{0.31\linewidth}{\centering\footnotesize (b) Iteration count (log--log).}
        &
        \parbox{0.31\linewidth}{\centering\footnotesize (c) Graph and partition.}
    \end{tabular}
    \caption{Network-size scaling for logistic-regression problems on a ring. Panel~(a) reports the iteration counts versus the number of agents \(m\), while panel~(b) presents the same data on log--log axes to make their dependence on \(m\) more apparent.  Panel~(c) shows the graph and edge partition.  The stopping criterion is \(\|X^\nu-\mathbf 1(x^\star)^\top\|<10^{-4}\).}
    \label{fig:clean_ring_logreg_dense_scaling}
\end{figure}

\par\addvspace{\bigskipamount}
\noindent\emph{Dependence on the dominant-tree diameter.}\enspace
We next study the dependence of the convergence rate on \(D\) for a quadratic problem over a non-ring graph.  The graph consists of two three-node cliques connected by a growing non-path tree, with the partition containing one dominant tree cluster and six singleton-edge clusters.  We set \(H_i=H_0\in\mathbb S_{++}^5\) for all \(i\), where \(H_0\) is dense and \(\kappa(H_0)=2\).  For GATE-S1, we use a diagonal metric with entries \((M_i)_{kk}=(H_0^{-1})_{kk}+\sum_{\ell\ne k}|(H_0^{-1})_{k\ell}|\), so \(M_i\succeq H_0^{-1}\).

With \(m\) fixed, Figure~\ref{fig:clean_nonring_comb_qp_scaling}(b) indicates the \(O(D)\) and \(O(\sqrt{Dm})\) rate scalings for GATE and GATE-S1, respectively, consistent with Theorems~\ref{thm:clean_linear_tree} and~\ref{thm:clean_surrogate_linear}. Figure~\ref{fig:clean_nonring_comb_qp_scaling}(a) then considers \(D=\Theta(\sqrt m)\) as \(m\) varies and exhibits the corresponding \(m^{1/2}\) and \(m^{3/4}\) scalings for GATE and GATE-S1, respectively, further confirming the theoretical scaling.

\begin{figure}[th]
    \centering
    \setlength{\tabcolsep}{2pt}
    \begin{tabular}{ccc}
        \includegraphics[width=0.31\linewidth]{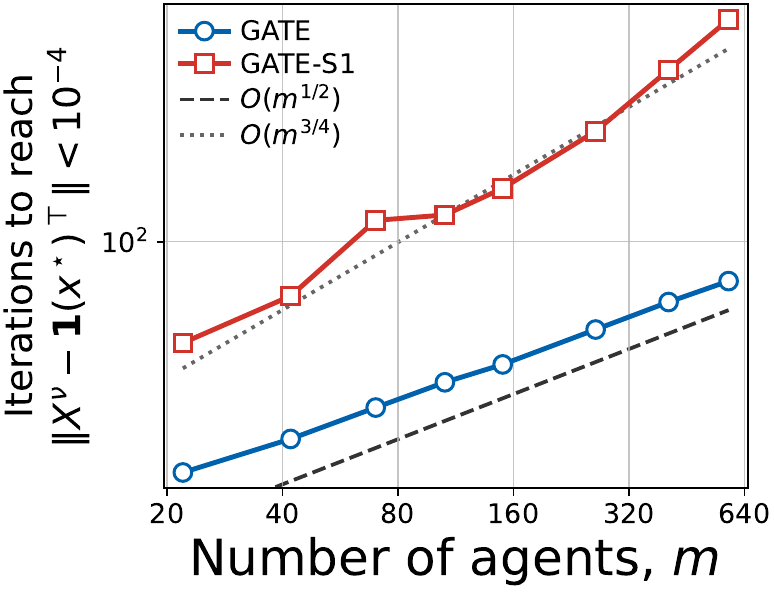}
        &
        \includegraphics[width=0.31\linewidth]{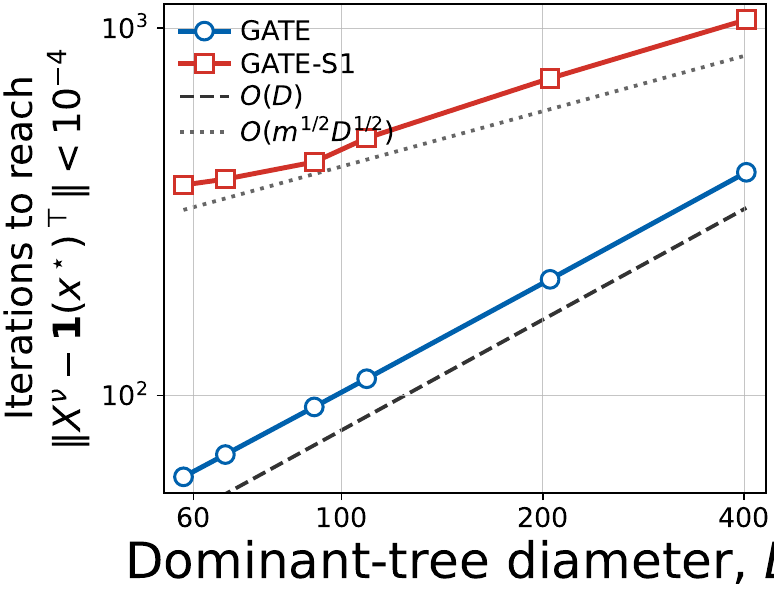}
        &
        \includegraphics[width=0.31\linewidth]{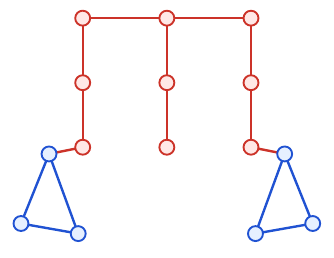}
        \\
        \parbox{0.31\linewidth}{\centering\footnotesize (a) Scaling in \(m\).}
        &
        \parbox{0.31\linewidth}{\centering\footnotesize (b) Scaling with \(m\) fixed.}
        &
        \parbox{0.31\linewidth}{\centering\footnotesize (c) Graph and partition.}
    \end{tabular}
    \caption{Network-size and diameter scaling for quadratic problems on a non-ring graph family.  Panel~(a) reports the iteration counts versus the number of agents \(m\) for graphs with \(D=\Theta(\sqrt m)\).  Panel~(b) reports the iteration counts versus  \(D\) at fixed \(m=406\).  Panel~(c) shows the graph and edge partition.  The stopping criterion is \(\|X^\nu-\mathbf 1(x^\star)^\top\|<10^{-4}\).}
    \label{fig:clean_nonring_comb_qp_scaling}
\end{figure}

\par\addvspace{\bigskipamount}
\noindent\emph{Effect of the partition.}\enspace
We next examine the effect of the partition for a quadratic problem on a ring.  For each network size, we compare three partitions: P1 contains one \((m-2)\)-edge path and two singleton-edge clusters; P2 arranges two paths of length \(m/2\); and P3 contains \(\sqrt m\) consecutive path blocks of \(\sqrt m\) edges each.

Figure~\ref{fig:clean_ring_qp_partition_ablation} shows that P1 and P2 substantially outperform P3, while P1 is modestly faster than P2.  This ordering is consistent with the discussion following Proposition~\ref{prop:clean_ring_partition_scaling}: \(D=\Theta(m)\) minimizes the rate bound for equal-length partitions of the ring, whereas Proposition~\ref{prop:clean_large_tree_scaling} gives the sharper \(O(m)\) rate for the long-path partition.

\begin{figure}[th]
    \centering
    \begin{minipage}[c]{0.69\linewidth}
        \centering
        \includegraphics[width=\linewidth]{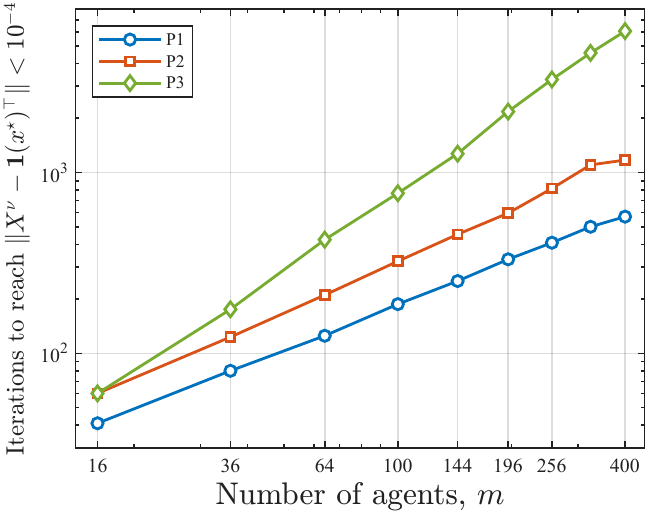}
        \par\vspace{2pt}
        {\footnotesize (a) Iteration-count comparison across partitions.\par}
    \end{minipage}
    \hfill
    \begin{minipage}[c]{0.27\linewidth}
        \centering
        \includegraphics[width=0.96\linewidth]{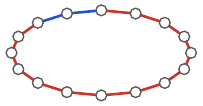}
        \par\vspace{-1pt}
        {\footnotesize (b) P1: one path and two singleton edges.\par}
        \vspace{4pt}
        \includegraphics[width=0.96\linewidth]{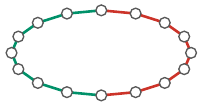}
        \par\vspace{-1pt}
        {\footnotesize (c) P2: two paths.\par}
        \vspace{4pt}
        \includegraphics[width=0.96\linewidth]{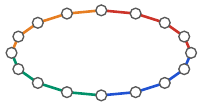}
        \par\vspace{-1pt}
        {\footnotesize (d) P3: \(\sqrt m\) paths.\par}
    \end{minipage}
    \caption{Effect of the partition on GATE for quadratic problems on a ring (\(d=10\)).  Panel~(a) reports the iteration counts versus the number of agents \(m\) for P1, P2, and P3.  Panels~(b)--(d) show the corresponding partitions on a representative ring.  The stopping criterion is \(\|X^\nu-\mathbf 1(x^\star)^\top\|<10^{-4}\).}
    \label{fig:clean_ring_qp_partition_ablation}
\end{figure}

\subsection{Performance comparison}
\label{subsec:clean_numerical_performance}

We now compare  the algorithms in terms of iteration count and cumulative communication on tree-dominant barbell networks and a multi-cycle network.  Communication is measured by the total number of scalars transmitted over directed links, normalized by the ambient dimension \(d\).  

\par\addvspace{\bigskipamount}
\noindent\emph{Tree-dominant barbell networks.}\enspace
We consider barbell graphs formed by two four-node cliques connected by a path of \(L\in\{10,100\}\) edges.  The partition contains one dominant tree cluster formed by the path and a three-edge star centered at the path endpoint in each clique, while each of the six remaining clique edges forms a singleton-edge cluster. 

For the quadratic problems, we set \(d=100\) and use dense $H_i$ generated from Gaussian Gram matrices with a diagonal correction, yielding \(\kappa(H_i)\approx100\).  GATE uses exact quadratic messages, while GATE-S1 and GATE-S2 (cache Hessian) use a fixed diagonal metric and a slice-wise Hessian cache, respectively.  Figure~\ref{fig:clean_barbell_qp_all_algorithms} shows that GATE is markedly faster than all other methods, reaching the target in only \(31\) and \(114\) iterations for \(L=10\) and \(L=100\), respectively.  Although GATE initially transmits dense message Hessians, these matrices stabilize after finitely many rounds and only the message vectors are sent thereafter; consequently, GATE also has the lowest cumulative communication.  GATE-S2 (cache Hessian) sends only message vectors and \(d\)-vector Hessian slices.  Once its cache has recovered the stabilized message Hessians, its tail follows essentially the same linear rate as GATE, consistent with the finite cache-recovery property in Sec.~\ref{subsec:clean_surrogate_message_structure}.  GATE-S1 is comparable to EXTRA, the fastest consensus baseline on these instances.  CTA and ATC do not reach \(10^{-6}\) within \(10^5\) iterations.

\begin{figure}[th]
    \centering
    \setlength{\tabcolsep}{1pt}
    \begin{tabular}{ccc}
        \includegraphics[height=0.24\linewidth,keepaspectratio]{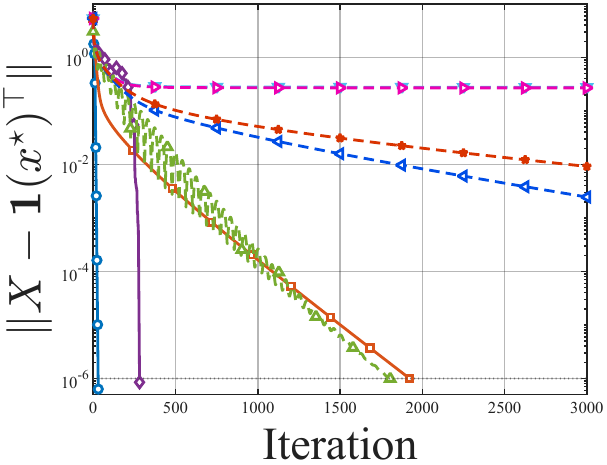}
        &
        \includegraphics[height=0.24\linewidth,keepaspectratio]{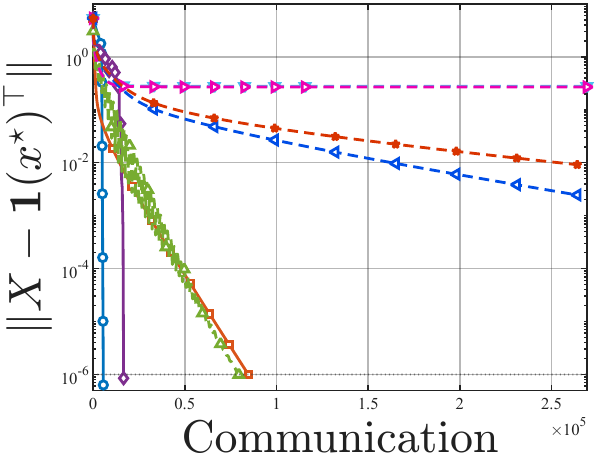}
        &
        \raisebox{0.035\linewidth}{\includegraphics[width=0.25\linewidth]{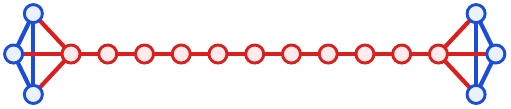}}
        \\
        \parbox{0.32\linewidth}{\centering\footnotesize (a) Iteration comparison.}
        &
        \parbox{0.32\linewidth}{\centering\footnotesize (b) Communication comparison.}
        &
        \parbox{0.32\linewidth}{\centering\footnotesize (c) Barbell, \(L=10\).}
        \\[4pt]
        \includegraphics[height=0.24\linewidth,keepaspectratio]{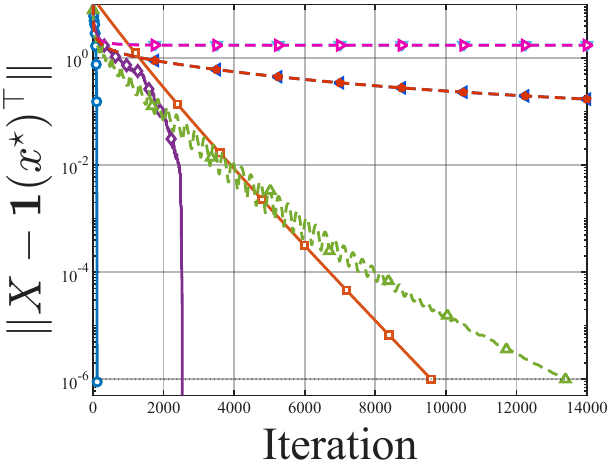}
        &
        \includegraphics[height=0.24\linewidth,keepaspectratio]{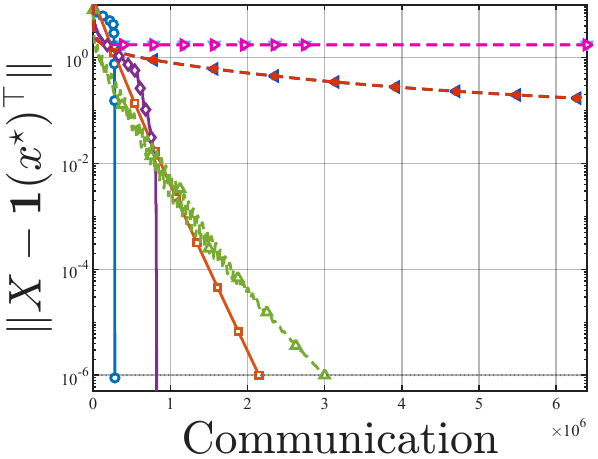}
        &
        \raisebox{0.035\linewidth}{\includegraphics[width=0.25\linewidth]{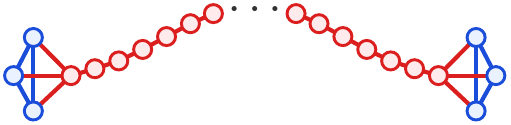}}
        \\
        \parbox{0.32\linewidth}{\centering\footnotesize (d) Iteration comparison.}
        &
        \parbox{0.32\linewidth}{\centering\footnotesize (e) Communication comparison.}
        &
        \parbox{0.32\linewidth}{\centering\footnotesize (f) Barbell, \(L=100\).}
        \\
        \multicolumn{3}{c}{%
            \begingroup
            \setlength{\fboxsep}{2pt}%
            \setlength{\fboxrule}{0.4pt}%
            \fbox{\includegraphics[width=0.64\linewidth]{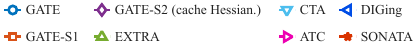}}%
            \endgroup}
    \end{tabular}
    \caption{Performance comparison for quadratic problems on two barbell graphs.  Panels~(a)--(c) use \(L=10\), while panels~(d)--(f) use \(L=100\).  Panels~(a) and (d) report the primal error versus iteration count, while panels~(b) and (e) report it versus cumulative communication.  Panels~(c) and (f) show the corresponding graphs and partitions; red edges form the dominant tree cluster and blue edges form singleton-edge clusters.  The stopping criterion is \(\|X^\nu-\mathbf 1(x^\star)^\top\|\leq10^{-6}\).}
    \label{fig:clean_barbell_qp_all_algorithms}
\end{figure}

We next test logistic-regression problems on the same graphs and partitions, with \(d=100\) and \(N_i=80\).  GATE-S2 (exact Hessian), GATE-S1, and GATE-S2 (cache Hessian) use the full center-dependent Hessian, its diagonal, and a slice-wise Hessian cache, respectively.  The exact-Hessian variant is included only in the iteration plots because it transmits dense Hessian matrices at every iteration.  Figure~\ref{fig:clean_barbell_logreg_all_algorithms} shows that the cache variant is second only to GATE-S2 (exact Hessian) in iteration count and requires the least communication among all other methods that reach the target. For \(L=100\), GATE-S2 (cache Hessian) requires about one third as much communication as either GATE-S1 or EXTRA.

\begin{figure}[th]
    \centering
    \setlength{\tabcolsep}{1pt}
    \begin{tabular}{ccc}
        \includegraphics[height=0.24\linewidth,keepaspectratio]{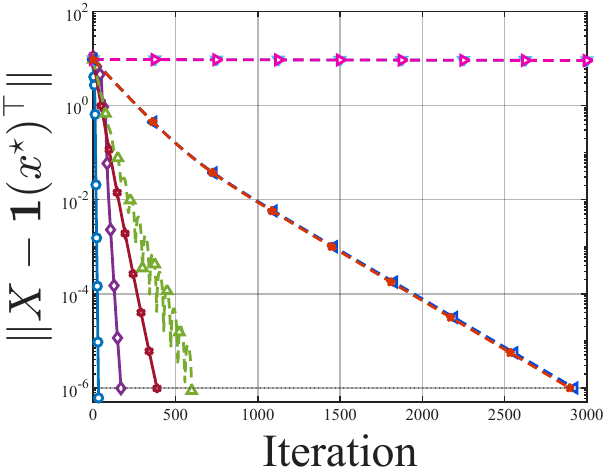}
        &
        \includegraphics[height=0.24\linewidth,keepaspectratio]{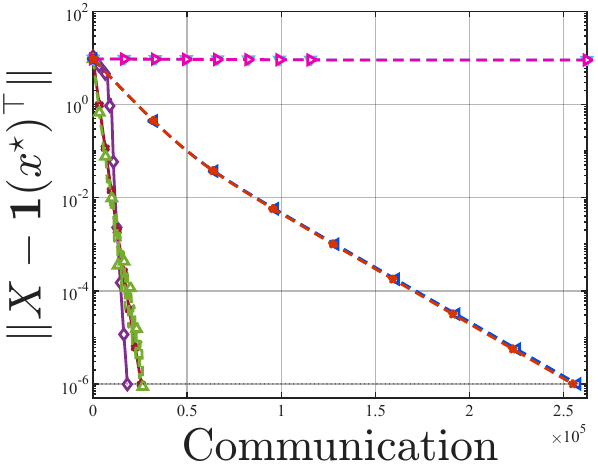}
        &
        \raisebox{0.035\linewidth}{\includegraphics[width=0.25\linewidth]{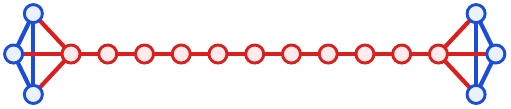}}
        \\
        \parbox{0.32\linewidth}{\centering\footnotesize (a) Iteration comparison.}
        &
        \parbox{0.32\linewidth}{\centering\footnotesize (b) Communication comparison.}
        &
        \parbox{0.32\linewidth}{\centering\footnotesize (c) Barbell, \(L=10\).}
        \\[4pt]
        \includegraphics[height=0.24\linewidth,keepaspectratio]{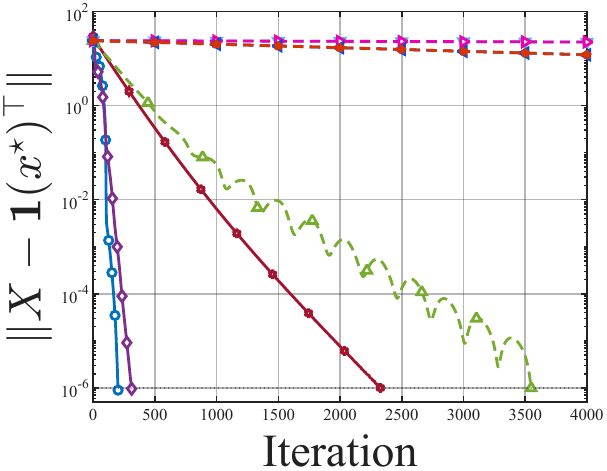}
        &
        \includegraphics[height=0.24\linewidth,keepaspectratio]{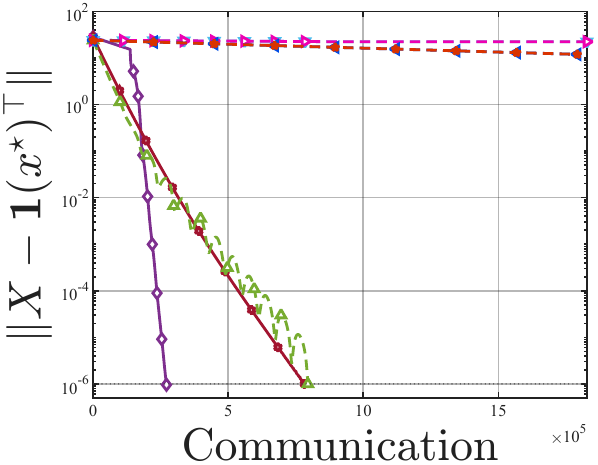}
        &
        \raisebox{0.035\linewidth}{\includegraphics[width=0.25\linewidth]{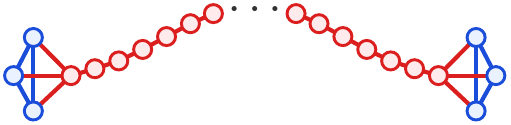}}
        \\
        \parbox{0.32\linewidth}{\centering\footnotesize (d) Iteration comparison.}
        &
        \parbox{0.32\linewidth}{\centering\footnotesize (e) Communication comparison.}
        &
        \parbox{0.32\linewidth}{\centering\footnotesize (f) Barbell, \(L=100\).}
        \\
        \multicolumn{3}{c}{%
            \begingroup
            \setlength{\fboxsep}{2pt}%
            \setlength{\fboxrule}{0.4pt}%
            \fbox{\includegraphics[width=0.68\linewidth]{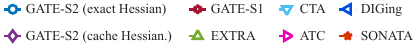}}%
            \endgroup}
    \end{tabular}
    \caption{Logistic-regression comparison on two barbell graphs.  Panels~(a)--(c) use \(L=10\), while panels~(d)--(f) use \(L=100\).  Panels~(a) and (d) report the error versus iteration count, while panels~(b) and (e) report it versus cumulative communication; GATE-S2 (exact Hessian) is omitted from the communication panels.  Panels~(c) and (f) show the corresponding graphs and partitions; red edges form the dominant tree cluster and blue edges form singleton-edge clusters.  The stopping criterion is \(\|X^\nu-\mathbf 1(x^\star)^\top\|\leq10^{-6}\).}
    \label{fig:clean_barbell_logreg_all_algorithms}
\end{figure}

\par\addvspace{\bigskipamount}
\noindent\emph{A multi-cycle network.}\enspace
We next test whether the gains persist on the network in Figure~\ref{fig:clean_multicycle_comparison}, which contains several cycles but no pronounced long-tree structure. We consider partitions beyond Assumption~\ref{ass:clean_tree_delay} to assess the empirical robustness of GATE to the choice of partition.
The quadratic problem uses \(d=200\) and dense local Hessians with condition number eight, while the logistic-regression problem uses \(d=50\) and \(N_i=80\).  For the quadratic problem, we compare exact GATE, two GATE-S1 variants, and GATE-S2 (cache Hessian) against the consensus baselines.  For logistic regression, we replace exact GATE with GATE-S2 (exact Hessian) and retain the remaining methods.

Figure~\ref{fig:clean_multicycle_comparison}(a) shows that exact GATE converges fastest for the quadratic problem. GATE-S2 (cache Hessian) follows essentially the same tail rate, and Figure~\ref{fig:clean_multicycle_comparison}(b) shows that it requires the least communication.  The scalar and diagonal GATE-S1 variants also outperform all consensus baselines in both iteration count and communication.  For logistic regression, Figure~\ref{fig:clean_multicycle_comparison}(d) shows that GATE-S2 (exact Hessian) is fastest, followed by the cache variant; Figure~\ref{fig:clean_multicycle_comparison}(e) shows that the cache variant again requires the least communication among the displayed methods.  The diagonal GATE-S1 variant remains faster than all consensus baselines in both measures, whereas the scalar GATE-S1 variant is less competitive on this instance.  EXTRA is the strongest consensus baseline in both problems; DIGing and SONATA require more communication, while CTA and ATC do not reach the target.  Thus, the practical gains observed on the barbell graphs persist on this multi-cycle network.

\begin{figure}[th]
    \centering
    \setlength{\tabcolsep}{1pt}
    \begin{tabular}{ccc}
        \includegraphics[height=0.24\linewidth,keepaspectratio]{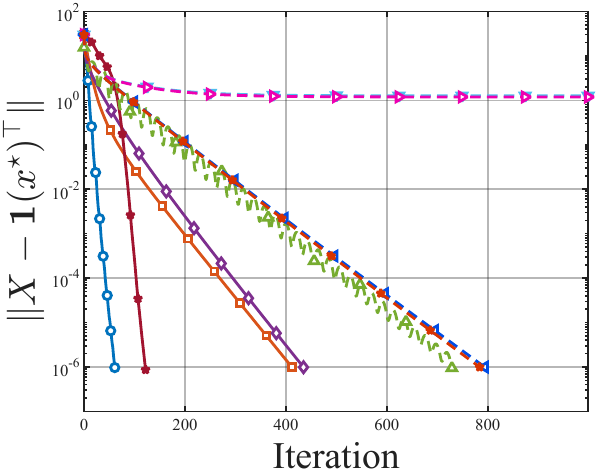}
        &
        \includegraphics[height=0.24\linewidth,keepaspectratio]{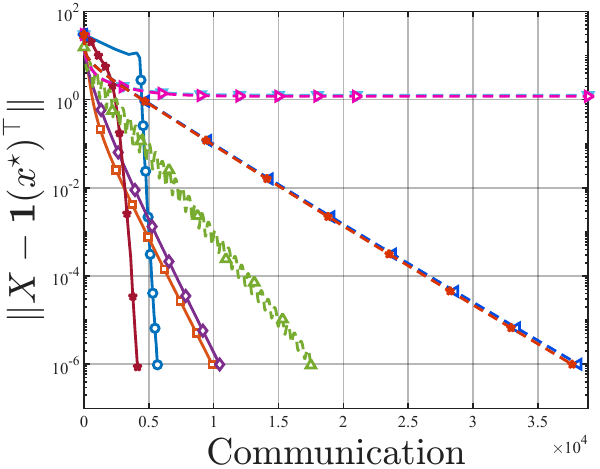}
        &
        \raisebox{0.035\linewidth}{\includegraphics[width=0.25\linewidth]{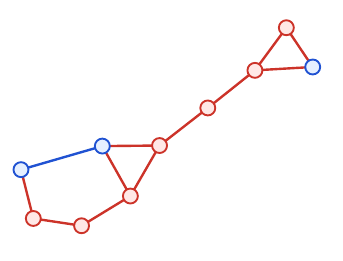}}
        \\
        \parbox{0.32\linewidth}{\centering\footnotesize (a) QP: iteration.}
        &
        \parbox{0.32\linewidth}{\centering\footnotesize (b) QP: communication.}
        &
        \parbox{0.32\linewidth}{\centering\footnotesize (c) Graph and partition.}
        \\[4pt]
        \includegraphics[height=0.24\linewidth,keepaspectratio]{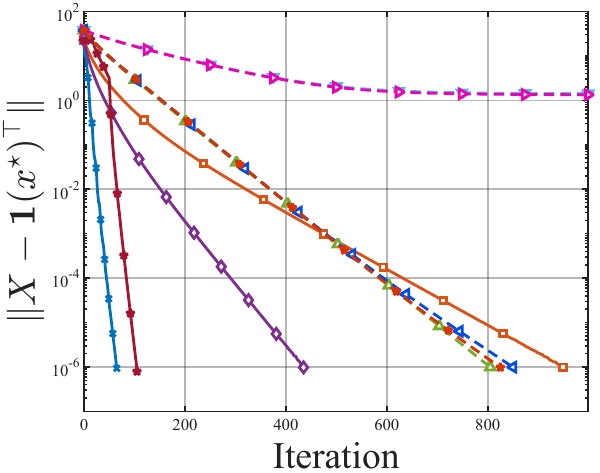}
        &
        \includegraphics[height=0.24\linewidth,keepaspectratio]{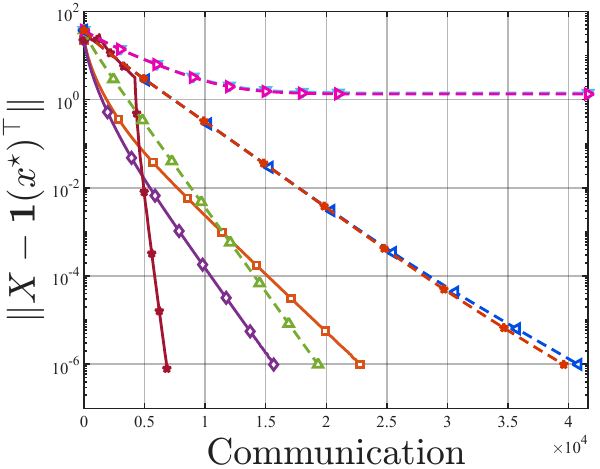}
        &
        \raisebox{0.035\linewidth}{\includegraphics[width=0.25\linewidth]{logreg_test_all_graph_partition.pdf}}
        \\
        \parbox{0.32\linewidth}{\centering\footnotesize (d) LR: iteration.}
        &
        \parbox{0.32\linewidth}{\centering\footnotesize (e) LR: communication.}
        &
        \parbox{0.32\linewidth}{\centering\footnotesize (f) Graph and partition.}
        \\
        \multicolumn{3}{c}{\includegraphics[width=0.90\linewidth]{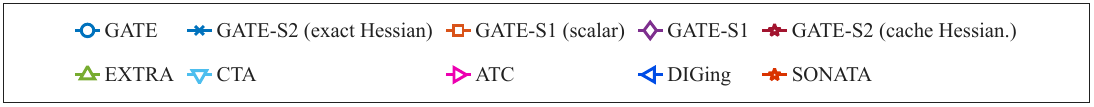}}
    \end{tabular}
    \caption{Performance comparison for quadratic and logistic-regression problems on a multi-cycle network.  Panels~(a)--(c) show the quadratic results with \(d=200\), while panels~(d)--(f) show the logistic-regression (LR) results with \(d=50\).  Panels~(a) and (d) report the error versus iteration count, while panels~(b) and (e) report it versus cumulative communication; GATE-S2 (exact Hessian) is omitted from panel~(e).  Panels~(c) and (f) show the common graph and partition; red edges form the message-passing cluster and the blue edge forms the singleton-edge cluster.  The stopping criterion is \(\|X^\nu-\mathbf 1(x^\star)^\top\|\leq10^{-6}\).}
    \label{fig:clean_multicycle_comparison}
\end{figure}

\section{Conclusion}
\label{sec:conclusion}
 In this work, we developed a graph-decomposition framework for jointly designing the consensus representation, dual-coordinate
  blocks, and connected carriers. The framework makes graph
  structure part of the optimization decomposition itself,
  beyond its role in communicating prescribed updates.
  Specializing to the incidence representation and tree-structured  edge blocks, we proposed \texttt{GATE}, a single-loop decentralized message-passing algorithm derived from the block fixed-point
  decomposition.    
  
   For smooth strongly convex local objectives, we established
  linear convergence and quantified a \emph{block-preconditioning
  effect}: the partition changes the coupling between subproblems,
  so network topology and function geometry enter the rate jointly.
  The resulting bounds guide partition selection for the given
  functions and network. For the tree-dominant families and
  partitions considered, under the stated curvature assumptions,
  the guaranteed factor \(\rho\) satisfies
  \((1-\rho)^{-1}=O(D/\log D)\), with \(D\) comparable to the
  dominant-tree diameter and \(O(D)\) initial propagation rounds.
  GATE-S reduces per-iteration computation and communication
  through tractable local surrogates and compact message
  parametrizations, while retaining linear convergence under the regularity assumptions.

   Numerical experiments on quadratic and logistic-regression
  problems support the theoretical scaling and demonstrate
  iteration and communication savings over the tested
  consensus-based baselines. Future directions include
  automated partition selection guided by the rate bounds,
  conjugate-free variants, and extensions to non(-strongly)
  convex objectives and stochastic settings.

\section*{Appendix}

\appendix

\section{Proofs of Results in Sec.~\ref{sec:clean_convergence}}

\subsection{Proof of Proposition~\ref{prop:clean_delayed_bilevel}}
\label{app:clean_delayed_bilevel_proof}

Fix a block \(r\), an edge \(e\in\cE_r\), and an iteration \(\nu\ge\operatorname{diam}(\cT_r)\).  Since \(e\in\cE_r\) is a vertex of the factor graph \(\cF_r\), root \(\cF_r\) at \(e\).  For every vertex \(v\neq e\), let \({\rm pa}(v)\) be its parent.   For \(i\in\cV_r\), let \(\cT_i\) be the component of \(\cF_r\setminus\{(i,{\rm pa}(i))\}\) containing \(i\), and denote by \(\cV_i\) and \(\cE_i\) its vertices in \(\cV_r\) and \(\cE_r\), respectively.  Thus \({\rm pa}(i)\notin\cE_i\), where \(\cE_i\) consists exactly of the active edge variables strictly below \(i\). For \(\ell\in\cV_r\), define the rooted local factor
\[
\Psi_\ell^{\nu,e}(Y_{\cA_\ell^\circ})
:=
f_\ell^*\!\left(
\mathsf B_{\ell,\cA_\ell^\circ}Y_{\cA_\ell^\circ}
	+
	\sum_{h\in\cA_\ell^\partial}\mathsf B_{\ell,h}Y_h^{\nu-\delta_{e,\ell}}
	\right),
\]
where the second term is absent when \(\ell\in\cV_r^\circ\).

\noindent\textbf{Step 1: Subtree message identity.}  For each  \(i\in\cV_r\), set \(q_i:=(\dist_{\cF_r}(e,i)-1)/2\).   Thus \(q_i=\delta_{e,i}\) for boundary nodes. Since \(q_i+1\le\operatorname{diam}(\cT_r)\le\nu\), every message in the expansion has index \(\nu-q_i\ge1\) and satisfies \eqref{eq:clean_factor_to_var}; no initial message function remains. We prove that
\begin{equation}
\mu_{i\to{\rm pa}(i)}^{\,\nu-q_i}\!\left(Y_{{\rm pa}(i)}\right)
=
\min_{Y_{\cE_i}}
\sum_{\ell\in\cV_i}
\Psi_\ell^{\nu,e}(Y_{\cA_\ell^\circ}).
\label{eq:app_clean_subtree_message}
\end{equation}
We prove~\eqref{eq:app_clean_subtree_message} by induction on the height of \(\cT_i\).

\textit{Base case.}  If \(\cV_i=\{i\}\), then \(\cE_i=\emptyset\) and \(\cA_i^\circ=\{{\rm pa}(i)\}\). Evaluating \eqref{eq:clean_factor_to_var} with \(\nu+1\) replaced by \(\nu-q_i\) gives
\[
\mu_{i\to{\rm pa}(i)}^{\,\nu-q_i}\!\left(Y_{{\rm pa}(i)}\right)
=
\Psi_i^{\nu,e}(Y_{\cA_i^\circ}),
\]
which is~\eqref{eq:app_clean_subtree_message} for the base case.

\textit{Inductive step.} Assume now that~\eqref{eq:app_clean_subtree_message} holds for all strict descendants of \(i\).  The same substitution in \eqref{eq:clean_factor_to_var} gives
\begin{equation}
\begin{aligned}
\mu_{i\to{\rm pa}(i)}^{\,\nu-q_i}\!\left(Y_{{\rm pa}(i)}\right)
&=
\min_{Y_{\cA_i^\circ\setminus\{{\rm pa}(i)\}}}
\Bigg\{
\Psi_i^{\nu,e}(Y_{\cA_i^\circ})
+
\sum_{h\in\cA_i^\circ\setminus\{{\rm pa}(i)\}}
\ \sum_{j\in\cI_h\setminus\{i\}}
\mu_{j\to h}^{\,\nu-q_i-1}(Y_h)
\Bigg\}.
\end{aligned}
\label{eq:app_clean_message_expand}
\end{equation}
For each pair \((h,j)\) in the double sum, \({\rm pa}(h)=i\), \({\rm pa}(j)=h\), and \(q_j=q_i+1\). The induction hypothesis therefore gives
\[
\mu_{j\to h}^{\,\nu-q_i-1}(Y_h)
=
\mu_{j\to{\rm pa}(j)}^{\,\nu-q_j}(Y_{{\rm pa}(j)})
=
\min_{Y_{\cE_j}}
\sum_{\ell\in\cV_j}
\Psi_\ell^{\nu,e}(Y_{\cA_\ell^\circ}).
\]
Since \(h={\rm pa}(j)\), the pairs \((h,j)\) correspond one-to-one to \(j\in\cV_i\setminus\{i\}\) with \({\rm pa}({\rm pa}(j))=i\). The sets \(\cA_i^\circ\setminus\{{\rm pa}(i)\}\) and these \(\cE_j\) partition \(\cE_i\); hence substituting into~\eqref{eq:app_clean_message_expand}, reindexing by \(j\), and merging the minimizations gives
\[
\begin{aligned}
\mu_{i\to{\rm pa}(i)}^{\,\nu-q_i}\!\left(Y_{{\rm pa}(i)}\right)
&=
\min_{Y_{\cE_i}}
\Bigg\{
\Psi_i^{\nu,e}(Y_{\cA_i^\circ})
+
\sum_{\substack{j\in\cV_i\setminus\{i\}\\{\rm pa}({\rm pa}(j))=i}}
\sum_{\ell\in\cV_j}
\Psi_\ell^{\nu,e}(Y_{\cA_\ell^\circ})
\Bigg\}.
\end{aligned}
\]
Using \(\cV_i=\{i\}\sqcup\bigsqcup_{\substack{j\in\cV_i\setminus\{i\}\\{\rm pa}({\rm pa}(j))=i}}\cV_j\), we obtain
\[
\mu_{i\to{\rm pa}(i)}^{\,\nu-q_i}\!\left(Y_{{\rm pa}(i)}\right)
=
\min_{Y_{\cE_i}}
\sum_{\ell\in\cV_i}
\Psi_\ell^{\nu,e}(Y_{\cA_\ell^\circ}).
\]

\noindent\textbf{Step 2: Sum of messages at the root.}  Applying the identity to \eqref{eq:clean_dmp_updates_a}, we can write
\[
\hat Y_e^{\nu+1}
\in
\argmin_{Y_e}
\left\{
\sum_{i\in\cI_e}\mu_{i\to e}^{\nu}(Y_e)
\right\}.
\]
For  \(i\in\cI_e\), one has \(q_i=0\), so \eqref{eq:app_clean_subtree_message} applies with \({\rm pa}(i)=e\).   The sets \(\cV_i\), \(i\in\cI_e\), partition \(\cV_r\), and the corresponding sets \(\cE_i\) partition \(\cE_r\setminus\{e\}\). Therefore
\[
\hat Y_e^{\nu+1}
\in
\argmin_{Y_e}\ \min_{Y_{\cE_r\setminus\{e\}}}
\sum_{\ell\in\cV_r}\Psi_\ell^{\nu,e}(Y_{\cA_\ell^\circ}).
\]
This is exactly \eqref{eq:clean_expanded_delayed_oracle}. \qed

\subsection{Proof of Lemma~\ref{lem:clean_boundary_sensitivity}}
\label{app:clean_tree_oracle}

The next two lemmas isolate the two ingredients for Lemma~\ref{lem:clean_boundary_sensitivity}: for each fixed boundary vector, the minimization over \(Y_{\cE_r}\) generates a common primal variable on \(\cT_r\), and the linear system \(\mathsf B_{\cV_r,\cE_r}Y_{\cE_r}=q\) admits an explicit graph-compliant solution. For each block \(r\), write \(N_r:=|\cV_r|\) and \(\widehat J_r:=|\widehat\partial_r|\). For \(E\in\R^{\widehat J_r\times d}\) with rows indexed by \((i,h)\in\widehat{\partial}_r\), define
\[
c_i(E):=\sum_{h\in\cA_i^\partial}\mathsf B_{i,h}E_{i,h},
\quad
\eta(E):=\sum_{i\in\cV_r}c_i(E)
=\sum_{(i,h)\in\widehat{\partial}_r}\mathsf B_{i,h}E_{i,h}.
\]
Also write, for a minimizer \(Y_r(E)\) in~\eqref{eq:clean_boundary_solution_mapping},
\[
\bar x_i(E):=
    \nabla f_i^*\bigl((\mathsf B_{\cV_r,\cE_r}Y_r(E))_i+c_i(E)\bigr),
    \quad i\in\cV_r .
\]
For \(e\in\cE_r\) and \(v\in\cV_r\), write
\[
\mathsf{Y}_{r,e}(E):=(Y_r(E))_e,\qquad
\mathsf{Z}_{r,v}(E):=\mathsf B_{v,\cE_r}Y_r(E).
\]

\begin{lemma}
\label{lem:app_tree_oracle_jacobians}
Under Assumptions~\ref{ass:clean_problem_class} and~\ref{ass:clean_tree_delay}, the vectors \(\bar x_i(E)\), \(i\in\cV_r\), have a common value, which depends on \(E\) only through \(\eta(E)\); write it as \(\chi_r(\eta(E))\). That is
\begin{subequations}\label{eq:app_common_state}
\begin{align}
&\bar x_i(E)=\chi_r(\eta(E)),
&& i\in\cV_r, \label{eq:app_common_state_a}\\
&\sum_{i\in\cV_r}\nabla f_i(\chi_r(\eta(E)))=\eta(E),
&& \label{eq:app_common_state_b}\\
&(\mathsf B_{\cV_r,\cE_r}Y_r(E))_i
=\nabla f_i(\chi_r(\eta(E)))-c_i(E),
&& i\in\cV_r . \label{eq:app_common_state_c}
\end{align}
\end{subequations}
\end{lemma}

\begin{proof}
Fix \(E\) and the minimizer \(Y_r(E)\). For \(e=(u,v)\in\cE_r\), the optimality condition gives
\[
0
=\mathsf B_{u,e}\bar x_u(E)+\mathsf B_{v,e}\bar x_v(E)
=\mathsf B_{u,e}\bigl(\bar x_u(E)-\bar x_v(E)\bigr).
\]
Thus \(\bar x_u(E)=\bar x_v(E)\) for every \(e=(u,v)\in\cE_r\).  Since \(\cT_r\) is connected, all \(\bar x_i(E)\) are equal.  Because \(\nabla f_i=(\nabla f_i^*)^{-1}\), \(\nabla f_i(\bar x_i(E))=(\mathsf B_{\cV_r,\cE_r}Y_r(E))_i+c_i(E)\) for \(i\in\cV_r\). Summing over \(i\in\cV_r\) and using the zero column sums of \(\mathsf B_{\cV_r,\cE_r}\) gives \(\sum_{i\in\cV_r}\nabla f_i(\bar x_i(E))=\eta(E)\). Since all \(\bar x_i(E)\) are equal, this equation identifies their common value as \(\chi_r(\eta(E))\).  Thus the common value depends on \(E\) only through \(\eta(E)\), and \eqref{eq:app_common_state_a}--\eqref{eq:app_common_state_b} hold.  Rearranging the first identity before summation gives~\eqref{eq:app_common_state_c}. \qed
\end{proof}

\medskip
\begin{lemma}
\label{lem:clean_tree_inverse}
For any \(q=(q_i)_{i\in\cV_r}\) with \(q_i\in\R^d\) and \(\sum_{i\in\cV_r}q_i=0\), the equation \(\mathsf B_{\cV_r,\cE_r}Y_{\cE_r}=q\) has a unique solution.  If \(e\in\cE_r\) and \(\cS\) is one of the two connected components of \(\cT_r\setminus\{e\}\), then \(Y_e=\varepsilon_{e,\cS}\sum_{i\in\cS}q_i\) for some sign \(\varepsilon_{e,\cS}\in\{\pm1\}\).
\end{lemma}

\begin{proof}
Each column of \(\mathsf B_{\cV_r,\cE_r}\) sums to zero.  Since \(\cT_r\) is a tree, \(\rank(\mathsf B_{\cV_r,\cE_r})=|\cE_r|=|\cV_r|-1\), and hence \(\Range(\mathsf B_{\cV_r,\cE_r}) =\{q\in\R^{|\cV_r|\times d}:\sum_{i\in\cV_r}q_i=0\}\). Thus \(\mathsf B_{\cV_r,\cE_r}Y_{\cE_r}=q\) has a unique solution.  Fix one component \(\cS\) of \(\cT_r\setminus\{e\}\).  Summing the rows indexed by \(\cS\) gives
\[
    \sum_{i\in\cS}q_i
    =
    \sum_{a\in\cE_r}\left(\sum_{i\in\cS}\mathsf B_{i,a}\right)Y_a
    =
    \left(\sum_{i\in\cS}\mathsf B_{i,e}\right)Y_e,
    \qquad
    \sum_{i\in\cS}\mathsf B_{i,e}\in\{\pm1\},
\]
where the terms with \(a\neq e\) vanish because those edges do not cross from \(\cS\) to \(\cV_r\setminus\cS\). This gives the formula. \qed
\end{proof}

\medskip
\noindent\textit{Proof of Lemma~\ref{lem:clean_boundary_sensitivity}.} For \(E\in\R^{\widehat J_r\times d}\) with rows indexed by \((i,h)\in\widehat{\partial}_r\), set \(q_i(E):=\nabla f_i(\chi_r(\eta(E)))-c_i(E)\). Lemma~\ref{lem:app_tree_oracle_jacobians} gives \(\mathsf B_{\cV_r,\cE_r}Y_r(E)=q(E)\), and \(\sum_{i\in\cV_r}q_i(E)=0\).  Lemma~\ref{lem:clean_tree_inverse} gives a unique solution of \(\mathsf B_{\cV_r,\cE_r}Y_{\cE_r}=q(E)\).  Hence \(Y_r(E)\) is single-valued. It also gives, for a boundary incidence \((i,h)\in\widehat{\partial}_r\) and the component \(\overline{\cV}_e(i)\) not containing \(i\),
\begin{equation}
\mathsf{Y}_{r,e}(E)
=
\varepsilon_{e,i}
\sum_{j\in\overline{\cV}_e(i)}
\bigl(\nabla f_j(\chi_r(\eta(E)))-c_j(E)\bigr),
\quad
\varepsilon_{e,i}\in\{\pm1\}.
\label{eq:app_cut_formula}
\end{equation}
Let \(E,E'\) differ only in coordinate \((i,h)\).  Since \(i\notin\overline{\cV}_e(i)\), the boundary terms \(c_j(\cdot)\) in \eqref{eq:app_cut_formula} are unchanged for all \(j\in\overline{\cV}_e(i)\).  Also \(\eta(E')-\eta(E)=\mathsf B_{i,h}(E'_{i,h}-E_{i,h})\).  Hence
\[
\mathsf{Y}_{r,e}(E')-\mathsf{Y}_{r,e}(E)
=
\varepsilon_{e,i}
\bigl(\cL_{r,\overline{\cV}_e(i)}(\eta(E'))-\cL_{r,\overline{\cV}_e(i)}(\eta(E))\bigr),
\]
and \eqref{eq:clean_partial_edge_sens} follows from the definition of \(\kappa_{\partial,r}\).

For \(v\ne i\), the coordinate \((i,h)\) does not enter \(c_v\), so \(c_v(E')=c_v(E)\).  By \eqref{eq:app_common_state_c} and the definition of \(\mathsf{Z}_{r,v}\), \(\mathsf{Z}_{r,v}(E)=\nabla f_v(\chi_r(\eta(E)))-c_v(E)\).  Since \(\cL_{r,\{v\}}(\eta)=\nabla f_v(\chi_r(\eta))\), we obtain
\[
\mathsf{Z}_{r,v}(E')-\mathsf{Z}_{r,v}(E)
=
\cL_{r,\{v\}}(\eta(E'))-\cL_{r,\{v\}}(\eta(E)).
\]
Then, \eqref{eq:clean_partial_node_sens} follows from the definition of \(\kappa_{\partial,r}\). \qed

\subsection{Proof of Lemma~\ref{lem:clean_tree_mismatch}}
\label{app:clean_tree_mismatch_proof}

The proof combines the sensitivity bounds in Lemma~\ref{lem:clean_boundary_sensitivity} with the following counting bounds for the rooted tree.

\begin{lemma}
\label{lem:app_rooted_delay_combinatorics}
Fix \((i,h)\in\widehat{\partial}_r\) and root the active tree \(\cT_r\) at the boundary node \(i\).  For \(v\in\cV_r\), set
\begin{equation}
\delta_{r,i}(v):=\max\{\dist_{\cT_r}(i,v)-1,0\}.
\label{eq:app_rooted_boundary_delay}
\end{equation}
For \(v\ne i\), let \(\mathrm{Ch}_{r,i}(v)\) be the set of active edges from \(v\) to its children in this rooted tree; set \(\mathrm{Ch}_{r,i}(i):=\emptyset\).  Let \(p(v)\) be the active edge from \(v\ne i\) to its parent, and set
\begin{equation*}
n_{r,i}(i):=N_r,
\qquad
n_{r,i}(v):=|\overline{\cV}_{p(v)}(i)|,
\quad v\ne i.
\end{equation*}
Then
\begin{equation}
\sum_{e\in\mathrm{Ch}_{r,i}(v)}|\overline{\cV}_e(i)|
\le
n_{r,i}(v)-1,
\qquad
\sum_{\dist_{\cT_r}(i,v)=t}\frac{(n_{r,i}(v)-1)^2}{N_r^2}\le 1,
\quad t\ge0.
\label{eq:app_rooted_counts}
\end{equation}
\end{lemma}

\begin{proof}
For the first bound in~\eqref{eq:app_rooted_counts}, take \(v\neq i\) and put \(p=p(v)\).  The components \(\overline{\cV}_e(i)\) with \(e\in\mathrm{Ch}_{r,i}(v)\) are pairwise disjoint, and their union is \(\overline{\cV}_p(i)\setminus\{v\}\).  Hence equality holds in the first bound of \eqref{eq:app_rooted_counts} for \(v\neq i\).  If \(v=i\), then \(\mathrm{Ch}_{r,i}(v)=\emptyset\), and the displayed inequality is immediate.

Fix \(t\ge1\).  If two distinct nodes \(v,w\) have \(\dist_{\cT_r}(i,v)=\dist_{\cT_r}(i,w)=t\), then \(\overline{\cV}_{p(v)}(i)\cap\overline{\cV}_{p(w)}(i)=\emptyset\): otherwise some node \(u\) would have both \(v\) and \(w\) on the unique path from \(i\) to \(u\), and the two vertices at depth \(t\) on this path would have to coincide.  Hence
\[
\begin{gathered}
    \sum_{\dist_{\cT_r}(i,v)=t} n_{r,i}(v)\le N_r.
    \\
    \sum_{\dist_{\cT_r}(i,v)=t}(n_{r,i}(v)-1)^2
    \le
    \left(\sum_{\dist_{\cT_r}(i,v)=t}(n_{r,i}(v)-1)\right)^2
    \le N_r^2,
\end{gathered}
\]
The same conclusion is immediate for \(t=0\), since \(n_{r,i}(i)=N_r\). This proves the second bound in~\eqref{eq:app_rooted_counts}. \qed
\end{proof}

\medskip
\noindent\textit{Proof of Lemma~\ref{lem:clean_tree_mismatch}.}
By~\eqref{eq:clean_block_images}, \(\bar Z_r^{\nu+1}=\mathsf B_{\cV,\cE_r}Y_r(E^\nu)\) with \((E^\nu)_{(i,h)}=Y_h^\nu\), and \(\hat Y_e^{\nu+1}=\mathsf{Y}_{r,e}(\widehat E_e^\nu)\) with \((\widehat E_e^\nu)_{(i,h)}=Y_h^{\nu-\delta_{e,i}}\) for \((i,h)\in\widehat{\partial}_r\). For \(v\in\cV_r\), define the node mismatch
\[
m_v:=(\hat Z_r^{\nu+1}-\bar Z_r^{\nu+1})_v.
\]
The vector \(\widehat E_e^\nu\) depends on the incident edge \(e\), whereas \(m_v\) fixes the node \(v\).  To estimate \(m_v\), insert the auxiliary node-aligned boundary vector
\begin{equation*}
    (E_v^\nu)_{(i,h)}:=Y_h^{\nu-\delta_{r,i}(v)},\qquad (i,h)\in\widehat{\partial}_r.
\end{equation*}
For \(e\in\cA_v^\circ\), the definition of \(\delta_{e,i}\) gives \(\delta_{e,i}=\delta_{r,i}(v)\) unless \(e\in\mathrm{Ch}_{r,i}(v)\), in which case \(\delta_{e,i}=\delta_{r,i}(v)+1\).   At \(v=i\), both delays are zero for every incident edge. Thus \(E_v^\nu\) uses, in each boundary coordinate \((i,h)\), the freshest value seen by any active edge incident to \(v\). Since \((\hat Z_r^{\nu+1})_v =\sum_{e\in\cA_v^\circ}\mathsf B_{v,e}\mathsf{Y}_{r,e}(\widehat E_e^\nu)\) and \((\bar Z_r^{\nu+1})_v=\mathsf{Z}_{r,v}(E^\nu)\), inserting and subtracting \(\mathsf{Z}_{r,v}(E_v^\nu)\) gives
\begin{equation*}
\begin{aligned}
m_v
&=
\underbrace{
\sum_{e\in\cA_v^\circ}\mathsf B_{v,e}
\left(\mathsf{Y}_{r,e}(\widehat E_e^\nu)
-\mathsf{Y}_{r,e}(E_v^\nu)\right)
}_{=:p_v}
+
\underbrace{
\mathsf{Z}_{r,v}(E_v^\nu)-\mathsf{Z}_{r,v}(E^\nu)
}_{=:q_v}.
\end{aligned}
\end{equation*}
Write \(\Delta Y_h^\ell:=Y_h^{\ell+1}-Y_h^\ell\).
\smallskip

\noindent\textbf{Estimate of \(p_v\).} Fix an ordering of the boundary coordinates \((i,h)\in\widehat{\partial}_r\).  For each incident edge \(e\in\cA_v^\circ\), telescope the difference from \(E_v^\nu\) to \(\widehat E_e^\nu\) by changing one boundary coordinate at a time.  In the step that changes coordinate \((i,h)\), let \(E_{e,(i,h)}^-\) and \(E_{e,(i,h)}^+\) denote the two consecutive boundary vectors.  Then
\[
\mathsf{Y}_{r,e}(\widehat E_e^\nu)-\mathsf{Y}_{r,e}(E_v^\nu)
=
\sum_{(i,h)\in\widehat{\partial}_r}
\bigl(\mathsf{Y}_{r,e}(E_{e,(i,h)}^+)-\mathsf{Y}_{r,e}(E_{e,(i,h)}^-)\bigr),
\]
where the \((i,h)\)-summand is zero if \((\widehat E_e^\nu)_{(i,h)}=(E_v^\nu)_{(i,h)}\).  The preceding observation shows that this coordinate difference is zero unless \(e\in\mathrm{Ch}_{r,i}(v)\).  If \(e\in\mathrm{Ch}_{r,i}(v)\), then
\[
(\widehat E_e^\nu)_{(i,h)}-(E_v^\nu)_{(i,h)}
=
-\Delta Y_h^{\nu-(\delta_{r,i}(v)+1)}.
\]
Therefore
\[
p_v
=
\sum_{(i,h)\in\widehat{\partial}_r}
\underbrace{\sum_{e\in\mathrm{Ch}_{r,i}(v)}
\mathsf B_{v,e}
\bigl(\mathsf{Y}_{r,e}(E_{e,(i,h)}^+)-\mathsf{Y}_{r,e}(E_{e,(i,h)}^-)\bigr)}_{=:p_{v,(i,h)}}.
\]
If \(\mathrm{Ch}_{r,i}(v)=\emptyset\), then \(p_{v,(i,h)}=0\). Otherwise, using the sensitivity bound \eqref{eq:clean_partial_edge_sens} and the first bound in \eqref{eq:app_rooted_counts},
\[
\|p_{v,(i,h)}\|
\le
\kappa_{\partial,r}
\frac{n_{r,i}(v)-1}{N_r}
\|\Delta Y_h^{\nu-(\delta_{r,i}(v)+1)}\|.
\]
By Cauchy-Schwarz over the \(\widehat J_r\) boundary coordinates,
\begin{equation}
\|p_v\|^2
\le
\widehat J_r\kappa_{\partial,r}^2
\sum_{\substack{(i,h)\in\widehat{\partial}_r\\\mathrm{Ch}_{r,i}(v)\ne\emptyset}}
\frac{(n_{r,i}(v)-1)^2}{N_r^2}
\|\Delta Y_h^{\nu-(\delta_{r,i}(v)+1)}\|^2.
\label{eq:app_pv_pointwise}
\end{equation}
Summing~\eqref{eq:app_pv_pointwise} over \(v\) and exchanging the \(v\)- and \((i,h)\)-sums gives
\[
\sum_{v\in\cV_r}\|p_v\|^2
\le
\widehat J_r\kappa_{\partial,r}^2
\sum_{(i,h)\in\widehat{\partial}_r}
\sum_{\substack{v\in\cV_r\\\mathrm{Ch}_{r,i}(v)\ne\emptyset}}
\frac{(n_{r,i}(v)-1)^2}{N_r^2}
\|\Delta Y_h^{\nu-(\delta_{r,i}(v)+1)}\|^2.
\]
For fixed \((i,h)\), group the inner sum by rooted depth \(t=\dist_{\cT_r}(i,v)\).  On the nodes with \(\mathrm{Ch}_{r,i}(v)\ne\emptyset\), one has \(v\ne i\) and \(1\le\delta_{r,i}(v)+1=t\le D_r\). Thus all nodes in the same depth shell use the same increment \(\Delta Y_h^{\nu-t}\), and the second bound in~\eqref{eq:app_rooted_counts} gives
\[
\sum_{\substack{v\in\cV_r\\\mathrm{Ch}_{r,i}(v)\ne\emptyset}}
\frac{(n_{r,i}(v)-1)^2}{N_r^2}
\|\Delta Y_h^{\nu-(\delta_{r,i}(v)+1)}\|^2
\le
\sum_{t=1}^{D_r}
\|\Delta Y_h^{\nu-t}\|^2.
\]
Summing over \((i,h)\) gives
\begin{equation}
\sum_{v\in\cV_r}\|p_v\|^2
\le
\widehat J_r\kappa_{\partial,r}^2
\sum_{\ell=\nu-D_r}^{\nu-1}
\sum_{(i,h)\in\widehat{\partial}_r}\|\Delta Y_h^\ell\|^2.
\label{eq:app_transport_bound}
\end{equation}

\smallskip

\noindent\textbf{Estimate of \(q_v\).} Telescope the difference from \(E^\nu\) to \(E_v^\nu\) by changing one boundary coordinate at a time.  If \(v=i\), then \(\delta_{r,i}(v)=0\), so the corresponding coordinate difference is zero. Hence only coordinates \((i,h)\) with \(v\neq i\) remain, and the sensitivity bound~\eqref{eq:clean_partial_node_sens} applies:
\begin{equation}
\|q_v\|^2
\le
\widehat J_r\kappa_{\partial,r}^2
\sum_{(i,h)\in\widehat{\partial}_r:\,i\neq v}
\frac{1}{N_r^2}
\|Y_h^{\nu-\delta_{r,i}(v)}-Y_h^\nu\|^2.
\label{eq:app_qv_pointwise}
\end{equation}
For every integer \(t\ge 0\),
\[
\|Y_h^{\nu-t}-Y_h^\nu\|^2
\le
t\sum_{m=1}^t\|\Delta Y_h^{\nu-m}\|^2.
\]
Substituting this into~\eqref{eq:app_qv_pointwise}, summing over \(v\), and using the bound \(\sum_{v\in\cV_r}\delta_{r,i}(v)\le N_rD_r\), which follows from \(0\le\delta_{r,i}(v)\le D_r\), together with \(D_r/N_r\le1\), we obtain
\begin{equation}
\sum_{v\in\cV_r}\|q_v\|^2
\le
\widehat J_r\kappa_{\partial,r}^2
\sum_{\ell=\nu-D_r}^{\nu-1}
\sum_{(i,h)\in\widehat{\partial}_r}\|\Delta Y_h^\ell\|^2.
\label{eq:app_common_state_bound}
\end{equation}
Finally,
\[
\|\hat Z_r^{\nu+1}-\bar Z_r^{\nu+1}\|^2
=
\sum_{v\in\cV_r}\|m_v\|^2
\le
2\sum_{v\in\cV_r}\|p_v\|^2
+2\sum_{v\in\cV_r}\|q_v\|^2.
\]
Combining~\eqref{eq:app_transport_bound} and~\eqref{eq:app_common_state_bound} proves \eqref{eq:clean_raw_mismatch}. \qed

\subsection{Proof of Proposition~\ref{prop:clean_large_tree_scaling}}
\label{app:clean_large_tree_scaling}

For \(Z\in\cS\), put \(x_i:=\nabla f_i^*(Z_i)\).  We first express the block progress terms in \eqref{eq:clean_block_image_minimizer}.  Let \(h=\{i,j\}\in \overline{\cE}_1\) be a singleton edge block, and write \(\Delta_h\) for the progress of this block.  Changing only \(Y_h\) by \(s\) changes the image coordinates to \(Z_i+\mathsf B_{i,h}s\) and \(Z_j+\mathsf B_{j,h}s\).  Since \(\mathsf B_{i,h}+\mathsf B_{j,h}=0\) and \(|\mathsf B_{i,h}|=1\), the singleton block minimization in \eqref{eq:clean_block_image_minimizer} reduces to
\[
    \min_{u_i+u_j=Z_i+Z_j} f_i^*(u_i)+f_j^*(u_j).
\]
Let \((u_i^h,u_j^h)\) be its minimizer.  The first-order condition gives \(\nabla f_i^*(u_i^h)=\nabla f_j^*(u_j^h)=p_h\).  Thus \(u_i^h=\nabla f_i(p_h)\) and \(u_j^h=\nabla f_j(p_h)\), where \(p_h\) is the unique solution of \(\nabla f_i(p_h)+\nabla f_j(p_h)=Z_i+Z_j\).  By the definition of the block progress and this identity,
\[
\begin{aligned}
    \Delta_h(Z)
&=
f_i^*(Z_i)+f_j^*(Z_j)-f_i^*(u_i^h)-f_j^*(u_j^h)\\
&=
\sum_{k\in\{i,j\}}\left[
f_k^*(Z_k)-f_k^*(\nabla f_k(p_h))
-\inner{p_h,Z_k-\nabla f_k(p_h)}
\right].
\end{aligned}
\]
By the definition of \(x_k\), \(\nabla f_k(x_k)=Z_k\).  Fenchel's identity gives
\[
    f_k(x_k)+f_k^*(Z_k)=\inner{Z_k,x_k},\quad
    f_k(p_h)+f_k^*(\nabla f_k(p_h))=\inner{\nabla f_k(p_h),p_h}.
\]
Substituting these identities yields, for $k\in\{i,j\}$,
\[
    f_k^*(Z_k)-f_k^*(\nabla f_k(p_h))
    -\inner{p_h,Z_k-\nabla f_k(p_h)}
    =
    f_k(p_h)-f_k(x_k)-\inner{\nabla f_k(x_k),p_h-x_k}.
\]
Define the Bregman divergence \(D_{f_i}(u,v):=f_i(u)-f_i(v)-\inner{\nabla f_i(v),u-v}\). Then
\[
    \Delta_h(Z)=D_{f_i}(p_h,x_i)+D_{f_j}(p_h,x_j).
\]
Write \(\|u\|_{R}^2:=u^\top R u\) for \(u\in\R^d\).
The lower Hessian bound in Proposition~\ref{prop:clean_large_tree_scaling} yields
\begin{equation}
\Delta_h(Z)
\ge
\frac{\underline\alpha}{2}
\bigl(\|p_h-x_i\|_{R}^2+\|p_h-x_j\|_{R}^2\bigr)
\ge
\frac{\underline\alpha}{4}\|x_i-x_j\|_{R}^2 .
\label{eq:app_singleton_edge_progress}
\end{equation}
The same argument applied to \eqref{eq:clean_block_image_minimizer} for the tree block yields \(u_i=\nabla f_i(p_0)\), \(i\in\cV_1\), where \(p_0\) is the unique solution of \(\sum_{i\in\cV_1}\nabla f_i(p_0)=\sum_{i\in\cV_1}Z_i\).  Let \(\cI_\partial:=\bigcup_{h\in\overline{\cE}_1}\cI_h\) be the set of agents incident to the singleton edge blocks.  The corresponding progress \(\Delta_1\) satisfies
\begin{equation}
    \Delta_1(Z)=\sum_{i\in\cV_1}D_{f_i}(p_0,x_i)
    \ge
    \frac{\underline\alpha}{2}
    \sum_{i\in\cI_\partial\cap\cV_1}\|x_i-p_0\|_{R}^2 .
\label{eq:app_tree_block_progress}
\end{equation}

We next compare the global gap with the block progress.  By Fenchel's identity,
\[
\begin{aligned}
\sum_{i\in\cV}D_{f_i}(p,x_i)
&=
\sum_{i\in\cV}\left(f_i(p)-f_i(x_i)-\inner{Z_i,p-x_i}\right)=
G(Z)+\sum_{i\in\cV}f_i(p)-\inner{\sum_{i\in\cV}Z_i,p}.
\end{aligned}
\]
Since \(Z\in\cS=\Range(\mathsf B)\), \(\sum_{i\in\cV}Z_i=0\).  Taking \(p=p_0\), using \(\cV\setminus\cV_1\subseteq\cI_\partial\), and applying weak duality for \eqref{eq:clean_prob} and~\eqref{eq:clean_dual}, which gives \(-G^\star\le\sum_{i\in\cV}f_i(p_0)\), we obtain
\[
\begin{aligned}
    G(Z)-G^\star
    &\le G(Z)+\sum_{i\in\cV}f_i(p_0)
    =\sum_{i\in\cV}D_{f_i}(p_0,x_i)=
    \Delta_1(Z)+\sum_{i\in\cV\setminus\cV_1}D_{f_i}(p_0,x_i)\\
    &\le
    \Delta_1(Z)
    +\frac{\overline\alpha}{2}
    \sum_{i\in\cI_\partial}\|x_i-p_0\|_{R}^2 .
\end{aligned}
\]
It remains to bound the finite boundary term by block progress.  Extending \(\cT_1\) to a spanning tree of \(\cG\) gives \(|\cV\setminus\cV_1|\le|\overline{\cE}_1|\). Each vertex outside \(\cV_1\) has a path to \(\cV_1\) with at most \(|\overline{\cE}_1|\) edges in \(\overline{\cE}_1\), ending at a vertex of \(\cI_\partial\cap\cV_1\). Applying Cauchy's inequality along these paths and summing gives, with \(C:=1+|\overline{\cE}_1|(|\overline{\cE}_1|+1)=O(1)\),
\[
    \sum_{i\in\cI_\partial}\|x_i-p_0\|_{R}^2
    \le
    C\left(
    \sum_{i\in\cI_\partial\cap\cV_1}
    \|x_i-p_0\|_{R}^2
    +
    \sum_{\{i,j\}\in\overline{\cE}_1}
    \|x_i-x_j\|_{R}^2
    \right).
\]
Combining the last two bounds with \eqref{eq:app_tree_block_progress} and \eqref{eq:app_singleton_edge_progress} gives
\[
    G(Z)-G^\star
    \le
    \left(1+2C\frac{\overline\alpha}{\underline\alpha}\right)
    \left(\Delta_1(Z)+\sum_{h\in\overline{\cE}_1}\Delta_h(Z)\right),
\]
which is exactly the block error bound for the partition \(\cP=\{\cE_1\}\cup\{\{h\}:h\in\overline{\cE}_1\}\).  Since \(C=O(1)\) and \(\overline\alpha/\underline\alpha=O(1)\), this proves \(\kappa_{\cP}=O(1)\).

Finally we bound the constants in Theorem~\ref{thm:clean_linear_tree}.  The partition has \(p=1+|\overline{\cE}_1|=O(1)\). Since \(\Delta_r(Z)\le G(Z)-G^\star\), the block error bound implies \(\kappa_{\cP}\ge1/p\), and hence \(\kappa_{\cP}=\Theta(1)\).  Singleton blocks have no interior propagation, so \(D_r=0\) for them; only block \(1\) can contribute to \(A\).  For block \(1\), \(\widehat J_1\le2|\overline{\cE}_1|=O(1)\). Every neighboring block is a singleton edge \(\cE_s=\{h\}\), and its incidence column has norm \(\sqrt2\), so \(\Gamma_1\le1\) by its definition.  Moreover, for \(S\subseteq\cV_1\), the map \(\sum_{i\in S}\nabla f_i\) is \(|S|L\)-Lipschitz, while \((\sum_{i\in\cV_1}\nabla f_i)^{-1}\) is \((N_1\mu_1)^{-1}\)-Lipschitz.  Hence \(\kappa_{\partial,1}\le L/\mu_1=O(1)\).  Together with the uniform bounds on \(\mu_1^{-1}\) and \(L\), the definition of \(A\) gives \(A=O(1)\).  Choose \(\tau=\kappa_{\cP}\log D/D=\Theta(\log D/D)\). The first inequality in~\eqref{eq:clean_tree_stepsize} clearly holds for sufficiently large \(D\). Moreover, \(\rho=1-\log D/(2D)\), so \(\rho^{-D}=D^{1/2}(1+o(1))\) and
\[
8LA\tau^2\sum_{j=1}^{D}\rho^{-j}
=16LA\kappa_{\cP}\tau(\rho^{-D}-1)
=O\!\left(\frac{\log D}{\sqrt D}\right)\longrightarrow0.
\]
Thus, \eqref{eq:clean_tree_stepsize} holds for all sufficiently large \(D\), yielding \((1-\rho)^{-1}=2D/\log D\).\qed

\subsection{Proof of Proposition~\ref{prop:clean_ring_partition_scaling}}
\label{app:clean_ring_partition_scaling}

For \(Z\in\cS\), set \(G_0(Z):=\frac12\operatorname{tr}(ZR^{-1}Z^\top)\). Let \(\kappa_{\cP}^{(0)}\) be its smallest admissible block error-bound constant and set \(W:=ZR^{-1/2}\).  This bijection maps \(\cS\) and each \(\cU_r\) onto themselves, so block error-bound constants are invariant under the corresponding change of objective.  Moreover, \(G_0(WR^{1/2})=\frac12\|W\|^2\).  Below, we write \(Z\) for \(W\) and work with \(G_0(Z)=\frac12\|Z\|^2\). We first prove
\begin{equation}
    \frac{D+2}{4\sin^2(\pi/p)}
    \le \kappa_{\cP}^{(0)}
    \le 1+\frac{3(D+1)}{2\sin^2(\pi/p)},
    \qquad
    \kappa_{\cP}^{(0)}
    =\Theta\!\left(\frac{m^2}{D}\right),
    \label{eq:clean_equal_ring_beb}
\end{equation}
For a path block \(\cE_r\), the incidence matrix has rank \(|\cV_r|-1\), and hence
\[
    \cU_r=\left\{U\in\R^{m\times d}:
    U_i=0\ (i\notin\cV_r),\ 
    \sum_{i\in\cV_r}U_i=0\right\}.
\]
By definition, \(\Delta_r^{(0)}(Z)=\frac12\|Z\|^2- \min_{U\in\cU_r}\frac12\|Z+U\|^2\). The minimum is attained when \(Z_i+U_i=\bar Z_{\cV_r}:=|\cV_r|^{-1}\sum_{j\in\cV_r}Z_j\) for \(i\in\cV_r\).  Hence
\[
    \Delta_r^{(0)}(Z)
    =\frac12\sum_{i\in\cV_r}\|Z_i-\bar Z_{\cV_r}\|^2.
\]
Relabel the blocks by \(r=0,\ldots,p-1\) and index the ring nodes cyclically as \(v_0,\ldots,v_{m-1}\).  Let
\[
    Q_r:=\{v_{r(D+1)+j}:0\le j\le D\},
    \qquad b_r:=v_{(r+1)(D+1)},
\]
where indices are taken modulo \(m\).  The sets \(Q_r\) form a disjoint partition of \(\cV\), with \(|Q_r|=D+1\).  Since \(b_r\) is the first node of \(Q_{r+1}\), the nodes incident to block \(r\) are \(\cV_r=Q_r\cup\{b_r\}\).  With \(\xi_r:=(D+1)^{-1}\sum_{i\in Q_r}Z_i\), we have
\[
    2\Delta_r^{(0)}(Z)=\sum_{i\in Q_r}\|Z_i-\xi_r\|^2
    +\frac{D+1}{D+2}\|Z_{b_r}-\xi_r\|^2.
\]
Thus
\begin{equation}
    \sum_{i\in Q_r}\|Z_i-\xi_r\|^2\le2\Delta_r^{(0)}(Z),
    \qquad
    \|Z_{b_r}-\xi_r\|^2\le4\Delta_r^{(0)}(Z).
    \label{eq:clean_ring_local_bounds}
\end{equation}
Since \(b_r\in Q_{r+1}\), \eqref{eq:clean_ring_local_bounds} gives
\[
    \|\xi_{r+1}-\xi_r\|^2
    \le
    2\|\xi_r-Z_{b_r}\|^2
    +2\|Z_{b_r}-\xi_{r+1}\|^2
    \le 8\Delta_r^{(0)}(Z)+4\Delta_{r+1}^{(0)}(Z).
\]
Summing cyclically yields
\begin{equation}
    \sum_{r=0}^{p-1}\|\xi_{r+1}-\xi_r\|^2
    \le 12\sum_{r=0}^{p-1}\Delta_r^{(0)}(Z).
    \label{eq:clean_ring_mean_difference_sum}
\end{equation}
Since \(Z\in\cS\) and the sets \(Q_r\) form a partition of \(\cV\), \(\sum_{r=0}^{p-1}\xi_r=(D+1)^{-1}\sum_{i\in\cV}Z_i=0\). Set \(\Xi:=(\xi_0^\top,\ldots,\xi_{p-1}^\top)^\top\in\R^{p\times d}\) and
\[
    C_p:=
    \begin{bmatrix}
        -1 & 1 & & \\
           & \ddots & \ddots & \\
           & & -1 & 1\\
         1 & & & -1
    \end{bmatrix}.
\]
The cycle spectrum satisfies \(\lambda_{\min}^+(C_p^\top C_p) =4\sin^2(\pi/p)\); see, e.g., \cite{chung1997spectral}.  Therefore,
\begin{equation}
    \sum_{r=0}^{p-1}\|\xi_{r+1}-\xi_r\|^2
    =\|C_p\Xi\|_F^2
    \ge4\sin^2(\pi/p)\|\Xi\|_F^2
    =4\sin^2(\pi/p)\sum_{r=0}^{p-1}\|\xi_r\|^2.
    \label{eq:clean_ring_cycle_spectral_bound}
\end{equation}
Combining \eqref{eq:clean_ring_local_bounds}, \eqref{eq:clean_ring_mean_difference_sum}, and \eqref{eq:clean_ring_cycle_spectral_bound} yields
\[
    \|Z\|^2
    =\sum_{r=0}^{p-1}\left[\sum_{i\in Q_r}\|Z_i-\xi_r\|^2
    +(D+1)\|\xi_r\|^2\right]
    \le2\left(1+\frac{3(D+1)}{2\sin^2(\pi/p)}\right)
    \sum_{r=0}^{p-1}\Delta_r^{(0)}(Z).
\]
This proves the upper bound in \eqref{eq:clean_equal_ring_beb}.

For the lower bound, choose a nonzero \(\Xi\in\R^{p\times d}\) whose columns lie in the eigenspace of \(C_p^\top C_p\) associated with \(4\sin^2(\pi/p)\), and denote its rows by \(\xi_r\).  Then
\[
    \sum_{r=0}^{p-1}\xi_r=0,\qquad
    \sum_{r=0}^{p-1}\|\xi_{r+1}-\xi_r\|^2
    =4\sin^2(\pi/p)\sum_{r=0}^{p-1}\|\xi_r\|^2.
\]
Set \(Z_i=\xi_r\) for \(i\in Q_r\), \(r=0,\ldots,p-1\).  Then \(\sum_{i\in\cV}Z_i=(D+1)\sum_{r=0}^{p-1}\xi_r=0\) and
\[
    \begin{aligned}
    G_0(Z)
    &=\frac12\sum_{r=0}^{p-1}\sum_{i\in Q_r}\|Z_i\|^2
      =\frac{D+1}{2}\sum_{r=0}^{p-1}\|\xi_r\|^2.
    \end{aligned}
\]
Moreover, \(\cV_r=Q_r\cup\{b_r\}\), with \(Z_i=\xi_r\) for \(i\in Q_r\) and \(Z_{b_r}=\xi_{r+1}\), so \(\bar Z_{\cV_r}=((D+1)\xi_r+\xi_{r+1})/(D+2)\).  Substitution into the definition of \(\Delta_r^{(0)}\) yields
\[
    \begin{aligned}
    \Delta_r^{(0)}(Z)
    &=\frac12\left[(D+1)\|\xi_r-\bar Z_{\cV_r}\|^2
      +\|\xi_{r+1}-\bar Z_{\cV_r}\|^2\right]=\frac{D+1}{2(D+2)}\|\xi_{r+1}-\xi_r\|^2.
    \end{aligned}
\]
Hence \(G_0(Z)/\sum_{r=0}^{p-1}\Delta_r^{(0)}(Z) =(D+2)/(4\sin^2(\pi/p))\) and \eqref{eq:clean_equal_ring_beb} follows.

It remains to transfer the bound to \(G\).  Define \(\widehat G(W):=G(WR^{1/2})\) for \(W\in\cS\). The Hessian bounds in Proposition~\ref{prop:clean_ring_partition_scaling} give \(\overline\alpha^{-1}I\preceq \nabla^2\!\left[f_i^*(R^{1/2}\,\cdot)\right](W_i) \preceq\underline\alpha^{-1}I\).  Let \(P_{\cS}\) and \(P_r\) be the Euclidean projectors onto \(\cS\) and \(\cU_r\), respectively, and define \(\widehat\Delta_r(W):=\widehat G(W)- \min_{U\in W+\cU_r}\widehat G(U)\). For \(G_0(V)=\frac12\|V\|^2\), the minimizer over \(V+\cU_r\) is \(V-P_rV\); hence \(\Delta_r^{(0)}(V)=\frac12\|V\|^2-\frac12\|V-P_rV\|^2 =\frac12\|P_rV\|^2\).  The quadratic block error bound therefore gives \(\sum_{r=0}^{p-1}\|P_rV\|^2\ge (\kappa_{\cP}^{(0)})^{-1}\|V\|^2\) for \(V\in\cS\).  Since \(\cU_r\subseteq\cS\), \(P_rP_{\cS}=P_r\).  The smoothness and strong convexity bounds above then yield
\[
\begin{aligned}
    \sum_{r=0}^{p-1}\widehat\Delta_r(W)
    &\ge\frac{\underline\alpha}{2}
    \sum_{r=0}^{p-1}\|P_r\nabla\widehat G(W)^\top\|^2\ge\frac{\underline\alpha}{2\kappa_{\cP}^{(0)}}
    \|P_{\cS}\nabla\widehat G(W)^\top\|^2
    \ge\frac{\underline\alpha}{\overline\alpha\kappa_{\cP}^{(0)}}
    \bigl(\widehat G(W)-G^\star\bigr).
\end{aligned}
\]
Therefore \(\kappa_{\cP}\le (\overline\alpha/\underline\alpha)\kappa_{\cP}^{(0)}\), and \eqref{eq:clean_equal_ring_beb} proves \(\kappa_{\cP}=O(m^2/D)\). \qed

{
To obtain the rate bound, we apply Theorem~\ref{thm:clean_linear_tree}
to \(\widehat f_i(u):=f_i(R^{-1/2}u)\), which are
\(\underline\alpha\)-strongly convex and \(\overline\alpha\)-smooth.
Thus \(\kappa_{\partial,r}\le\overline\alpha/\underline\alpha\).
With two boundary incidences, at most two neighbors per block, and
\(\Gamma_r\le1\), \eqref{eq:clean_tree_A} gives
\[
    LA\le4\left(2\frac{\overline\alpha}{\underline\alpha}+1\right)
    \left(\frac{\overline\alpha}{\underline\alpha}\right)^2=O(1).
\]
With the same stepsize and initial data \(Y^0R^{-1/2}\) and
\(\mu_{i\to e}^0(\,\cdot\,R^{1/2})\),
Algorithm~\ref{alg:gate_uniform_proposed} for \(\widehat f_i\) generates
edge iterates \(Y^\nu R^{-1/2}\) and primal iterates \(X^\nu R^{1/2}\).
These fixed invertible transformations change only the constant \(C\)
in the convergence bound, so the same \(\rho\) applies to the original iterates.
}


\section{Proofs of Results in Sec.~\ref{sec:clean_surrogate}}
\label{app:clean_surrogate_proofs}

For the proof notation, fix a block \(r\). Write \(N_r:=|\cV_r|\), \(J_r:=|\partial_r|\), and \(\widehat J_r:=|\widehat\partial_r|\).  For \(E\in\R^{\widehat J_r\times d}\) with rows indexed by \((i,h)\in\widehat{\partial}_r\), use the boundary contribution \(c_i(E)=\sum_{h\in\cA_i^\partial}\mathsf B_{i,h}E_{i,h}\) from Appendix~\ref{app:clean_tree_oracle}, and write \(c_r(E):=(c_i(E))_{i\in\cV_r}\).  For \(U,\vartheta\in\R^{N_r\times d}\), the exact and surrogate block slices are
\[
G_r(U):=\sum_{i\in\cV_r}f_i^*(U_i),
\quad
\widetilde G_r(U;\vartheta):=\sum_{i\in\cV_r}\widetilde f_i^*(U_i;\vartheta_i).
\]
Only nodes in \(\cV_r\) appear in these slices; the remaining terms are constant under a block-\(r\) update.
For the block solution map, write
\[
\begin{aligned}
\widetilde{\mathsf{Y}}_{r,e}(E,\vartheta)&:=[\widetilde Y_r(E,\vartheta)]_e,
&&e\in\cE_r,\\
\widetilde{\mathsf{Z}}_{r,v}(E,\vartheta)&:=\mathsf B_{v,\cE_r}\widetilde Y_r(E,\vartheta),
&&v\in\cV_r.
\end{aligned}
\]

\subsection{Proof of Lemma~\ref{lem:clean_surrogate_tree_sensitivity}}
\label{subsec:app_sur_tree_center}

\noindent\textit{Proof of \eqref{eq:clean_sur_center_edge_sens}--\eqref{eq:clean_sur_center_node_sens_nonlocal}.} Fix a block \(r\).  For \(E\in\R^{\widehat J_r\times d}\) with rows indexed by \((i,h)\in\widehat{\partial}_r\), set \(\eta(E):=\sum_{i\in\cV_r}c_i(E)\). Suppose \(\vartheta,\widehat\vartheta\in\R^{N_r\times d}\) differ only in coordinate \(i\).  Define
\[
u_j^0:=\widetilde{\mathsf{Z}}_{r,j}(E,\vartheta)+c_j(E),\qquad
u_j^1:=\widetilde{\mathsf{Z}}_{r,j}(E,\widehat\vartheta)+c_j(E).
\]
Applying the same optimality condition for each active edge variable as in the proof of Lemma~\ref{lem:app_tree_oracle_jacobians} to the two minimizers \(\widetilde Y_r(E,\vartheta)\) and \(\widetilde Y_r(E,\widehat\vartheta)\), with \(f_j^*\) replaced by \(\widetilde f_j^*(\cdot;\vartheta_j)\) and \(\widetilde f_j^*(\cdot;\widehat\vartheta_j)\), respectively, gives common values \(x^0,x^1\in\R^d\) such that
\[
\begin{gathered}
u_j^0=
\bigl(\nabla_u\widetilde f_j^*(\cdot;\vartheta_j)\bigr)^{-1}(x^0),
\quad
u_j^1=
\bigl(\nabla_u\widetilde f_j^*(\cdot;\widehat\vartheta_j)\bigr)^{-1}(x^1),
\quad j\in\cV_r,\\
\sum_{j\in\cV_r}u_j^0=\sum_{j\in\cV_r}u_j^1=\eta(E).
\end{gathered}
\]
By the definition of \(\widetilde\cL_{r,S}\) in~\eqref{eq:clean_sur_aggregate_map},
\[
u_j^0=\widetilde\cL_{r,\{j\}}(\eta(E);\vartheta),
\qquad
u_j^1=\widetilde\cL_{r,\{j\}}(\eta(E);\widehat\vartheta).
\]
Using~\eqref{eq:clean_sur_kappa_center} at the fixed input \(\eta(E)\) gives
\[
\|\widetilde{\mathsf{Z}}_{r,i}(E,\widehat\vartheta)
-\widetilde{\mathsf{Z}}_{r,i}(E,\vartheta)\|
=
\left\|
\widetilde\cL_{r,\{i\}}(\eta(E);\widehat\vartheta)
-\widetilde\cL_{r,\{i\}}(\eta(E);\vartheta)
\right\|
\le
\widetilde\kappa_{c,r}\|\widehat\vartheta_i-\vartheta_i\|.
\]
For \(v\ne i\),
\[
\|\widetilde{\mathsf{Z}}_{r,v}(E,\widehat\vartheta)
-\widetilde{\mathsf{Z}}_{r,v}(E,\vartheta)\|
=
\left\|
\widetilde\cL_{r,\{v\}}(\eta(E);\widehat\vartheta)
-\widetilde\cL_{r,\{v\}}(\eta(E);\vartheta)
\right\|
\le
\frac{\widetilde\kappa_{c,r}}{N_r}
\|\widehat\vartheta_i-\vartheta_i\|.
\]
By Lemma~\ref{lem:clean_tree_inverse}, for some sign \(\varepsilon_{e,i}\in\{\pm1\}\),
\[
\widetilde{\mathsf{Y}}_{r,e}(E,\widehat\vartheta)
-\widetilde{\mathsf{Y}}_{r,e}(E,\vartheta)
=
\varepsilon_{e,i}
\sum_{v\in\overline{\cV}_e(i)}(u_v^1-u_v^0).
\]
The choice \(i\notin\overline{\cV}_e(i)\) matches the edge term in \eqref{eq:clean_sur_kappa_center}: the center perturbation is in coordinate \(i\), while the aggregate is over the opposite component \(\overline{\cV}_e(i)\).  By \eqref{eq:clean_sur_aggregate_map},
\[
\sum_{v\in\overline{\cV}_e(i)}(u_v^1-u_v^0)
=
\widetilde\cL_{r,\overline{\cV}_e(i)}(\eta(E);\widehat\vartheta)
-\widetilde\cL_{r,\overline{\cV}_e(i)}(\eta(E);\vartheta).
\]
Therefore
\[
\|\widetilde{\mathsf{Y}}_{r,e}(E,\widehat\vartheta)
-\widetilde{\mathsf{Y}}_{r,e}(E,\vartheta)\|
\le
\widetilde\kappa_{c,r}
\frac{|\overline{\cV}_e(i)|}{N_r}\|\widehat\vartheta_i-\vartheta_i\|.
\]
This proves \eqref{eq:clean_sur_center_edge_sens}--\eqref{eq:clean_sur_center_node_sens_nonlocal}.

\smallskip
\noindent\textit{Proof of \eqref{eq:clean_sur_boundary_edge_sens}--\eqref{eq:clean_sur_boundary_node_sens}.} Fix a center tuple \(\vartheta=(\vartheta_i)_{i\in\cV_r}\), and suppose \(E,E'\) differ only in boundary coordinate \((i,h)\).  With \(\vartheta\) fixed, \eqref{eq:clean_sur_aggregate_map} is the surrogate counterpart of the aggregate map used in Lemma~\ref{lem:clean_boundary_sensitivity}.  Repeating the proof of Lemma~\ref{lem:clean_boundary_sensitivity} with \(\cL_{r,S}\) replaced by \(\widetilde\cL_{r,S}(\cdot;\vartheta)\) gives
\[
\|\widetilde{\mathsf{Y}}_{r,e}(E',\vartheta)
-\widetilde{\mathsf{Y}}_{r,e}(E,\vartheta)\|
\le
\widetilde\kappa_{\partial,r}\frac{|\overline{\cV}_e(i)|}{N_r}\|E'_{i,h}-E_{i,h}\|
\]
and
\[
\|\widetilde{\mathsf{Z}}_{r,v}(E',\vartheta)
-\widetilde{\mathsf{Z}}_{r,v}(E,\vartheta)\|
\le
\widetilde\kappa_{\partial,r}\frac1{N_r}\|E'_{i,h}-E_{i,h}\|,\quad v\ne i.
\]
This proves \eqref{eq:clean_sur_boundary_edge_sens}--\eqref{eq:clean_sur_boundary_node_sens}.\qed

\smallskip
{
\noindent\textbf{Remark on \(\widetilde\kappa_{c,r}\).} The constant \(\widetilde\kappa_{c,r}\) in~\eqref{eq:clean_sur_kappa_center} measures the center sensitivity of the frozen aggregate maps and can reflect the associated surrogate approximation quality. It vanishes for the exact surrogate \(\widetilde f_i^*(U_i;Z_i)=f_i^*(U_i)\). Finiteness requires control of the center dependence in addition to the curvature bounds in Assumption~\ref{ass:clean_surrogate_nodewise}(iii). When \(\widetilde f_j^*\in C^2\) in \((u,\theta)\) for \(j\in\cV_r\), fix \(i\in\cV_r\), \(\eta\in\R^d\), and a center tuple \(\vartheta\in\R^{|\cV_r|\times d}\).  To identify the local derivative entering \(\operatorname{Lip}_{\vartheta_i}\) in \eqref{eq:clean_sur_kappa_center}, vary only the \(i\)-th center, denoted by \(\theta\).  By~\eqref{eq:clean_sur_aggregate_map}, along this one-coordinate variation, for any nonempty \(S\subseteq\cV_r\),
\[
\left.\widetilde\cL_{r,S}(\eta;\vartheta)\right|_{\vartheta_i=\theta}
=\sum_{j\in S}u_j(\theta),
\]
where \(u_j(\theta)\), \(j\in\cV_r\), are determined by
\[
\nabla_u\widetilde f_j^*(u_j(\theta);\vartheta_j)=x(\theta)\quad(j\ne i),\qquad
\nabla_u\widetilde f_i^*(u_i(\theta);\theta)=x(\theta),\qquad
\sum_{j\in\cV_r}u_j(\theta)=\eta.
\]
Applying the implicit function theorem at \(\theta=\vartheta_i\), write
\[
H_j:=
\left(\nabla_{uu}^2\widetilde f_j^*(u_j(\vartheta_i);\vartheta_j)\right)^{-1}
\quad(j\ne i),
\]
\[
H_i:=
\left(\nabla_{uu}^2\widetilde f_i^*(u_i(\vartheta_i);\vartheta_i)\right)^{-1},\;
C_i:=\bigl(\nabla_\theta[\nabla_u\widetilde f_i^*(u_i(\vartheta_i);\theta)]\bigr)^\top\big|_{\theta=\vartheta_i},\;
H_\Sigma:=\sum_{j\in\cV_r}H_j.
\]
Differentiating the defining equations gives
\[
\nabla_\theta u_j(\vartheta_i)^\top
=H_j\nabla_\theta x(\vartheta_i)^\top\quad(j\ne i),\qquad
\nabla_\theta u_i(\vartheta_i)^\top
=H_i\nabla_\theta x(\vartheta_i)^\top-H_iC_i.
\]
Differentiating \(\sum_{j\in\cV_r}u_j(\theta)=\eta\) then yields
\[
0=H_\Sigma\nabla_\theta x(\vartheta_i)^\top-H_iC_i,
\qquad
\nabla_\theta x(\vartheta_i)^\top=H_\Sigma^{-1}H_iC_i.
\]
Therefore, for every component \(\overline{\cV}_e(i)\) not containing \(i\),
\[
\nabla_\theta\!\left(\sum_{v\in\overline{\cV}_e(i)}u_v\right)(\vartheta_i)^\top
=
\left(\sum_{v\in\overline{\cV}_e(i)}H_v\right)H_\Sigma^{-1}H_iC_i,
\]
\[
\nabla_\theta u_i(\vartheta_i)^\top
=H_i(H_\Sigma^{-1}H_i-I)C_i,\qquad
\nabla_\theta u_v(\vartheta_i)^\top
=H_vH_\Sigma^{-1}H_iC_i\quad(v\ne i).
\]
Thus \(\widetilde\kappa_{c,r}\) is certified by any upper bound on
\[
\begin{aligned}
\sup_{\eta,\vartheta}
\max_{i\in\cV_r}
\Bigg\{&
\max_{e\in\cE_r}
\frac{|\cV_r|}{|\overline{\cV}_e(i)|}
\left\|
\left(\sum_{v\in\overline{\cV}_e(i)}H_v\right)
H_\Sigma^{-1}H_iC_i
\right\|,
\\
&
\left\|H_i(H_\Sigma^{-1}H_i-I)C_i\right\|,
\quad
\max_{v\ne i}|\cV_r|
\left\|H_vH_\Sigma^{-1}H_iC_i\right\|
\Bigg\}.
\end{aligned}
\]
Thus \(\widetilde\kappa_{c,r}\) exhibits a preconditioning effect analogous to that in Remark~\ref{rem:clean_kappa_partial}.
The additional factor \(C_i\) captures the center sensitivity of the local surrogate gradient. Assumption~\ref{ass:clean_surrogate_nodewise}(iv) ensures that the matrices \(C_i\) are uniformly bounded. Together with the uniform curvature bounds in Assumption~\ref{ass:clean_surrogate_nodewise}(iii), this ensures that \(\widetilde\kappa_{c,r}<\infty\).}

\subsection{Proof of Lemma~\ref{lem:clean_surrogate_tree_mismatch}}
\label{subsec:app_sur_boundary_center_mismatch}

Fix a block \(r\) and an iteration \(\nu\ge\max_{s\in[p]}\operatorname{diam}(\cT_s)\).  For \(i,v\in\cV_r\), set \(d_i(v):=0\) if \(i=v\), and \(d_i(v):=\min_{e\in\cA_v^\circ}\delta_{e,i}\) otherwise.
Use the node-aligned boundary vector \(E_v^\nu\) from the proof of Lemma~\ref{lem:clean_tree_mismatch}, and set \(Z_{r,v}^\nu:=(Z_i^{\nu-d_i(v)})_{i\in\cV_r}\) for the center coordinates. In particular, \([Z_{r,v}^\nu]_v=Z_v^\nu\).

\smallskip
\noindent\textbf{Mismatch decomposition.} For \(v\in\cV_r\), define the node mismatch
\[
    m_v:=(\hat Z_r^{\nu+1}-\bar Z_r^{\nu+1})_v.
\]
The pair \((\widehat E_e^\nu,Z_{r,e}^\nu)\) depends on the incident edge \(e\), whereas \(m_v\) fixes the node \(v\).  To estimate \(m_v\), insert the node-aligned pair \((E_v^\nu,Z_{r,v}^\nu)\). Since
\[
(\hat Z_r^{\nu+1})_v
=
\sum_{e\in\cA_v^\circ}
\mathsf B_{v,e}\widetilde{\mathsf{Y}}_{r,e}
(\widehat E_e^\nu,Z_{r,e}^\nu)
\]
and \((\bar Z_r^{\nu+1})_v=\widetilde{\mathsf{Z}}_{r,v}(E^\nu,Z_{\cV_r}^\nu)\), adding and subtracting
\[
\sum_{e\in\cA_v^\circ}
\mathsf B_{v,e}\widetilde{\mathsf{Y}}_{r,e}(E_v^\nu,Z_{r,v}^\nu)
=
\widetilde{\mathsf{Z}}_{r,v}(E_v^\nu,Z_{r,v}^\nu)
\]
gives the exact decomposition
\begin{equation*}
    m_v=p_v+q_v,
\end{equation*}
where
\begin{align}
p_v
&:=
\sum_{e\in\cA_v^\circ}\mathsf B_{v,e}
\left[
\widetilde{\mathsf{Y}}_{r,e}(\widehat E_e^\nu,Z_{r,e}^\nu)
-
\widetilde{\mathsf{Y}}_{r,e}(E_v^\nu,Z_{r,v}^\nu)
\right],
\label{eq:app_sur_pv}\\
q_v
&:=
\widetilde{\mathsf{Z}}_{r,v}(E_v^\nu,Z_{r,v}^\nu)
-
\widetilde{\mathsf{Z}}_{r,v}(E^\nu,Z_{\cV_r}^\nu).
\label{eq:app_sur_qv}
\end{align}
\smallskip

\noindent\textbf{Estimate of \(p_v\).} Insert \(\widetilde{\mathsf{Y}}_{r,e}(E_v^\nu,Z_{r,e}^\nu)\) in each summand of~\eqref{eq:app_sur_pv}.  This gives
\[
p_v=p_v^\partial+p_v^c,
\]
where
\[
\begin{aligned}
p_v^\partial
&:=
\sum_{e\in\cA_v^\circ}\mathsf B_{v,e}
\left[
\widetilde{\mathsf{Y}}_{r,e}(\widehat E_e^\nu,Z_{r,e}^\nu)
-
\widetilde{\mathsf{Y}}_{r,e}(E_v^\nu,Z_{r,e}^\nu)
\right],\\
p_v^c
&:=
\sum_{e\in\cA_v^\circ}\mathsf B_{v,e}
\left[
\widetilde{\mathsf{Y}}_{r,e}(E_v^\nu,Z_{r,e}^\nu)
-
\widetilde{\mathsf{Y}}_{r,e}(E_v^\nu,Z_{r,v}^\nu)
\right].
\end{aligned}
\]
The term \(p_v^\partial\) changes only the boundary vector.  For \(e\in\cA_v^\circ\) and \((i,h)\in\widehat{\partial}_r\), the delay definition and \eqref{eq:app_rooted_boundary_delay} give \(\delta_{e,i}\in\{\delta_{r,i}(v),\delta_{r,i}(v)+1\}\).  Hence the same argument in~\eqref{eq:app_transport_bound} gives
\begin{equation}
\sum_{v\in\cV_r}\|p_v^\partial\|^2
\le
\widehat J_r\widetilde\kappa_{\partial,r}^2
\sum_{\ell=\nu-\widetilde D_r}^{\nu-1}
\sum_{(i,h)\in\widehat{\partial}_r}\|Y_h^{\ell+1}-Y_h^\ell\|^2.
\label{eq:app_sur_pv_boundary}
\end{equation}
For the center part, write \(\Delta Z_i^\ell:=Z_i^{\ell+1}-Z_i^\ell\).  For each \(i\in\cV_r\), root the tree at \(i\), and let \(\mathrm{Ch}_{r,i}(v):=\{e\in\cA_v^\circ:\delta_{e,i}=d_i(v)+1\}\), whose elements are among the child edges of \(v\) in this rooted tree.  For \(v\ne i\), let \(p_i(v)\) be the active edge from \(v\) to its parent, and set
\[
n_{r,i}(i):=N_r,\qquad
n_{r,i}(v):=|\overline{\cV}_{p_i(v)}(i)|,
\quad v\ne i.
\]
Since \(\delta_{e,i}=0\) for every \(e\in\cA_i^\circ\), \(\mathrm{Ch}_{r,i}(i)=\emptyset\).
For each \(e\in\cA_v^\circ\), telescope differences from \(Z_{r,v}^\nu\) to \(Z_{r,e}^\nu\) by changing center coordinates one at a time.  In the coordinate-\(i\) step, let \(\vartheta_{e,i}^-\) and \(\vartheta_{e,i}^+\) denote the consecutive center tuples, and define
\[
p_{v,i}^c
:=
\sum_{e\in\cA_v^\circ}\mathsf B_{v,e}
\left[
\widetilde{\mathsf{Y}}_{r,e}(E_v^\nu,\vartheta_{e,i}^+)
-
\widetilde{\mathsf{Y}}_{r,e}(E_v^\nu,\vartheta_{e,i}^-)
\right].
\]
Then \(p_v^c=\sum_{i\in\cV_r}p_{v,i}^c\).  Since \(\delta_{e,i}\in\{d_i(v),d_i(v)+1\}\) for every \(e\in\cA_v^\circ\), the \(i\)-coordinate difference is zero unless \(e\in\mathrm{Ch}_{r,i}(v)\).  If \(e\in\mathrm{Ch}_{r,i}(v)\), then
\[
(\vartheta_{e,i}^+-\vartheta_{e,i}^-)_i
=
[Z_{r,e}^\nu]_i-[Z_{r,v}^\nu]_i
=
Z_i^{\nu-\delta_{e,i}}-Z_i^{\nu-d_i(v)}
=-\Delta Z_i^{\nu-d_i(v)-1}.
\]
If \(\mathrm{Ch}_{r,i}(v)=\emptyset\), then \(p_{v,i}^c=0\).  Otherwise \(d_i(v)+1\le\widetilde D_r\), and the sensitivity bound~\eqref{eq:clean_sur_center_edge_sens} together with the counting argument for \(v\ne i\) in~\eqref{eq:app_rooted_counts} gives
\[
\begin{aligned}
\|p_{v,i}^c\|^2
&\le
\widetilde\kappa_{c,r}^2
\frac{\left(\sum_{e\in\mathrm{Ch}_{r,i}(v)}|\overline{\cV}_e(i)|\right)^2}{N_r^2}
\|\Delta Z_i^{\nu-d_i(v)-1}\|^2\\
&\le
\widetilde\kappa_{c,r}^2
\frac{(n_{r,i}(v)-1)^2}{N_r^2}
\|\Delta Z_i^{\nu-d_i(v)-1}\|^2.
\end{aligned}
\]
By Cauchy-Schwarz,
\[
\|p_v^c\|^2
\le
N_r\widetilde\kappa_{c,r}^2
\sum_{\substack{i\in\cV_r,\,\mathrm{Ch}_{r,i}(v)\ne\emptyset}}
\frac{(n_{r,i}(v)-1)^2}{N_r^2}
\|\Delta Z_i^{\nu-d_i(v)-1}\|^2.
\]
Summing over \(v\) and rearranging terms gives
\begin{align}
\sum_{v\in\cV_r}\|p_v^c\|^2
&\le
N_r\widetilde\kappa_{c,r}^2
\sum_{i\in\cV_r}
\sum_{v\in\cV_r:\,\mathrm{Ch}_{r,i}(v)\ne\emptyset}
\frac{(n_{r,i}(v)-1)^2}{N_r^2}
\|\Delta Z_i^{\nu-d_i(v)-1}\|^2
\notag\\
&=
N_r\widetilde\kappa_{c,r}^2
\sum_{i\in\cV_r}\sum_{t=1}^{\widetilde D_r}
\left(
\sum_{v\in\cV_r:\,\mathrm{Ch}_{r,i}(v)\ne\emptyset,\,d_i(v)+1=t}
\frac{(n_{r,i}(v)-1)^2}{N_r^2}
\right)
\|\Delta Z_i^{\nu-t}\|^2
\notag\\
&\le
N_r\widetilde\kappa_{c,r}^2
\sum_{i\in\cV_r}\sum_{t=1}^{\widetilde D_r}
\|\Delta Z_i^{\nu-t}\|^2=
N_r\widetilde\kappa_{c,r}^2
\sum_{\ell=\nu-\widetilde D_r}^{\nu-1}
\sum_{i\in\cV_r}\|Z_i^{\ell+1}-Z_i^\ell\|^2,
\label{eq:app_sur_pv_center}
\end{align}
Here \(d_i(v)=\max\{\dist_{\cT_r}(i,v)-1,0\}\), so each inner sum is over vertices at depth \(t\) and is at most one by~\eqref{eq:app_rooted_counts}. Combining~\eqref{eq:app_sur_pv_boundary} and~\eqref{eq:app_sur_pv_center} yields
\begin{equation}
\sum_{v\in\cV_r}\|p_v\|^2
\le
2\widehat J_r\widetilde\kappa_{\partial,r}^2
\sum_{\ell=\nu-\widetilde D_r}^{\nu-1}
\sum_{(i,h)\in\widehat{\partial}_r}\|Y_h^{\ell+1}-Y_h^\ell\|^2
+
2N_r\widetilde\kappa_{c,r}^2
\sum_{\ell=\nu-\widetilde D_r}^{\nu-1}
\sum_{i\in\cV_r}\|Z_i^{\ell+1}-Z_i^\ell\|^2.
\label{eq:app_sur_pv_bound}
\end{equation}

\smallskip

\noindent\textbf{Estimate of \(q_v\).} Split \(q_v=q_v^\partial+q_v^c\), where
\[
q_v^\partial
:=
\widetilde{\mathsf{Z}}_{r,v}(E_v^\nu,Z_{r,v}^\nu)
-
\widetilde{\mathsf{Z}}_{r,v}(E^\nu,Z_{r,v}^\nu),
\quad
q_v^c
:=
\widetilde{\mathsf{Z}}_{r,v}(E^\nu,Z_{r,v}^\nu)
-
\widetilde{\mathsf{Z}}_{r,v}(E^\nu,Z_{\cV_r}^\nu).
\]
With \(Z_{r,v}^\nu\) fixed, the proof of Lemma~\ref{lem:clean_tree_mismatch} gives the same \(q_v\)-estimate with \(\widetilde\kappa_{\partial,r}\) in place of \(\kappa_{\partial,r}\):
\begin{equation}
\sum_{v\in\cV_r}\|q_v^\partial\|^2
\le
\widehat J_r\widetilde\kappa_{\partial,r}^2
\sum_{\ell=\nu-\widetilde D_r}^{\nu-1}
\sum_{(i,h)\in\widehat{\partial}_r}\|Y_h^{\ell+1}-Y_h^\ell\|^2.
\label{eq:app_sur_qv_boundary}
\end{equation}
For the center part, keep the boundary vector fixed at \(E^\nu\) and change \(Z_{\cV_r}^\nu\) to \(Z_{r,v}^\nu\) one center coordinate at a time. Let \(q_{v,i}^c\) be the change in \(\widetilde{\mathsf{Z}}_{r,v}\) during the \(i\)-coordinate replacement.  Since \(d_v(v)=0\), we have \(q_{v,v}^c=0\). Thus \(q_v^c=\sum_{i\in\cV_r\setminus\{v\}}q_{v,i}^c\), and~\eqref{eq:clean_sur_center_node_sens_nonlocal} gives \(\|q_{v,i}^c\|\le\frac{\widetilde\kappa_{c,r}}{N_r}\|Z_i^{\nu-d_i(v)}-Z_i^\nu\|\) for \(i\ne v\). Hence, using Cauchy--Schwarz and \((N_r-1)/N_r^2\le1/N_r\),
\begin{equation}
\|q_v^c\|^2
\le
\frac{\widetilde\kappa_{c,r}^2}{N_r}
\sum_{i\in\cV_r}\|Z_i^{\nu-d_i(v)}-Z_i^\nu\|^2.
\label{eq:app_sur_qvc_point}
\end{equation}
For every \(i\), telescoping gives
\begin{equation*}
\|Z_i^{\nu-d_i(v)}-Z_i^\nu\|^2
\le
d_i(v)\sum_{m=1}^{d_i(v)}
\|Z_i^{\nu-m+1}-Z_i^{\nu-m}\|^2,
\quad 0\le d_i(v)\le \widetilde D_r.
\end{equation*}
Therefore
\begin{equation}
\frac1{N_r}\sum_{v\in\cV_r}\sum_{i\in\cV_r}
\|Z_i^{\nu-d_i(v)}-Z_i^\nu\|^2
\le
\widetilde D_r
\sum_{\ell=\nu-\widetilde D_r}^{\nu-1}
\sum_{i\in\cV_r}\|Z_i^{\ell+1}-Z_i^\ell\|^2,
\label{eq:app_sur_center_diff_sums}
\end{equation}
where we used \(\sum_{v\in\cV_r}d_i(v)\le N_r\widetilde D_r\).  Combining~\eqref{eq:app_sur_qvc_point} and~\eqref{eq:app_sur_center_diff_sums} yields
\begin{equation*}
\sum_{v\in\cV_r}\|q_v^c\|^2
\le
\widetilde D_r\widetilde\kappa_{c,r}^2
\sum_{\ell=\nu-\widetilde D_r}^{\nu-1}
\sum_{i\in\cV_r}\|Z_i^{\ell+1}-Z_i^\ell\|^2.
\end{equation*}
Together with~\eqref{eq:app_sur_qv_boundary},
\begin{equation}
\sum_{v\in\cV_r}\|q_v\|^2
\le
2\widehat J_r\widetilde\kappa_{\partial,r}^2
\sum_{\ell=\nu-\widetilde D_r}^{\nu-1}
\sum_{(i,h)\in\widehat{\partial}_r}\|Y_h^{\ell+1}-Y_h^\ell\|^2
+
2\widetilde D_r\widetilde\kappa_{c,r}^2
\sum_{\ell=\nu-\widetilde D_r}^{\nu-1}
\sum_{i\in\cV_r}\|Z_i^{\ell+1}-Z_i^\ell\|^2.
\label{eq:app_sur_qv_bound}
\end{equation}
Since
\[
\|\hat Z_r^{\nu+1}-\bar Z_r^{\nu+1}\|^2
=
\sum_{v\in\cV_r}\|m_v\|^2
\le
2\sum_{v\in\cV_r}\|p_v\|^2
+
2\sum_{v\in\cV_r}\|q_v\|^2,
\]
combining~\eqref{eq:app_sur_pv_bound} and~\eqref{eq:app_sur_qv_bound} gives
\begin{align}
\|\hat Z_r^{\nu+1}-\bar Z_r^{\nu+1}\|^2
&\le
8\widehat J_r\widetilde\kappa_{\partial,r}^2
\sum_{\ell=\nu-\widetilde D_r}^{\nu-1}
\sum_{(i,h)\in\widehat{\partial}_r}\|Y_h^{\ell+1}-Y_h^\ell\|^2
\nonumber\\
&\quad+
4(N_r+\widetilde D_r)\widetilde\kappa_{c,r}^2
\sum_{\ell=\nu-\widetilde D_r}^{\nu-1}
\sum_{i\in\cV_r}\|Z_i^{\ell+1}-Z_i^\ell\|^2.
\label{eq:app_sur_raw_boundary_center_mismatch}
\end{align}

\smallskip
\noindent\textbf{Reduction to block increments.} It remains to replace the coordinate increments by block and boundary increments. Since \(Z^\ell=\mathsf BY^\ell\), \(Z_r^\ell=\mathsf B_{\cV,\cE_r}Y_{\cE_r}^\ell\), and \(\cA_i=\cA_i^\circ\sqcup\cA_i^\partial\), for \(i\in\cV_r\),
\[
\begin{aligned}
Z_i^{\ell+1}-Z_i^\ell
&=
\sum_{h\in\cA_i}\mathsf B_{i,h}(Y_h^{\ell+1}-Y_h^\ell)=(Z_r^{\ell+1}-Z_r^\ell)_i
+
\sum_{h\in\cA_i^\partial}\mathsf B_{i,h}(Y_h^{\ell+1}-Y_h^\ell).
\end{aligned}
\]
Cauchy's inequality, \(J_{r,i}:=|\cA_i^\partial|\le J_r\), and \(\widehat{\partial}_r=\bigsqcup_{i\in\cV_r}(\{i\}\times\cA_i^\partial)\) give
\begin{equation*}
\sum_{i\in\cV_r}\|Z_i^{\ell+1}-Z_i^\ell\|^2
\le
2\|Z_r^{\ell+1}-Z_r^\ell\|^2
+2J_r\sum_{(i,h)\in\widehat{\partial}_r}\|Y_h^{\ell+1}-Y_h^\ell\|^2.
\end{equation*}
Summing over \(\ell=\nu-\widetilde D_r,\ldots,\nu-1\) and substituting into~\eqref{eq:app_sur_raw_boundary_center_mismatch} gives
\begin{align}
\|\hat Z_r^{\nu+1}-\bar Z_r^{\nu+1}\|^2
&\le
\left(
8\widetilde\kappa_{\partial,r}^2\widehat J_r
+8J_r(N_r+\widetilde D_r)\widetilde\kappa_{c,r}^2
\right)
\sum_{\ell=\nu-\widetilde D_r}^{\nu-1}
\sum_{(i,h)\in\widehat{\partial}_r}\|Y_h^{\ell+1}-Y_h^\ell\|^2
\nonumber\\*
&\quad+
8(N_r+\widetilde D_r)\widetilde\kappa_{c,r}^2
\sum_{\ell=\nu-\widetilde D_r}^{\nu-1}
\|Z_r^{\ell+1}-Z_r^\ell\|^2.
\label{eq:app_sur_raw_boundary_center_mismatch_image}
\end{align}
By~\eqref{eq:clean_increment_conversion},
\[
\sum_{\substack{(i,h)\in\widehat{\partial}_r\\h\in\cE_s}}
\|Y_h^{\ell+1}-Y_h^\ell\|^2
\le\Gamma_r\|Z_s^{\ell+1}-Z_s^\ell\|^2,
\quad s\in\cM_r.
\]
Since \(\widehat{\partial}_r=\bigsqcup_{s\in\cM_r}\{(i,h)\in\widehat{\partial}_r:h\in\cE_s\}\), summing this estimate over \(s\) and \(\ell=\nu-\widetilde D_r,\ldots,\nu-1\) yields
\[
\sum_{\ell=\nu-\widetilde D_r}^{\nu-1}
\sum_{(i,h)\in\widehat{\partial}_r}\|Y_h^{\ell+1}-Y_h^\ell\|^2
\le
\Gamma_r
\sum_{\ell=\nu-\widetilde D_r}^{\nu-1}
\sum_{s\in\cM_r}
\|Z_s^{\ell+1}-Z_s^\ell\|^2.
\]
Substitution into~\eqref{eq:app_sur_raw_boundary_center_mismatch_image}, together with the definitions of \(\widetilde M_{r,\partial}\) and \(\widetilde M_{r,c}\), proves \eqref{eq:clean_sur_boundary_center_mismatch}. \qed

\subsection{Proof of Theorem~\ref{thm:clean_surrogate_linear}}
\label{subsec:app_sur_descent}

As in the proof of Theorem~\ref{thm:clean_linear_tree}, we first establish a descent estimate and apply the surrogate block error bound. Combining the resulting recursion with the mismatch bound in Lemma~\ref{lem:clean_surrogate_tree_mismatch} then yields linear convergence.

\begin{lemma}[Descent]
\label{lem:app_sur_one_step}
Under Assumption~\ref{ass:clean_surrogate_nodewise}, if \(0\le\tau_r^\nu\le1/\omega_{\cP}\) for every \(r\in[p]\), then
\begin{equation}
G(Z^{\nu+1})
\le
G(Z^\nu)-\sum_{r=1}^p\tau_r^\nu\widetilde\Delta_r(Z^\nu)
+
\sum_{r=1}^p\tau_r^\nu\frac{\widetilde L_r^*}{2}
\|\hat Z_r^{\nu+1}-\bar Z_r^{\nu+1}\|^2.
\label{eq:app_sur_one_step}
\end{equation}
\end{lemma}

\begin{proof}
Since \(\tau_r^\nu\le1/\omega_{\cP}\), applying Jensen's inequality to each \(f_i^*\) and summing over the nodes yields
\[
G(Z^{\nu+1})-G(Z^\nu)
\le
\sum_{r=1}^p\tau_r^\nu
\left[G(Z^\nu-Z_r^\nu+\hat Z_r^{\nu+1})-G(Z^\nu)\right].
\]
For fixed \(r\), all outside coordinates are unchanged.  Majorization and touching give
\[
G(Z^\nu-Z_r^\nu+\hat Z_r^{\nu+1})-G(Z^\nu)
\le
\widetilde G(Z^\nu-Z_r^\nu+\hat Z_r^{\nu+1};Z^\nu)-
\widetilde G(Z^\nu;Z^\nu).
\]
Since \(Z^\nu-Z_r^\nu+\bar Z_r^{\nu+1}\) minimizes \(\widetilde G(\cdot;Z^\nu)\) over \(Z^\nu+\cU_r\) and \(\hat Z_r^{\nu+1}-\bar Z_r^{\nu+1}\in\cU_r\), first-order optimality gives
\[
\left\langle
\nabla_U\widetilde G(Z^\nu-Z_r^\nu+\bar Z_r^{\nu+1};Z^\nu)^\top,
\hat Z_r^{\nu+1}-\bar Z_r^{\nu+1}
\right\rangle=0.
\]
The \(\widetilde L_r^*\)-smoothness on this slice then yields
\[
\widetilde G(Z^\nu-Z_r^\nu+\hat Z_r^{\nu+1};Z^\nu)
\le
\widetilde G(Z^\nu-Z_r^\nu+\bar Z_r^{\nu+1};Z^\nu)
+
\frac{\widetilde L_r^*}{2}\|\hat Z_r^{\nu+1}-\bar Z_r^{\nu+1}\|^2.
\]
Combining these bounds with the definition of \(\widetilde\Delta_r(Z^\nu)\) proves~\eqref{eq:app_sur_one_step}.
\qed
\end{proof}

\begin{lemma}[Surrogate block error bound]
\label{lem:app_sur_main_recursion}
Under Assumptions~\ref{ass:clean_surrogate_nodewise} and~\ref{ass:clean_sur_beb}, assume that \(0<\tau_r^\nu\le1/\omega_{\cP}\) for every \(r\in[p]\), and let \(\underline\tau^\nu:=\min_{r\in[p]}\tau_r^\nu\). Then
\begin{equation}
\begin{aligned}
G(Z^{\nu+1})-G^\star
\le\ &
\left(1-\frac{\underline\tau^\nu}{2\widetilde\kappa_{\cP}}\right)
\bigl(G(Z^\nu)-G^\star\bigr)
-
\sum_{r=1}^p\frac{\widetilde\mu_r^*}{8\tau_r^\nu}\|Z_r^{\nu+1}-Z_r^\nu\|^2  \\
&\quad+
\sum_{r=1}^p\tau_r^\nu\frac{2\widetilde L_r^*+\widetilde\mu_r^*}{4}
\|\hat Z_r^{\nu+1}-\bar Z_r^{\nu+1}\|^2.
\end{aligned}
\label{eq:app_sur_main_recursion}
\end{equation}
\end{lemma}

\begin{proof}
Split \(\sum_{r=1}^p\tau_r^\nu\widetilde\Delta_r(Z^\nu)\) in~\eqref{eq:app_sur_one_step} into two halves.  The surrogate block error bound gives
\[
\frac12\sum_{r=1}^p\tau_r^\nu\widetilde\Delta_r(Z^\nu)
\ge
\frac{\underline\tau^\nu}{2\widetilde\kappa_{\cP}}
\bigl(G(Z^\nu)-G^\star\bigr).
\]
For the second half, touching and strong convexity of the nondelayed surrogate slice give
\[
\widetilde\Delta_r(Z^\nu)
\ge
\frac{\widetilde\mu_r^*}{2}\|\bar Z_r^{\nu+1}-Z_r^\nu\|^2.
\]
Using
\[
\|\bar Z_r^{\nu+1}-Z_r^\nu\|^2
\ge
\frac12\|\hat Z_r^{\nu+1}-Z_r^\nu\|^2-
\|\hat Z_r^{\nu+1}-\bar Z_r^{\nu+1}\|^2
\]
and \(Z_r^{\nu+1}-Z_r^\nu=\tau_r^\nu(\hat Z_r^{\nu+1}-Z_r^\nu)\), the second half of the progress satisfies
\[
\frac12\tau_r^\nu\widetilde\Delta_r(Z^\nu)
\ge
\frac{\widetilde\mu_r^*}{8\tau_r^\nu}\|Z_r^{\nu+1}-Z_r^\nu\|^2
-
\tau_r^\nu\frac{\widetilde\mu_r^*}{4}
\|\hat Z_r^{\nu+1}-\bar Z_r^{\nu+1}\|^2.
\]
Substituting these bounds into Lemma~\ref{lem:app_sur_one_step} gives~\eqref{eq:app_sur_main_recursion}.\qed
\end{proof}

\noindent\textit{Proof of Theorem~\ref{thm:clean_surrogate_linear}.} Set \(\nu_0:=\widetilde D+1=\max_{r\in[p]}\operatorname{diam}(\cT_r)\). Substituting~\eqref{eq:clean_sur_boundary_center_mismatch} into \eqref{eq:app_sur_main_recursion} and using~\eqref{eq:clean_sur_A} gives for all \(\nu\ge\nu_0\)
\[
\begin{aligned}
G(Z^{\nu+1})-G^\star
&\le
\left(1-\frac{\tau}{2\widetilde\kappa_{\cP}}\right)
\bigl(G(Z^\nu)-G^\star\bigr)
-
\frac{\underline{\widetilde\mu}^*}{8\tau}
\sum_{r=1}^p\|Z_r^{\nu+1}-Z_r^\nu\|^2
\\
&\quad
+
\tau\widetilde A
\sum_{\ell=\nu-\widetilde D}^{\nu-1}
\sum_{r=1}^p\|Z_r^{\ell+1}-Z_r^\ell\|^2.
\end{aligned}
\]
Define the Lyapunov function
\[
\widetilde{\cL}^{\nu}:=G(Z^\nu)-G^\star
+\tau\widetilde A\sum_{j=1}^{\widetilde D}
\left(\sum_{k=1}^{\widetilde D-j+1}\widetilde\rho^{-k}\right)
\sum_{r=1}^p\|Z_r^{\nu-j+1}-Z_r^{\nu-j}\|^2.
\]
Condition~\eqref{eq:clean_sur_stepsize} gives \(\frac{8\widetilde A\tau^2}{\underline{\widetilde\mu}^*}\sum_{j=1}^{\widetilde D}\widetilde\rho^{-j}\le1\). Then for \(\nu\ge\nu_0\),
\[
\begin{aligned}
\widetilde{\cL}^{\nu+1}
&\le\widetilde\rho\widetilde{\cL}^{\nu}
-\left(\frac{\underline{\widetilde\mu}^*}{8\tau}
-\tau\widetilde A\sum_{j=1}^{\widetilde D}\widetilde\rho^{-j}\right)
\sum_{r=1}^p\|Z_r^{\nu+1}-Z_r^\nu\|^2\le\widetilde\rho\widetilde{\cL}^{\nu}.
\end{aligned}
\]
Thus \(G(Z^\nu)-G^\star\le\widetilde{\cL}^{\nu_0}\widetilde\rho^{\nu-\nu_0}\), and
\[
\|X^\nu-X^\star\|^2
\le\frac{1}{\mu^2}\|Z^\nu-Z^\star\|^2
\le\frac{2L\widetilde{\cL}^{\nu_0}}{\mu^2}
\widetilde\rho^{\nu-\nu_0},
\qquad \nu\ge\nu_0.
\]
Absorbing constants gives~\eqref{eq:clean_sur_primal_rate}.
 \qed

\bibliography{master-bibliography}
\bibliographystyle{plain}

\end{document}